\documentclass[10pt,reqno]{amsart}
\usepackage[margin=0.4in]{geometry}
\usepackage{amsthm,amsmath,amsfonts,amssymb,euscript,hyperref,graphics,color,slashed, mathrsfs, graphicx,makecell,rotating,multirow}
\newtheorem{theorem}{Theorem}[section]
\newtheorem{lemma}[theorem]{Lemma}
\newtheorem{proposition}[theorem]{Proposition}
\newtheorem{corollary}[theorem]{Corollary}
\newtheorem{definition}[theorem]{Definition}
\newtheorem{remark}[theorem]{Remark}

\numberwithin{equation}{section}

\def\alphaone{{\alpha}^{({\bf 1})}}

\def\rhoone{{\rho}^{({\bf 1})}}

\def\les{\lesssim}
\def\sigmaone{{\sigma}^{({\bf 1})}}

\def\alphabone{{\underline{\alpha}}^{({\bf 1})}}

\def\A{\mathbf{A}}
\def\alphab{\underline{\alpha}}
\def\B{\mathcal{B}}
\def\D{\mathcal{D}}
\def\Dt{\widehat{D}}
\def\H{\mathcal{H}}
\def\S{\mathcal{S}}
\def\F{\mathcal{F}}
\def\Fb{{\underline{\mathcal{F}}}}
\def\E{\mathcal{E}}
\def\Hb{\underline{\mathcal{H}}}
\def\Lb{{\underline{L}}}

\def\cL{\mathcal{L}}
\def\Div{\mathbf{div\,}}

\def\Divergence{\mathbf{Div\,}}
\def\snabla{\slashed{\nabla}}
\def\Fc{{\mathring{F}}}
\def\alphac{\mathring{\alpha}}
\def\alphabc{\underline{\mathring{\alpha}}}
\def\rhoc{\mathring{\rho}}
\def\sigmac{\mathring{\sigma}}
\def\alphat{\widetilde{\alpha}}
\def\alphabt{\underline{\widetilde{\alpha}}}
\def\rhot{\widetilde{\rho}}
\def\sigmat{\widetilde{\sigma}}
\def\k{\mathbf{k}}

\def\I{\mathbf{I}}
\def\II{\mathbf{I\!I}}
\def\R{\mathbf{R}}
\def\SS{\mathbf{S}}
\def\T{\mathbf{T}}
\def\LZk{\mathcal{L}_Z^\mathbf{k}}

\def\DZtk{\widehat{D}_Z^\mathbf{k}}
\def\alphak{\alpha^{(\mathbf{k})}}
\def\alphabk{\underline{\alpha}^{(\mathbf{k})}}
\def\rhok{\rho^{(\mathbf{k})}}
\def\sigmak{\sigma^{(\mathbf{k})}}
\def\Jk{{J^{(\mathbf{k})}}}
\def\phik{{\phi^{(\mathbf{k})}}}
\def\sJk{\slashed{J}^{(\mathbf{k})}}
\def\LJk{J_L^{(\mathbf{k})}}
\def\LbJk{J_{\underline{L}}^{(\mathbf{k})}}

\def\LbJl{J_{\underline{L}}^{(\mathbf{l})}}
\def\Nk{{N^{(\mathbf{k})}}}
\def\epsilonc{\mathring{\varepsilon}}
\begin{document}
\title[MKG]{On global dynamics of the Maxwell-Klein-Gordon equations}



\author[Shiwu Yang]{Shiwu Yang}
\address{Beijing International Center for Mathematical Research, Peking University\\ Beijing, China}
\email{shiwuyang@math.pku.edu.cn}

\author[Pin Yu]{Pin Yu}
\address{Department of Mathematics and Yau Mathematical Sciences Center, Tsinghua University\\ Beijing, China}
\email{yupin@mail.tsinghua.edu.cn}

\begin{abstract}
On the three dimensional Euclidean space, for data with finite energy, it is well-known that the Maxwell-Klein-Gordon equations admit global solutions. However, the asymptotic behaviours of the solutions for the data with non-vanishing charge and arbitrary large size are unknown. It is conjectured that the solutions disperse as linear waves and enjoy the so-called peeling properties for pointwise estimates. We provide a gauge independent proof of the conjecture.
\end{abstract}
\maketitle




The Maxwell-Klein-Gordon (MKG) equations are a nonlinear system modeling the motion of a charged particle moving in an electric-magnetic field. From the physics point of view, the electric-magnetic field fades away and the final state of the particle must be static. The mathematical interpretation of this conjecture is that the MKG equations admit global \textsl{decaying} solutions. The global solvability of this system has been well understood since the work of Eardley-Moncrief in 1980's.
But very little is known on the aysmptotic behaviours of the solutions for the general large initial data. The aim of our current study is to show that the solutions enjoy the peeling estimates for the massless case.

\bigskip

We begin with a quick introduction to the MKG equations. Let $A=A_\mu d x^\mu$ be a $\mathbb{R}$-valued connection 1-form for a given complex line bundle $\mathbf{L}$ over the Minkowski spacetime $\mathbb{R}^{3+1}$. The curvature 2-form $F=\big(F_{\mu\nu}\big)$ associated to $A$ is simply $dA$, i.e., $F_{\mu\nu} = \partial_\mu A_\nu-\partial_\nu A_\mu$. In particular, $F$ is a closed 2-form or equivalently $\nabla_{[\alpha}F_{\beta\gamma]}=0$. The pair of square brackets denote the anti-symmetrization of the three indices $\alpha, \beta$ and $\gamma$. We will use Einstein summation convention throughout the paper. The above equation is also called the Bianchi equation of $F$. It is also equivalent to
\begin{equation*}
\nabla^\mu \,^*F_{\mu\nu}=0,
\end{equation*}
where $\,^*F_{\mu\nu}$ is the Hodge dual of $F_{\mu\nu}$. We recall that the Hodge dual $\,^*G_{\alpha\beta}$ of a 2-form $G_{\alpha\beta}$ is $\frac{1}{2}\mathscr{E}_{\alpha\beta\gamma\delta}G^{\gamma\delta}$ with $\mathscr{E}_{\alpha\beta\gamma\delta}$ is the volume form of the Minkowski metric $m_{\alpha\beta}$.

A section of the bundle $\mathbf{L}$ can be represented by a $\mathbb{C}$-valued function $\phi$.  The covariant derivative of $\phi$ with respect to $A$ is
\begin{equation*}
D_\mu \phi =  \partial_\mu \phi+ \sqrt{-1}A_\mu \cdot \phi.
\end{equation*} The curvature form measures the non-commutativity of the covariant derivatives
\begin{equation*}
[D_\mu, D_\nu]\phi = \sqrt{-1}F_{\mu\nu}\cdot\phi.
\end{equation*}
The massless MKG equations is a system of equations for a connection $A$ on $\mathbf{L}$ and a section $\phi$ of $\mathbf{L}$:
\begin{equation}\label{MKG}
\left\{
\begin{aligned}
\nabla^\mu F_{\mu\nu} &=-J_\nu, \\
\Box_A \phi &=0,\\
\end{aligned}
\right.
\end{equation}
where $J_\mu = \Im(\phi\cdot \overline{D_\mu\phi})$ is called the \emph{current} 
 and $\Box_A =D^\mu D_\mu$. It can be derived as the Euler-Lagrange equations for the action
\begin{equation*}
\mathcal{L}(A,\phi)=\frac{1}{2}\int_{\mathbb{R}^{3+1}}D^\mu \phi \cdot \overline{D_\mu \phi} +\frac{1}{4}\int_{\mathbb{R}^{3+1}}F^{\mu\nu}F_{\mu\nu}.
\end{equation*}
We use the volume form of the Minkowski metric in the action. The system is a $\mathbf{U}(1)$-gauge theory, namely, if $(A,\phi)$ is a solution of \eqref{MKG}, then $(A-d\chi, e^{i\chi}\phi)$ is also a solution for any smooth function $\chi$.

The total charge of the system is given by
\begin{equation*}
q_0 =\frac{1}{4\pi}\int_{\mathbb{R}^3}\mathbf{div}E dx= \frac{1}{4\pi}\int_{\mathbb{R}^3}\Im(\overline{D_t\phi}\cdot \phi)=\frac{1}{4\pi}\lim\limits_{r\rightarrow\infty}\int_{|\omega|=1}r^{2}E_i \frac{x_i}{r} d\omega,
\end{equation*}
where $E_i = F_{0i}$ is the electric field. It is easy to check that the total charge is conserved. This in particular implies that the electric field $E$ has the nontrivial tail $q_0 r^{-3}x$ at any fixed time.

\bigskip


The pioneering works \cite{Moncrief1} and \cite{Moncrief2} of Eardley-Moncrief established the celebrated global existence result to the general Yang-Mills-Higgs equations with sufficiently smooth initial data. Around ten years later, by introducing the weighted Sobolev spaces, Klainerman-Machedon systematically studied the bilinear estimates of the null forms. As a consequence, they derived the notable global existence result for data merely bounded in the energy space. The idea of proving bilinear estimates of null form introduced in \cite{MKGkl} leads to an revolution on the global well-posedness of PDEs of classical field theory, such as MKG equations, Yang-Mills equations, wave maps, etc., aiming at studying low regularity initial data in order to construct global solutions, see \cite{YMkl} and references therein. For a more recent and comprehensive summary of the progresses along this line, we refer to the work of Oh-Tataru \cite{OhMKG4}. The common feature of all these works is to construct a local solution with rough data so that the global well-posedness follows from conserved energy quantities. However regarding the global dynamics of the solutions, very little can be obtained through this approach.

\medskip

The long time dynamics of solutions of MKG equations have only been well understood for sufficiently small initial data or data which are essentially compactly supported. The robust vector field method introduced by Klainerman in \cite{klinvar} has been successfully applied to derive the decay estimates for linear fields in \cite{asymLkl} or nonlinear spin fields in \cite{Shu} with small initial data. If the data are compactly supported, one can also use the conformal compactification method (see e.g. \cite{fieldschrist}) but this approach requires strong decay of the initial data, which in particular forces the total charge to be vanishing. To tackle the general case with nonzero charge, Shu in \cite{shu2} proposed a frame work but without details. A complete proof towards this direction was contributed by Lindblad-Sterbenz in \cite{LindbladMKG}, also see a recent work \cite{Lydia:MKG:small} of Bieri-Miao-Shahshahani. However all these works are restricted to the small data regime or can be viewed as global stability problems of trivial solutions.

\medskip

As for the large data problem, by using the conformal compactification method together with Eardley-Moncrief's results, Petrescu in \cite{Deivy:decayMKG:trivialout} obtained the asymptotic decay properties of solutions to MKG equations with essentially compactly supported data, i.e., the scalar field has compact support and the Maxwell field is electrostatic outside the support. A similar result for Yang-Mills equation on $\mathbb{R}^{3+1}$ was obtained by Georgiev-Schirmer in \cite{Pedro:YM:sph} but with spherically symmetric data bounded in the conformal energy space (in particular charges must be vanishing!). For general initial data, the global asymptotic behaviour is only partially known. A Morawetz type of integrated local energy estimate was obtained by Psarelli in \cite{Maria:timedecay:mMKG} for solutions of massive Maxwell-Klein-Gordon equations with data bounded in the energy space. For massless MKG equations, Yang in \cite{yangILEMKG} derived the stronger inverse polynomial decay of the energy flux through outgoing null cones with data bounded in weighted energy space. Both results allow the existence of nonzero charges. However the decay estimates in \cite{Maria:timedecay:mMKG} are too weak to see the effect of the charges while the latter work in addition affirmatively answered a conjecture of Shu in \cite{Shu} that the nonzero charge can only affect the asymptotic behaviours of the solutions outside a forward light cone. Another consequence of the method used in \cite{yangILEMKG} allowed the author to improve the small data results to data merely small on the scalar field while the Maxwell field can be arbitrarily large, see details in \cite{yangMKG}. This result can be interpreted as the global nonlinear stability of linear Maxwell fields.

 \medskip

The aim of the present paper is to prove the final state conjecture of charged scalar fields: the solution should eventually decay as long as it decays suitably initially. More precisely, we show that the global solutions of the MKG equations with general large initial data enjoy the peeling decay properties, which in particular improves the existing results with restricted data, either with compact support or small size. In addition, our general data indicate that the conserved charge is arbitrary large, which will cause a long range effect to the asymptotic behaviour of the solutions. In the following, we will propose a sequence of new ideas to handle this long range effect and prove the conjecture for rapidly decaying data.

\section{Introduction to the main result}

Throughout the paper, we use the following conventions:
\begin{itemize}
\item The Greek letters $\alpha,\beta,\cdots$ denote indices from $0$ to $3$. The capital Latin letters $A,B,\cdots$ denote indices from $1$ to $2$. The little Latin letters $i,j,k,\cdots$ denote indices from $1$ to $3$.
\item $(\phi,F)$ is a \emph{given} finite energy smooth solution of the MKG equations. It exists globally and remains smooth according to the classical result of Klainerman-Machedon \cite{MKGkl}.
\item The letter $f$ denotes an \emph{arbitrary} section of the bundle $\mathbf{L}$ (it may not be $\phi$). The letter $G$ denotes an \emph{arbitrary} 2-form $G_{\mu\nu}$ (it may not be $F$).
\item We define $\psi=r\phi$.
\end{itemize}
We use two coordinate systems on the Minkowski spacetime $\mathbb{R}^{3+1}$: the Cartesian coordinates $(x^0=t,x^1,x^2,x^3)$ and the polar coordinates $(t,r,\vartheta)$. The optical functions $u$ and $v$ are defined as
\begin{equation*}
u=\frac{1}{2}(t-r), \ \ v=\frac{1}{2}(t+r), \quad u_+=1+|u|,\quad v_+=1+|v|.
\end{equation*}
A \emph{null frame} is defined by $(e_1,e_2,e_3=\Lb,e_4=L)$, where $L=\partial_t+\partial_r$, $\Lb=\partial_t-\partial_r$ and $e_1, e_2$ is an orthonormal complement of $L$ and $\Lb$.  

The level sets of $u$ and $v$ define (locally) null foliations of the Minkowski spacetime. Given $r_2 >r_1 >0$, we define the outgoing (or incoming) null hypersurfaces $\H_{r_1}^{r_2}$ (or $\Hb_{r_2}^{r_1}$) as
\begin{equation*}
\H_{r_1}^{r_2} :=\big\{(t,r,\vartheta) \,\big|\, t\geq 0, u=-\frac{1}{2}r_1, r_1\leq r\leq r_2\big\}\ \ \ \text{or}\ \ \
\Hb_{r_2}^{r_1} :=\big\{(t,r,\vartheta) \,\big|\, t\geq 0, v= \frac{1}{2}r_2, r_1\leq r\leq r_2\big\}
\end{equation*}
respectively.
On the initial time slice $\big\{ t=0 \big\}$ where the Cauchy datum is given, we define
\begin{equation*}
\B_{r_1}^{r_2} :=\big\{(t,r,\vartheta) \,\big|\, t= 0, r_1\leq r\leq r_2\big\}.
\end{equation*}
In the limiting case where $r_2=\infty$, we write $\H_{r_1}=\H_{r_1}^{\infty}$, $\Hb_{r_1}=\Hb_{r_1}^{\infty}$ and $\B_{r_1}=\B_{r_1}^{\infty}$. Three hypersurfaces $\H_{r_1}^{r_2}$, $\Hb_{r_2}^{r_1}$ and $\B_{r_1}^{r_2}$ bound a spacetime region and it is denoted by $\D_{r_1}^{r_2}$. In the following picture, the gray region is $\D_{r_1}^{r_2}$. The truncated light cones $\H_{r_1}^{r_2}$ and $\Hb_{r_2}^{r_1}$ are denoted by the dashed line segments. Their intersection is a 2-sphere of radius $\frac{r_1+r_2}{2}$ and it is the tip of $\D_{r_1}^{r_2}$ in the picture. We denote this sphere by  $\S_{r_1}^{r_2}$. The dashed-dotted line segment on the bottom is $\B_{r_1}^{r_2}$.

\ \ \ \ \ \ \ \ \ \ \ \ \ \ \ \ \ \ \ \ \ \ \ \ \includegraphics[width=3.3in]{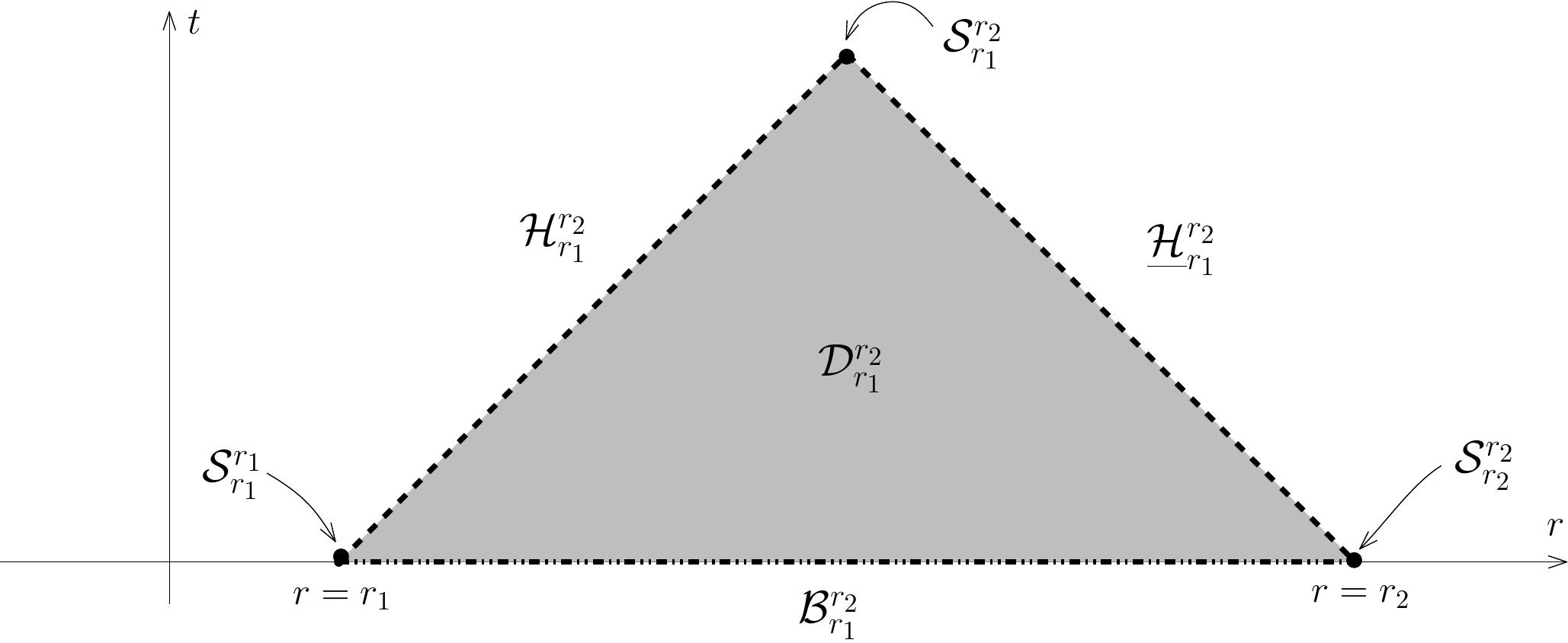}

In the null frame, we have $\nabla_L L= \nabla_L \Lb=\nabla_{\Lb} L=\nabla_\Lb \Lb=0$. Moreover, we have
\begin{equation*}
\begin{split}
\nabla_{e_A} L =\frac{1}{r} e_A, \  \nabla_{e_A} \Lb =-\frac{1}{r} e_A, \ \ \nabla_{e_A} e_B = \slashed{\nabla}_{e_A}e_B + \frac{1}{2r}\slashed{g}_{AB}(\Lb-L),
\end{split}
\end{equation*}
where $\slashed{\nabla}_{e_A}e_B$ is the projection of ${\nabla}_{e_A} e_B$ to a 2-sphere $\S_{r_1}^{r_2}$ (or to the span of $e_1$ and $e_2$) and $\slashed{g}_{AB}$ is the restriction of the Minkowski metric to  $\S_{r_1}^{r_2}$.

We can decompose $G_{\mu\nu}$ with respect to the null frame:
\begin{equation*}
\alpha(G)_A :=G(L,e_A), \ \alphab(G)_A:=G(\Lb,e_A), \ \rho(G) :=\frac{1}{2}G(\Lb,L), \ \sigma(G)_{AB} := G_{AB}.
\end{equation*}
For the special case $G_{\mu\nu}=F_{\mu\nu}$, we write
\begin{equation*}
\alpha_A =F(L,e_A), \ \alphab_A=F(\Lb,e_A), \ \rho :=\frac{1}{2}F(\Lb,L), \ \sigma_{AB} := F_{AB}.
\end{equation*}
Since $\sigma_{AB}$ is a 2-form on $\S_{r_1}^{r_2}$, there exists a function $\sigma$ so that $
\sigma_{AB} = \sigma \slashed{\mathscr{E}}_{AB}$ where $\slashed{\mathscr{E}}_{AB}$ is the volume form on $\S_{r_1}^{r_2}$. For the Hodge dual $\,^*F$ of $F$, if we denote $\,^*\alpha_A = -\slashed{\mathscr{E}}_{A}{}^{B}\alpha_B$ (the Hodge dual of $\alpha$ on $\S_{r_1}^{r_2}$), we have
\begin{equation*}
\alpha_A(\,^*F) 
=\,^*\alpha_A, \ \alphab_A(\,^*F)
=-\,^*\alphab_A, \ \rho(\,^*F) =\sigma, \ \sigma(\,^*F)_{AB}= -\rho\slashed{\mathscr{E}}_{AB}.
\end{equation*}

\subsection{The main theorem}
We consider Cauchy problem to \eqref{MKG} with initial data given by
\begin{equation*}
\phi_0(x) = \phi(0,x), \ \phi_1(x)=\partial_t\phi(0,x), \ E^{\text{(ini)}}_i(x)=E_i(0,x), \ B^{\text{(ini)}}_i(x)=B_{i}(0,x).
\end{equation*}
The initial data set $(\phi_0, \phi_1, \ E^{\text{(ini)}}, \ B^{\text{(ini)}} )$
is said to be \textsl{admissible} if it satisfies the compatibility condition
\begin{equation}
\label{eq:comp:cond}
\mathbf{div}(E^{\text{(ini)}})=\Im(\phi_0\cdot \overline{\phi_1}),\quad \mathbf{div} (B^{\text{(ini)}})=0,
\end{equation}
To impose precise assumptions on the initial data, split the electric field $E^{\text{(ini)}}$ into the divergence free part $E^{df}$ and the curl free part $E^{cf}$, that is,
\[
\mathbf{div}(E^{df})=0,\quad \mathbf{curl} (E^{cf})=0,\quad E^{\text{(ini)}}=E^{df}+E^{cf}.
\]
From the above constraint equation, $E^{cf}$ is uniquely determined by $\Im(\phi_0\cdot \overline{\phi_1})$. In particular we can freely assign $(\phi_0, \phi_0, E^{df}, B^{\text{(ini)}})$ as long as $E^{cf}$, $B^{\text{(ini)}}$ are divergence free on the initial hypersurface $\{t=0\}$. We require this part of data decay rapidly and belong to certain weighted Sobolev space. However, since $E^{cf}$ satisfies an elliptic equation on $\mathbb{R}^3$, it has a nontrivial tail $\frac{q_0 x}{r^3}$ even with $(\phi_0, \phi_1)$ compactly supported. To describe the asymptotic behaviour of the solutions, we need to precisely capture the asymptotic behaviour of the solution contributed by the charge. By formally expanding the Green's function for Laplacian:
\begin{align*}
|x-y|^{-1}=|x|^{-1}+|x|^{-3}x\cdot y+\sum_{i,j=1}^3\frac{1}{2} |x|^{-3}(3|x|^{-2}x_ix_j-\delta_{ij})y_iy_j+o(|y|^2),
\end{align*}
we can define a potential function $V(x)$ as
\[
V(x)=|x|^{-1}\frac{1}{4\pi}\int_{\mathbb{R}^3}(1+|x|^{-2} x\cdot y+\frac{1}{2} |x|^{-2}(3|x|^{-2}(x\cdot y)^2-|y|^2))\Im(\phi_0\cdot \bar \phi_1)dy,\quad |x|>0.
\]
The potential is well defined if the initial data $(\phi_0, \phi_1)$ of the scalar field  decay rapidly. With the potential $V(x)$, we can define the general charge 2-form $F[q_0]$ with components
\[
F[q_0]_{0i}=E_i[q_0]=\partial_{i}V(x),\quad F[q_0]_{ij}=0.
\]
It is straightforward to check that $F[q_0]$ satisfies the linear Maxwell equation on the region away from the axis $\{x=0\}$. Moreover, there is a constant $C$, depending only on $\phi_0$ and $\phi_1$, so that
\begin{equation}
\label{eq:bd4Fq}
|\rho(F[q_0])|\leq C r^{-2},\quad |\alphab(F[q_0])|=|\alpha(F[q_0])|\leq C r^{-3},\quad |\sigma(F[q_0])|=0.
\end{equation}
We remark that most commonly one uses $F[q_0]=\frac{q_0}{r^2}dt\wedge dr$ to denote the charge part near special infinity and it is a special case of the above construction.

Let $\varepsilon_0$ be a small positive constant (say $10^{-2}$ ). We assume that the initial data is bounded in the following gauge invariant weighted Sobolev norm
\begin{equation}\label{initial data 1}
\begin{split}
 C_0:=\sum\limits_{k\leq 2}\int_{\mathbb{R}^3} &\Big[(1+r)^{2k+6+8\varepsilon_0}\big(|D^{k+1} \phi_0|^2+|D^k\phi_1|^2 + |\nabla^k \big(E^{\text{(ini)}}-E[q_0]\mathbf{1}_{|x|\geq 1}\big)|^2\\
 &+ |\nabla^k B^{\text{(ini)}}|^2\big)+r^{4+8\varepsilon_0}|\phi_0|^2\Big] dx.
\end{split}
\end{equation}
The main theorem of the paper is as follows:
\begin{theorem}[\bf Main result] Consider the Cauchy problem to the massless MKG equation \eqref{MKG} with admissible initial data $(\phi_0,\phi_1,  E^{\text{(ini)}}_i,  B^{\text{(ini)}}_i)$ bounded in the above weighted norm \eqref{initial data 1}. Then there is a global in time solution $(\phi, F)$ satisfying the following pointwise peeling estimates
\begin{equation}\label{peeling estimates}
\begin{split}
|\phi| \leq  C u_+^{-1}v_+^{-1},\quad |D_L(r\phi)| \leq  C u_+^{-1} v_+^{-2},\quad |\alpha(\mathring{F})| \leq C u_+^{-1} v_+^{-3} ,\\
|\rho(\mathring{F})|+|\sigma(\mathring{F})| +|\slashed{D} \phi| \leq C u_+^{-2}v_+^{-2},\ \ |\alphab(\mathring{F})|+|D_\Lb \phi|\leq C u_+^{-3}v_+^{-1}.
\end{split}
\end{equation}
for some constant $C$ depending only on $C_0$, where $\mathring{F}=F-F[q_0]\mathbf{1}_{\{1+t\leq |x|\}}$ with $\mathbf{1}_{\{1+t\leq |x|\}}$ the characteristic function of the exterior region $\{(t, x)| t+1\leq |x|\}$.
\end{theorem}
We give several remarks.
\begin{remark}
There is no restriction on the size or on the support of the data. In particular, the charge $q_0$ can be large. Besides the above pointwise estimates, uniform energy estimates as well as weighted energy estimates can also be derived in the course of the proof.
\end{remark}
\begin{remark}
  The peeling estimates \eqref{peeling estimates} for the chargeless part of the solution together with the trivial bound \eqref{eq:bd4Fq} of the charge part describe the asymptotic behaviour of the full solution in the exterior region. Moreover the estimate implies that the nontrivial charge can only affect the asymptotic behaviour of the solution in the exterior region. This confirms the conjecture of Shu in \cite{Shu}.
\end{remark}
\begin{remark}
There is a heuristic explanation of the construction of the charge part $F[q_0]$ from the dipole expansion perspective: if we expand the Maxwell field $F$ in a Taylor series near spatial infinity $r=\infty$ as
\begin{equation*}
F = F_2 +F_3+F_4+F_5+\cdots,
\end{equation*}
where $F_k = O(r^{-k})$. The formal expansion of the Green function gives the $F[q_0]=F_2+F_3+F_4$. In this work we require that the perturbation starts from $F_5$. Indeed, the main reason for doing this is to make $F-F[q_0]$ decay sufficiently fast initially so that the chargeless part is bounded in the weighted Sobolev norm defined in \eqref{initial data 1}.
\end{remark}
\begin{remark}
Regarding the dependence of the constant $C$ on the size of the initial data, our proof can easily imply that $C$ depends exponentially on the zeroth order weighted energy (without derivative of the initial data) but polynomially on the higher order weighted energies. Simply from the charge part, it seems that exponential dependence on the zeroth order energy can not be improved. However from the point of view of the bilinear estimates in \cite{MKGkl}, we conjecture that the dependence on higher order energy should be linear.
\end{remark}

\subsection{An outline of the proof: difficulties, ideas and novelties}
The proof uses almost all the existing techniques and results for Maxwell-Klein-Gordon equations: the vector field method, the conformal compactification, the conformal analogues in the vector field method and the low regularity existence results of Klainerman-Machedon. Besides these, we will also introduce new commutation vector fields, new null forms and study some new structure of the nonlinearities. In the rest of the section, we will first sketch the proof in three steps. Then, we will present the difficulties in each step and provide heuristic ideas to handle these difficulties. Finally, we will summarize some new aspects of the proof.

\subsubsection{The structure of the proof}
The proof consists of three steps:
\begin{itemize}
\item[Step 1]
We take a positive number $R_*$ and it determines the so-called exterior region $\mathcal{D}_{R_*}$ (grey part).

\ \ \ \ \ \ \ \ \ \ \ \ \ \ \ \ \ \ \ \ \ \includegraphics[width=3.8in]{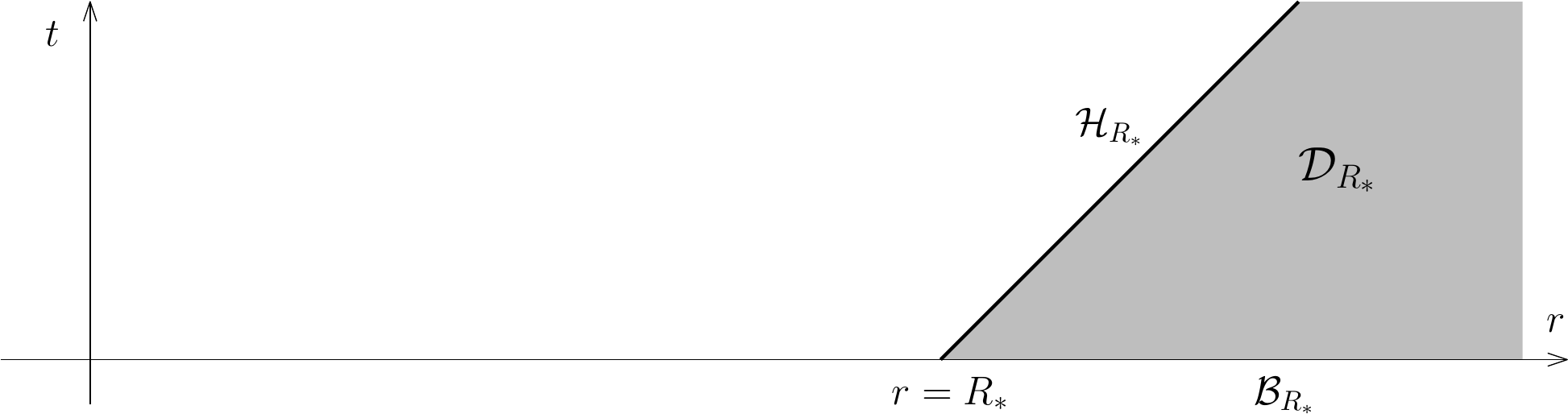}

\noindent For large $R_*$, by restricting data on the region where $r\geq R_*$, i.e., $\mathcal{B}_{R_*}$ (as the bottom of the grey region), we can assume the chargeless part of the restricted data is small. Since the grey region is the domain of dependence of $\mathcal{B}_{R_*}$, the solution in $\mathcal{D}_{R_*}$ is completely determined by the restricted data on $\mathcal{B}_{R_*}$. We therefore study the long time behaviour of solutions of MKG equations in the grey region $\mathcal{D}_{R_*}$ with data small in the chargeless part. We emphasize that this is not a small data problem as the charge part of the solution is large and is independent of the radius $R_*$.

\smallskip

\item[Step 2] This step connects the first step to the third. First of all, we will carefully choose a hyperboloid in $\mathcal{D}_{R_*}$ (on which we have precise control on the solution from the previous step). This hypersurface is denoted by $\Sigma_+$ in the next picture.

\ \ \ \ \ \ \ \ \ \ \ \ \ \ \ \ \ \ \ \ \ \ \ \ \ \ \ \ \ \ \ \ \ \ \ \ \ \ \includegraphics[width=3 in]{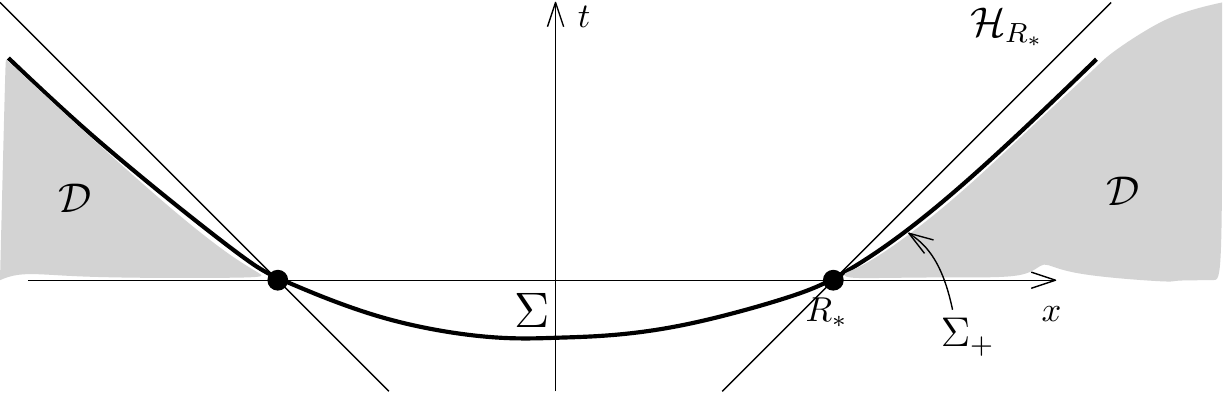}

\noindent The solution restricted on this hyperboloid can be viewed as initial datum for the solution in the interior region which is unknown so far. This step is devoted to showing that the solution obtained from the previous step is sufficient regular on $\Sigma_+$ so that we can conduct the next step.

\smallskip

\item[Step 3] In this last step, we will study the asymptotics of the solution in the causal future $\mathcal{J}^+(\Sigma)$ which is the grey region in the left figure (this is the white region in the previous picture).

\ \ \ \ \ \ \ \ \ \ \ \ \ \ \ \ \ \ \includegraphics[width=5in]{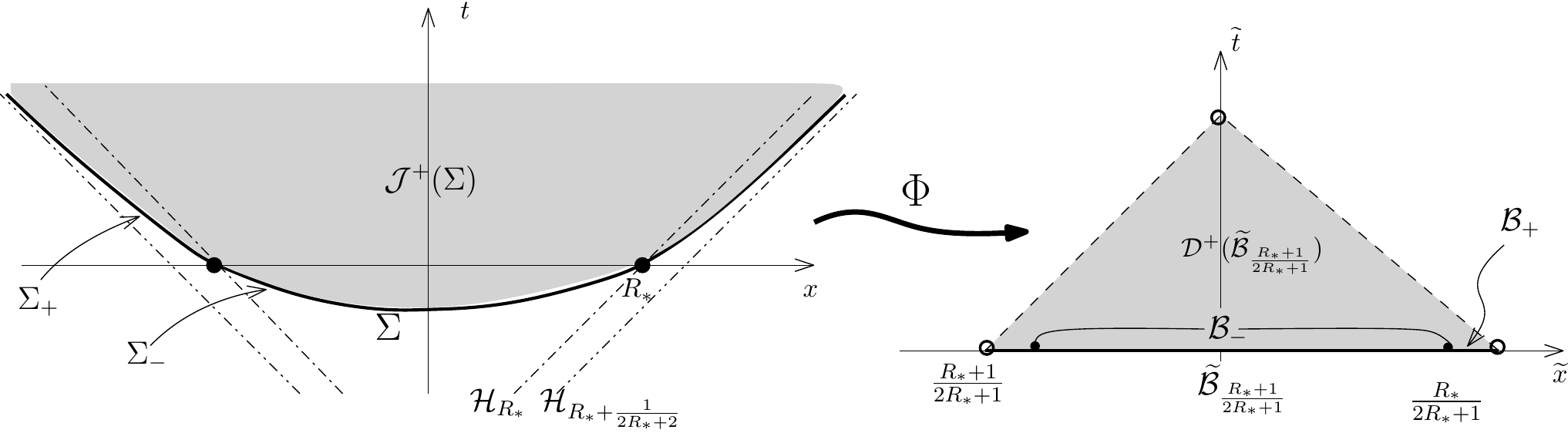}

The hypersurface $\Sigma$ consists of two parts: $\Sigma_-$ and $\Sigma_+$. Since $\Sigma_-$ is finite, the solution on $\Sigma_-$ can be well controlled by the data on the compact region $\{t=0, |x|\leq R_*\}$. This indeed follows from the classical result of Eardley-Moncrief or the result of Klainerman-Machedon. Together with Step 2, the restriction of the solution on $\Sigma$ will be well-understood in terms of the initial data.

Then we will perform a conformal transformation $\Psi$ to map $\mathcal{J}^+(\Sigma)$ to a backward finite light cone (the grey cone on the right hand side of the picture). The hypersurface $\Sigma$ will be mapped to the bottom of the cone. By multiplying conformal factors appropriately, the global dynamics of solutions to MKG equations defined on the left of the picture is then reduced to understanding the solution to MKG equations defined on the right of the picture. The estimates from Step 2 provide a bound of the $H^2$-norm of the solution on the bottom of the cone on the right hand side of the picture. This allows us to use the classical theory of Klainerman-Machedon to bound the solution on the cone up to two derivatives hence the $L^\infty$ norm of the solution. Finally, we undo the conformal transformation by rewriting the solution on the left hand side in terms of the solutions on the right hand side. The conformal factors then give the decay estimates of the solution in $\mathcal{J}^+(\Sigma)$. Together with the decay estimates from Step 1, we can derive the peeling estimates in the main theorem.
\end{itemize}

\subsubsection{Difficulties in the proof}
We list several difficulties which did not appear in previous works on MKG equations. We would like to emphasize that the first difficulty (the largeness of charge) listed below is related to all the rest. The remaining difficulties arise in course of the resolution of the first one. We also want to point out that the most difficult part of the proof is Step 1.
\begin{enumerate}
\item The large nonzero charge.

\smallskip

Although the energy norm of the chargeless part of the data in Step 1 is small, the charge $q_0$ can be large. First of all, the traditional conformal compactification method used in \cite{ChristodoulouYangM} by Christodoulou-Bruhat  requires strong decay of the data which forces the charge to be vanishing. Secondly, the presence of nonzero charge may cause a logarithmic divergence in the energy estimates, see a more thorough discussion in the work \cite{LindbladMKG} of Lindblad-Sterbenz for the purely small data case. The error term caused by the charge can in fact be absorbed if the charge is sufficiently small. We overcome this large charge difficulty by using the method developed by Yang in \cite{yangILEMKG}.

 We would also like to compare this charge difficulty with the massive case of recent work \cite{yang:mMKG} of Klainerman-Wang-Yang, in which they studied the massive MKG equations with small initial data. Their method also applies to the case with arbitrary large charge. However due to the existence of mass which gives control of the scalar field itself, the effect of nonzero charge can be easily controlled (see more detailed discussion in the next subsection). The main difficulty there, however, lies in the inconsistent asymptotic behaviours of Maxwell fields and solutions of Klein-Gordon equations.

\smallskip

\item The sharp peeling estimates.

\smallskip

Since in Step 3 we have to compactify $\mathcal{J}^+(\Sigma)$, the estimates for the solution obtained from Step 1 on $\Sigma$ must be sufficiently regular so that the solution on its conformal compactification are bounded in the right Sobolev spaces. In particular it requires to obtain the sharp decay estimates such as $D_L(r\phi)=O(r^{-2})$ and $\alpha=O(r^{-3})$ along outgoing light cones. So far as we know, even for the small data regime (with small charge), these estimates are unknown. 

\smallskip

\item New commutators to prove the necessary sharp peeling estimates.

\smallskip

The idea to obtain the above sharp peeling estimates is straightforward: we need to put more $r$-weights in the usual energy estimates so that the $r$-weights will be converted into extra decay via Sobolev inequality. We will use the conformal Morawetz vector field $K$ as commutators.  This vector field is of order $2$ in terms of weights $r$  and is used traditionally only as multipliers. In such a way, the structure of the nonlinear terms after commutation becomes the primary concern and we will show that it has some new null structure.

\smallskip

\item The hidden null structure of the MKG equations related to commutators.

\smallskip

This is related to point (2) above. When one commutes vector fields with the MKG equations, although it may generate many error terms, one needs to at least make sure some of the fundamental structures remain unchanged. Very often, these structures are important in the analytic perspective and more precisely they should be phrased in such a way that they fit into the energy estimates. We will show that there is a new null structure of the nonlinear terms which is invariant after commuting correct vector fields. Also, there is another important type structure, which we will call it \textsl{reduced structure}, also remains unchanged after commutations.

\smallskip

\item The choice of conformal compactifications.

\smallskip

The presence of nonzero charge prevents us to use the usual Penrose type compactification for the entire spacetime (see \cite{ChristodoulouYangM}): the $\rho$-component of the Maxwell field behaves as $\frac{q_0}{r^2}$ which cannot be compactified near the spatial infinity. However this effect of charge does not propagate from the spatial infinity to the future null infinity so that we can indeed perform a conformal transformation inside a null cone to avoid spatial infinity.

\smallskip

\item Precise energy estimates on the hypersurface $\Sigma$ in Step 2.

\smallskip

Since $\Sigma$ is a hyperboloid in Minkowski spacetime, the energy estimates on $\Sigma$, especially those needed in the Klainerman-Machedon theory after the compactification, are not straightforward. Nevertheless, this part is less serious compared to all the previous ones and can be derived by using the classical energy estimates in a geometric way.
\end{enumerate}

\subsubsection{Key ideas and novelties of the proof}

In this subsection, we will list all the ideas and new features of the proof in order to deal with the difficulties mentioned in the previous subsection.

\begin{enumerate}
\item The reduced structure and converting spatial decay against the logarithmic growth.

\smallskip

We first explain the reduced structure of the nonlinearity. Let $F=dA$. We may think of $A$ as $\phi$. Thus, the Maxwell equations are reduced to the form
\begin{equation*}
\Box A = \phi \cdot D\phi.
\end{equation*}
While most commonly, a nonlinear wave equation with quadratic interaction looks like
\begin{equation*}
\Box \phi = \nabla \phi \cdot \nabla \phi.
\end{equation*}
The MKG equations is one derivative less in the nonlinearities. This is the reduced structure.

In terms of energy estimates, the reduced structure will be reflected in the following formula:
\begin{equation*}
\int_{\H_{r_1}^{\infty}} |D_L \phi|^2 \leq C_1 \epsilonc r_1^{-\gamma_0}+C_2\int_{r_1}^{\infty}\int_{\H_{s}^{\infty}}\frac{|q_0|}{r^2}|\phi||D_L \phi|.
\end{equation*}
The left hand side is a classical energy flux term through outgoing null cones $\{u=r_1\}$. The first term on the right hand side is coming from the data and the exponent $-\gamma_0$ reflects the decay of the data near spatial infinity. The second term on the right hand side contains a $\phi$ without any derivative acting on it. We remark that the $\frac{q_0}{r^2}$ factor is arising from the charge. Heuristically for waves, a $\frac{1}{r}$ factor can be regard as $D_L$-derivative so that we should think of the second term as $\frac{1}{r}|D_L\phi|^2$ thus we see that there is a logarithmic growth when we integrate. We remark here that for the massive case in \cite{yang:mMKG}, since solution of massive Klein-Gordon equation decays as quickly as its derivatives, i.e., one can regard $\phi$ as $D_L\phi$, the above error term can be easily absorbed by using Gronwall's inequality.

We use an idea introduced by Yang in \cite{yangILEMKG} to handle this logarithmic loss. The precise statement is summarized and proved in Lemma \ref{lemma key}. Morally speaking, to obtain the estimates for the energy flux, we can afford a loss in $r$ instead of in time:
\begin{equation*}
\int_{\H_{r_1}^{\infty}} |D_L \phi|^2 \leq C\epsilonc \cdot r_1^{-\gamma_0+\varepsilon_0}.
\end{equation*}
In other words, the decay rate near null infinity changes from $\gamma_0$ to $\gamma_0 -\varepsilon_0$.

\smallskip

\item The $\varepsilon_0$-reductive argument for higher order energy estimates.

\smallskip

The argument is designed to make a better use of the reduced structure of the nonlinearity when we do higher order energy estimates. Although the charge vanishes when taking derivatives, the above type error term arises from the connection field $A$ and has the same structure as described previously. We design an ansatz which allows higher order derivative to lose more decay. To be more precise,  we will lose $2(k+1)\varepsilon_0$-decay for the $k$-th order derivatives. We will use the following example to illustrate how the argument works. In the course of deriving energy estimates for the first order derivatives, schematically, the nonlinear terms look like $\int |\nabla \phi||\nabla^{2}\phi |$. On the other hand, $|\nabla\phi|$ is already controlled when one derives estimates for the solution itself without commuting vector fields with equations, thus the estimates on $\nabla \phi$ only lose $2\varepsilon_0$ decay. Compared to the $4\varepsilon_0$ loss in the first order derivative case, we indeed have a gain in decay for the nonlinear terms. This gain will play an essential role in closing the estimates.

\smallskip
\item Morawetz vector field as commutator and new commutation formulas.

\smallskip

Traditionally, the Morawetz vector field $K$ is only used as multipliers in the energy estimates. In this work, we will commute $K$ with the equation. The extra weights compared to the classical commutators such as rotations and scaling provide an extra decay factor for the solutions near null infinity. This extra decay factor is indispensable when we perform the conformal compactification. We would like to remark that, since $K$ is the image of $\partial_t$ under the inversion map, commuting $K$ with the equation can be regarded as the usual commutation of $\partial_t$ after the conformal transformation. Thus, this idea should be viewed as a vector field method version of conformal transformations. We also remark here that using such weighted vector fields with order $2$ has been used in the works \cite{Stefanos:avectorfield}, \cite{Stefanos:Latetime} of Angelopoulos-Aretakis-Gajic for deriving the sharp decay of linear waves on black hole spacetimes. However in those works, commuting the equation with the vector field $r^2 L$ is straightforward when writing the equation in terms of the radiation field $r\phi$ under a suitable null frame. This idea also applies to the scalar field equation of the MKG system in our situation but may not be so successful for the Maxwell equation since the Maxwell equation does not commute with the vector field $K$. Our new observation is as follows: for $Z\in \mathcal{Z}=\big\{T,\Omega_{12},\Omega_{23},\Omega_{31},S, K\big\}$, where $T$ is the time translation, $\Omega_{ij}$ are rotations and $S$ is scaling, for $\mathbf{Div}$ (the principle part of the Maxwell equations) and $\Box_A$, we have the following two formulas
\begin{equation}
[r^2 \Divergence, \mathcal{L}_Z]G = 0, \ \
[r^2\Box_A, D_Z+\frac{Z(r)}{r}]\phi = r^2 Q(\phi, F;Z)
\end{equation}
for any closed 2-form $G$ and complex scalar field $\phi$. In other words, although the equations do not commute with the vector fields $S$ or $K$, they commutes with the equations multiplied by $r^2$.
We emphasize that the formula holds for $K$ and $Q(\phi, F;Z)$ is quadratic in $\phi$ and $F=dA$. We also remark that to our knowledge these commutator formulas are new.

\smallskip

\item A new null form.

\smallskip

The quadratic form $Q(\phi, F;Z)$ is indeed a null form. Take $Z=S$ for example. It can be shown that
\begin{equation*}
\begin{split}
|Q(\phi,F;Z)| \lesssim &\big(\frac{r}{|u|}|\rho|+|\alphab|\big)|D_L(r\phi)|+\big(\frac{r}{|u|}|\alpha|+\frac{|u|}{r}|\alphab|+|\sigma|\big)|\slashed{D}(r\phi)| \\ &+\big(|\alpha|+\frac{|u|}{r}|\rho|\big)|D_\Lb(r\phi)|+\big(|\rho|+|\sigma|\big)|\phi|+{\text{cubic terms}}.
\end{split}
\end{equation*}
Similar estimates hold for other vector fields in $\mathcal{Z}$. We remark that rather than $\phi$ itself, the derivatives of $r\phi$ appear naturally in the above null structure estimate.

The most remarkable property of $Q(\phi,F;Z)$ is that it has an iterative structure. This is crucial when we commutate multiple derivatives with equations. More precisely, if we define $\Dt_{Z} = D_Z + \frac{Z(r)}{r}$, we can show that
\begin{equation*}
\big[\Dt_{Y},[\Dt_{X}, r^2\Box_A]\big]\phi = -r^2Q(\phi,F;[Y,X])-r^2Q(\phi,\mathcal{L}_Y F;X)+2r^2F_{Y\mu}F_{X}{}^{\mu}\phi.
\end{equation*}
The right hand side after commutating two derivatives can still be expressed in terms of $Q$  and it still satisfies the null structure. This is one of the keys in the proof.

We remark that to our knowledge this null structure is also new.

\smallskip

\item The algebraic structure of $J$.

\smallskip

We have seen that $r\phi$ appears naturally in the null form estimates. We would like to point out another perspective. We mentioned previously that $D_L(r\phi) = O(\frac{1}{r^2})$. We can also show that the best decay estimates for $D_L \phi$ is still $O(\frac{1}{r^2})$ instead of $O(\frac{1}{r^3})$. From this point of view, we may consider $r\phi$ to be "better" than $\phi$ itself. On the other hand, for the Maxwell equation, instead of commuting with the operator $\mathbf{Div}$, we commute with $r^2\mathbf{Div}$. It thus requires to analyze $r^2\cdot J$, where the charge density $J$ has components $J_\mu = \Im(\phi\cdot \overline{D_\mu\phi})$. The special algebraic form implies
\begin{equation*}
r^2 \cdot J_\mu[\phi] = \Im\big((r\phi)\cdot \overline{D_\mu(r\phi)}\big)=J_\mu[r\phi].
\end{equation*}
Therefore, we only have to deal with the "better" field $r\phi$ rather than $\phi$ itself. This special cancellation from the algebraic structure is crucial to obtain the sharp peeling estimates and to close the energy estimates.

\smallskip

\item The conformal compactification.

Since the trace of the energy momentum tensor for MKG equations are not zero, this field theory is not conformal. However, for special conformal transformations, it can still be conformally invariant, e.g., if $\Box \Lambda = 0$ where $\Lambda$ is the conformal factor. The inversion map restricted in the forward light cone is such a conformal map in $\mathbb{R}^{3+1}$ (not in other dimensions).

On the other hand, there is another important observation: although the presence of a nonzero charge does not allow compactification around the spatial infinity, this effect indeed does not appear on the null infinity. This was first pointed out by Shu in \cite{Shu}. The following computation for $F[q_0]$ justifies this observation: on a outgoing light cone $\H_u$ defined by $r-t=2u$, the conformal energy flux passing through this light cone (this is the basic energy quantity needed after the conformal transformation) is given by
\begin{equation*}
\E\big[F[q_0]\big] \approx \int_{\H_u} |u|^4|\rho|^2.
\end{equation*}
Since $|\rho|\approx \frac{q_0}{r^2}$ (as $F[q_0]$ has the leading term $q_0 dt\wedge dr$) and $u$ is a constant on $\H_u$, the above energy flux is finite. On the other hand, if one considers conformal energy on a constant time slice, the factor $u^4$ would be replace by $r^2 u^2$ (near spatial infinity) so that the contribution of the charge part of the field would be divergent. This is why we choose inversions as the conformal mappings.

\smallskip

\item $r^p$-weighted energy estimates.

\smallskip

We use the $r^p$-weighted energy estimates which was first introduced by Dafermos-Rodnianski in \cite{newapp} for the study of decay of linear waves on black hole spacetimes. The method has also been used in the first author's works on MKG equations, see \cite{yangILEMKG, yangMKG}, where $p<2$. The new point in the current work is that we have to deal with the end point case $p=2$ to get the sharp peeling estimates.
\end{enumerate}

\subsection{Further discussions}\label{section future}

It is instructive to make a comparison to the works \cite{Kl:peeling:EE}, \cite{kl:EE} of Klainerman-Nicol{\`o} to prove higher peeling estimates near Minkowski spacetime in an exterior region and the work \cite{Jonathan:stabilityofEE} of Luk-Oh for proving global nonlinear stability of dispersive solutions to Einstein equations. Indeed, for a given initial datum of the vacuum Einstein field equations, one can work in the region $r\geq R_*$ and can assume that the datum is small provided $R_*$ is sufficiently large. The mass $m$ for the Einstein equations plays a similar role as the charge $q_0$ for the Maxwell-Klein-Gordon equations: they all represent a slow decay tail representing a static solution at spatial infinity (which is the Schwarzschild solutions in the Einstein equations' case). The proof of Klainerman and Nicol{\`o} indeed does not use the smallness of the mass $m$ and this is similar to our case where we do not assume $q_0$ is small. For vacuum Einstein field equations, the mass $m$ comes in through the $\rho$-component of the curvature:
\[\rho=\frac{m}{r^3}+\rhoc,
\]
where $\rhoc$ decays as $\frac{1}{r^4}$. However, for MKG equations, the charge $q_0$ comes in through the $\rho$-component of the Maxwell field:
\[\rho=\frac{q_0}{r^2}+\rhoc.
\]
The $r^{-3}$ decay is sufficient to apply Gronwall's inequality in the Einstein equations' case while for MKG equations we have to find a new way to compensate the logarithmic loss as we mentioned before.

Alternatively, for this large mass issue for Einstein equations, Luk-Oh in \cite{Jonathan:stabilityofEE} choose a special gauge condition so that such mass problem does not appear. Since our approach in this paper is gauge invariant and the charge is inherited in the connection field $A$, the charge difficulty is essentially different from the mass problem for Einstein field equations.

For Einstein field equations coupled with other fields, say a scalar field, the coupling field may bring a slower decay tail. We believe that our method in the exterior region can also be applied to these cases to derive sharp peeling estimates.

\textbf{Acknowledgments.} The authors would like to thank Sergiu Klainerman for helpful suggestions on the manuscript. The second author is also deeply indebt to Pengyu Le for teaching him the conformal aspects for the Maxwell-Klein-Gordon theory. The first author is partially supported by the Recruitment Program of Global Experts in China and a start-up grant at Peking University. The second author is supported by NSFC-11522111 and China National Support Program for Young Top-Notch Talents.

\section{Preparations}
\subsection{The null decompositions of equations}
Recall from the main theorem that the chargeless part $\mathring{F}$ of the solution is defined as
\[
\mathring{F}=F-F[q_0]\mathbf{1}_{\{t+1\leq |x|\}}.
\]
It is straightforward to see that $\Fc$ satisfies the same equations as $F$:
\begin{equation}\label{Maxwell for Fc}
\nabla^\mu \Fc_{\mu\nu}=-J_\nu
\end{equation}
in the exterior region $\{t+1\leq |x|\}$.
In terms of the null components, we can rewrite this equation as
\begin{equation}\label{Maxwell null Fc}
\left\{\begin{aligned}
L(r^2\rhoc)+\slashed{\Div} (r^2\alphac) =r^2 J_L, \ \ \Lb(r^2\rhoc)&-\slashed{\Div} (r^2\alphabc) =-r^2 J_\Lb,\\
L(r^2\sigmac)+\slashed{\Div} (r^2\,^*\alphac) =0, \ \ \Lb(r^2\sigmac)&+\slashed{\Div} (r^2\,^*\alphabc) =0,\\
\snabla_\Lb (r \alphac)_A-\snabla_A(r\rhoc)-\,^*\snabla_A(r\sigmac) &=rJ_A,\\
\snabla_L (r \alphabc)_A+\snabla_A(r\rhoc)-\,^*\snabla_A(r\sigmac) &=rJ_A.
\end{aligned}
\right.
\end{equation}
Here for simplicity, $(\mathring{\alpha}, \mathring{\alphab}, \mathring{\rho}, \mathring{\sigma})$ are the null components associated to the 2-form $\mathring{F}$. For any complex scalar field ${f}$, the covariant wave operator $\Box_A$ can be expressed in null frames:
\begin{equation}\label{scalar box null}
\begin{split}
r\Box_A {f} = -D_\Lb D_L \big(r {f} \big) + \slashed{D}^2\big( r{f} \big) + i\rho\cdot\big(r{f}\big)= -D_L D_\Lb \big(r {f} \big) + \slashed{D}^2\big( r{f} \big) - i\rho\cdot\big(r{f}\big),
\end{split}
\end{equation}
where $\slashed{D}^2(r{f}) =\sum_{A,B=1}^2 m^{AB} D_{e_A} D_{e_B}(r{f})$.

\subsection{Commutator vector fields and null structures}\label{section commutator vector fields}
We shall use the following set of vector fields as commutators:
\begin{equation*}
\mathcal{Z}=\big\{T,\Omega_{12},\Omega_{23},\Omega_{31},S, K\big\},
\end{equation*}
where $K=v^2 L + u^2\Lb$ is the Morawetz vector field, $S = vL+u\Lb$ is the scaling vector field, $\Omega_{ij}=x_i\partial_j-x_j\partial_i$ are the rotation vector fields and $T=\partial_t$ is the time translation. For vector fields in $\mathcal{Z}$, we define their discrepancy index as
 \[
 \xi(T)=-1,\quad \xi(\Omega_{ij})=\xi(S)=0,\quad \xi(K)=1.
 \]
In the energy estimates, it involves the the deformation tensor of these vector fields: $$\,^{(Z)}\pi_{\mu\nu}=\frac{1}{2}\mathcal{L}_Z m_{\mu\nu}=\frac{1}{2}(\nabla_\mu Z_\nu + \nabla_\nu Z_\mu), $$
where $\cL_Z m$ is the Lie derivative of the Minkowski metric.
 By computation, we have
\begin{equation*}
\,^{(K)}\pi_{\mu\nu}=t\cdot m_{\mu\nu}, \ \  \,^{(S)}\pi_{\mu\nu}= m_{\mu\nu},\quad ^{(\Omega_{ij})}\pi_{\mu\nu}=0,\quad ^{(T)}\pi_{\mu\nu}=0,
\end{equation*}
where $m_{\mu\nu}$ is the flat metric of the Minkowski spacetime. We also remark that the set $\mathcal{Z}$ is closed under the Lie bracket: the only non-vanishing $[Z_1,Z_2]$'s for $Z_1,Z_2\in \mathcal{Z}$ are
\begin{equation*}
[T,S]=T, \ \ [T,K]=2S, \ \ [S,K]=K.
\end{equation*}

For $Z\in \mathcal{Z}$, we define the modified covariant derivative acting on complex scalar field associated to the 1-form $A$ as follows:
\begin{equation*}
\Dt_Z=D_Z + \frac{Z(r)}{r}.
\end{equation*}
This is the conjugate of $D_Z$ by the function $r$, i.e., $\Dt_Z f = r^{-1}D_Z(r f)$.
\begin{lemma}[Commutator formula]\label{commutator lemma}
For any closed 2-form $G$ and any complex scalar field $f$, we have
\begin{equation}\label{commutator formula 2}
[r^2 \Divergence, \mathcal{L}_Z]G = 0,
\end{equation}
\begin{equation}\label{commutator formula 3}
[r^2\Box_A, \Dt_Z]f = 2\sqrt{-1}r^2 F_{\mu \nu}Z^\nu D^\mu f+\sqrt{-1}r^2\nabla^\mu \big( Z^\nu F_{\mu \nu} \big)f
\end{equation}
for all $Z\in \mathcal{Z}$.
\end{lemma}
\begin{remark}
To our knowledge, this set of commutator formulas are new and it is one of the key ingredients to the proof.
\end{remark}
\begin{proof}
We first show the following formula
\begin{equation}\label{commutator formula 1}
[\Box_A, D_Z + \frac{Z(r)}{r}]{f} = \frac{2Z(r)}{r}\Box_A{f} + 2\sqrt{-1}F_{\mu \nu}Z^\nu D^\mu{f}+\sqrt{-1}\nabla^\mu \big( Z^\nu F_{\mu \nu} \big)f.
\end{equation}
By commuting derivatives, we have
\begin{equation*}
[\Box_A, D_Z]{f} = \Box Z_\mu D^\mu {f} + 2\,^{(Z)}\pi_{\mu\nu}D^\mu D^\nu{f} +  2\sqrt{-1}F_{\mu \nu}Z^\nu D^\mu{f}+\sqrt{-1}\nabla^\mu \big( Z^\nu F_{\mu \nu} \big){f}.
\end{equation*}
For any function $f_1$, we have
\begin{equation*}
[\Box_A, f_1]{f} = \Box f_1\cdot {f}+2\nabla^\mu f_1 D_\mu{f},
\end{equation*}
where $f_1$ will be $\frac{Z(r)}{r}$.

For $Z\in \mathcal{Z}$, if $Z \neq K$ or $S$, we have $\,^{(Z)}\pi_{\mu\nu}=0$ and $f=0$, therefore, \eqref{commutator formula 1} holds.

For $K$, we have $f_1=t$, $[\Box_A, f_1]{f} = 2\nabla^\mu f_1 D_\mu{f}$ and $\Box K =-T$. Hence,
  \begin{equation*}
[\Box_A, D_K+\frac{K(r)}{r}]{f} = -T_\mu D^\mu {f} + 2t\Box_A{f} +  2\sqrt{-1}F_{\mu \nu}Z^\nu D^\mu{f}+\sqrt{-1}\nabla^\mu \big( Z^\nu F_{\mu \nu} \big){f}+2\nabla^\mu t D_\mu{f}.
\end{equation*}
The first term and the last term on the right hand side cancel. This proves the case for $Z=K$.

For $S$, we have $f_1=1$ and the proof follows exactly in the same manner. Thus formula \eqref{commutator formula 3} holds.

\medskip

We turn to the proof of \eqref{commutator formula 2}. By commuting the derivatives, we have
\begin{equation*}
[\Divergence, \mathcal{L}_Z]G_\nu =\Box Z^\mu \,G_{\mu\nu}+\nabla_\nu\nabla^\mu Z^{\delta}\,G_{\mu\delta} + 2\,^{(Z)}\pi^{\mu\delta}\nabla_{\delta}G_{\mu\nu}.
\end{equation*}
If $Z\in \mathcal{Z}$ but $Z\neq K$ or $S$, then $[r^2,\mathcal{L}_Z]=0$. The above formula shows that $[\Divergence, \mathcal{L}_Z]=0$. Hence, \eqref{commutator formula 2} holds.

For $K$, the above formula implies
\begin{align*}
[\Divergence, \mathcal{L}_K]G_\nu &= -2T^\mu \,G_{\mu\nu}+\nabla_\nu\nabla^\mu K^{\delta}\,G_{\mu\delta} + 2t\nabla^{\mu}G_{\mu\nu}.
\end{align*}
In the Cartesian coordinates, one can check immediately that
\begin{equation*}
\nabla_\nu\nabla^\mu K^{\delta}\,G_{\mu\delta}=2G(\partial_\nu,\partial_t).
\end{equation*}
Therefore, we obtain
\begin{align*}
[\Divergence, \mathcal{L}_K]G &= 2t\,\Divergence{G}.
\end{align*}
Finally, we have
\begin{align*}
 \mathcal{L}_K\big(r^2 \Divergence {G}\big) &= K(r^2)\Divergence{G}+r^2 \mathcal{L}_K \big(\Divergence {G}\big)\\
 &=2tr^2 \Divergence{G} + r^2 \Divergence \big( \mathcal{L}_K {G}\big) -r^2[\Divergence, \mathcal{L}_K]{G}\\
 &= r^2 \Divergence \big( \mathcal{L}_K G \big).
\end{align*}
For $Z=S$, recall that ${}^{(S)}\pi=m$. The computation in this case is straightforward. This yields \eqref{commutator formula 2}.
\end{proof}
Motivated by the formula \eqref{commutator formula 3}, we introduce the following commutator null form.
\begin{definition}
For any closed 2-form $G$ and any complex scalar field $f$, we define for any vector field $Z$ the quadratic form
\begin{equation*}
Q(f,G;Z)  =2\sqrt{-1} G_{\mu \nu}Z^\nu D^\mu f+\sqrt{-1} \nabla^\mu \big( Z^\nu G_{\mu \nu} \big)f.
\end{equation*}
\end{definition}
We then can write \eqref{commutator formula 3} as
\begin{equation}\label{commutator formula 4}
[r^2\Box_A, \Dt_Z]f = r^2 Q(f,F;Z).
\end{equation}

To avoid to many constants, in the sequel we use the convention that $B\lesssim K$ means that there is a constant $C$, depending only on the charge $q_0$ and the size of the initial data $C_0$ such that $B\leq CK$.
The next proposition manifests the  null structure of the quadratic form $Q(f,G;Z)$:
\begin{proposition}[Pointwise estimate of null form]\label{lemma null form}
For all $Z\in \mathcal{Z}$, $r\geq 1$ and $|u|\geq 1$, we have
\begin{equation}\label{null form estimate}
\begin{split}
|u|^{-\xi(Z)}|Q(f,G;Z)| &\lesssim \big(\frac{r}{|u|}|\rho|+|\alphab|\big)|D_L(rf)|+\big(\frac{r}{|u|}|\alpha|+\frac{|u|}{r}|\alphab|+|\sigma|\big)|\slashed{D}(rf)| \\
&\ \ \ +\big(|\alpha|+\frac{|u|}{r}|\rho|\big)|D_\Lb(rf)|+ \big(|\rho|+|\sigma|\big)|f|+\big(|u||J_\Lb|+\frac{r^2}{|u|}|J_L|+r|\slashed{J}|\big)|f|
\end{split}
\end{equation}
for all $G$ and  $f$. The current $J$ is associated to $G$, i.e., $J_\nu=\nabla^\mu G_{\mu\nu}$. Similarly, the null components $\alpha, \rho, \sigma$ and $\alphab$ are all defined with respect to $G$.
\end{proposition}

\begin{proof}
We show bound $Q(f, G;Z)$ for each $Z\in \mathcal{Z}$ one by one. We have
\begin{equation*}
\frac{Q(f,G;Z)}{\sqrt{-1}}  =\underbrace{2 r^{-1}\,G_{\mu \nu}\,Z^\nu D^\mu(rf)}_{\I_1}-  \underbrace{\big(Z^\nu J_\nu\big)\cdot f}_{\I_2} -\underbrace{\big(2r^{-1}\nabla^\mu r G_{\mu\nu}Z^\nu-\nabla^\mu Z^\nu G_{\mu\nu}\big)f}_{\I_3}.
\end{equation*}
For $Z=T$, we have
\begin{align*}
\I_1&=-\frac{1}{r}(\alpha+\alphab)\cdot \slashed{D}(rf)+\frac{1}{r}\rho\big(D_\Lb (rf)-D_L(rf)\big),\ \I_2=\frac{1}{2}(J_L+J_\Lb)f, \ \I_3=-r^{-1}\rho f.
\end{align*}
Therefore, we have
\begin{equation*}
|Q(f,G;T)| \lesssim\frac{|\slashed{D}(rf)|}{r}\big(|\alpha|+|\alphab|\big)+\frac{|\rho|}{r}\big(|f|+|D_L(rf)|+|D_\Lb (rf)|\big)+\big(|J_L|+|J_\Lb|\big)|f|.
\end{equation*}
For $Z=\Omega_{ij}$, we have
\begin{align*}
\I_1&\lesssim |D_L(rf)||\alphab|+|D_\Lb (rf)||\alpha|+|\sigma||\slashed{D} (rf)|,\ \ I_2\leq r|\slashed{J}|f,\\
\I_3&=\Big(\frac{1}{r}\big(G_{L\Omega_{ij}}-G_{\Lb \Omega_{ij}}\big) + \nabla_\Lb\Omega^A_{ij} G_{LA} + \nabla_L \Omega_{ij}^A G_{\Lb A}-\nabla^{A}\Omega_{ij}^{B}G_{AB}\Big) f =-\nabla^{A}\Omega_{ij}^{B}G_{AB}f.
\end{align*}
Therefore, we have
\begin{equation}\label{null form Omega}
|Q(f,G;\Omega_{ij})| \lesssim |D_L(rf)||\alphab|+|D_\Lb (rf)||\alpha|+|\sigma||f|+r |\slashed{J}||f|+|\sigma||\slashed{D} (rf)|.
\end{equation}
For $Z=S$, we have
\begin{align*}
\I_1&= 2\frac{u}{r}\rho D_\Lb (rf) - 2\frac{v}{r}\rho D_L (rf) -2\frac{v}{r}\alpha\cdot \slashed{D}(rf)-2\frac{u}{r}\alphab\cdot \slashed{D}(rf),\\
\I_2&=-v J_L f-u J_{\Lb}f,\ \ \I_3=-2\frac{u+v}{r}\rho f.
\end{align*}
Therefore, we have
\begin{equation*}
|Q(f,G;S)| \lesssim r^{-1} {|u|}\big(|\rho|| D_\Lb (rf)| + |\alphab||\slashed{D}(rf)|\big) + \big(|\rho||D_L (rf)| + |\alpha||\slashed{D}(rf)|\big)+|\rho||f|+r |J_L||f|+ |u| |J_{\Lb}||f|.
\end{equation*}
For $Z=K$, we have
\begin{align*}
\I_1&= -2\frac{u^2}{r}\rho D_\Lb (rf) + 2\frac{v^2}{r}\rho D_L (rf) + 2\frac{v^2}{r}\alpha\cdot \slashed{D}(rf)+2\frac{u^2}{r}\alphab\cdot \slashed{D}(rf),\\
\I_2&=-v^2 J_L f-u^2 J_{\Lb}f,\ \ \I_3=-4\frac{uv}{r}\rho f.
\end{align*}
Therefore, we have
\begin{equation}\label{null form estimate for K}
\begin{split}
|Q(f,G;K)| \lesssim & r^{-1} {u^2}\big(|\rho|| D_\Lb(rf)| + |\alphab||\slashed{D}(rf)|\big) + {r}\big(|\rho||D_L (rf)| + |\alpha||\slashed{D}(rf)|\big)\\
&+|u||\rho||f|+r^2 |J_L||f|+ u^2 |J_{\Lb}||f|.
\end{split}
\end{equation}
The lemma is an immediate consequence of the above estimates.
\end{proof}
 To analyze the higher order energy estimates of the solution, we will commute the equations with the vector fields twice. From the commutation formula \eqref{commutator formula 4}, we have the following identity:
\begin{equation}\label{formula to compute box z z}
\begin{split}
&\ \ \ r^2\Box_A\Dt_{Z_{1}} \Dt_{Z_2}{f}\\
&= [r^2\Box_A,\Dt_{Z_{1}}] \Dt_{Z_2}{f} +[r^2\Box_A, \Dt_{Z_2}]  \Dt_{Z_{1}}{f}+\big[\Dt_{Z_{1}},[r^2\Box_A, \Dt_{Z_2}]\big]{f}+\Dt_{Z_{1}} \Dt_{Z_2} \big(r^2\Box_A{f}\big)\\
&= r^2Q(\Dt_{Z_1}{f},F;Z_2)  + r^2Q(\Dt_{Z_2}{f},F;Z_1) +\big[\Dt_{Z_{1}},[r^2\Box_A, \Dt_{Z_2}]\big]{f}+\Dt_{Z_{1}} \Dt_{Z_2} \big(r^2\Box_A{f}\big).
	\end{split}
\end{equation}
Note that for solution of MKG equations the last term vanishes. In particular to derive the equation for the second order derivative of the solution, we need to estimate the double commutator.
\begin{proposition}\label{prop twice commutated}
For all $X,Y \in \mathcal{Z}$, we have
\begin{equation}\label{commutator twice commutated}
\big[\Dt_{Y},[r^2\Box_A, \Dt_{X}]\big]{f} = r^2Q({f}, F;[Y,X])+r^2Q({f}, \mathcal{L}_Y F;X)-2r^2F_{Y\mu}F_{X}{}^{\mu}{f}.
\end{equation}
\end{proposition}
\begin{proof}
First from Lemma \ref{commutator lemma}, we can write
\begin{equation*}
[r^2\Box_A, \Dt_{X}]{f}=r^2(2\sqrt{-1} X^\nu F_{\mu\nu}D^{\mu}{f}+\sqrt{-1} \nabla^\mu( F_{\mu\nu}X^\nu){f}).
\end{equation*}
Then for any two vector fields $X$ and $Y$, direction computation implies that
\begin{align*}
&\big[\Dt_{Y},[\Dt_{X}, r^2\Box_A]\big]{f}\\
&=-\nabla_Y\big(2\sqrt{-1}r^2X^\nu F_{\mu\nu}\big)D^\mu{f} - \nabla_Y\big(\sqrt{-1}r^2\nabla^\mu(F_{\mu\nu}X^\nu)\big){f}\\
&\quad +2\sqrt{-1}r^2X^\nu F_{\mu\nu}\nabla^\mu\big(\frac{Y(r)}{r}\big){f}
 -2\sqrt{-1}r^2X^\nu F_{\mu\nu}[D_Y,D^\mu]{f}\\
&=-\Big(\underbrace{\nabla_Y\big(2\sqrt{-1}r^2X^\nu F_{\mu\nu}\big)D^\mu{f}}_{\I_1} +\underbrace{\nabla_Y\big(\sqrt{-1}r^2\nabla^\mu(F_{\mu\nu}X^\nu)\big){f}}_{\I_2}\Big)\\
&\quad +2\sqrt{-1}r^2X^\nu F_{\mu\nu}\nabla^\mu\big(\frac{Y(r)}{r}\big){f} +2\sqrt{-1}r^2 X^\nu \nabla^\mu Y^\delta F_{\mu\nu} D_\delta {f} +2r^2F_{Y\mu}F_{X}{}^{\mu}{f}.
\end{align*}
Now for the term $\I_1$, we can compute that
\begin{align*}
\I_1&=2\sqrt{-1}Y(r^2)X^\nu F_{\mu\nu}D^\mu{f}+2\sqrt{-1}r^2\nabla_Y X^\nu F_{\mu\nu}D^\mu{f}+2\sqrt{-1}r^2X^\nu \nabla_Y F_{\mu\nu}D^\mu{f}\\
&=2\sqrt{-1}Y(r^2)X^\nu F_{\mu\nu}D^\mu{f}+2\sqrt{-1}r^2\big(\mathcal{L}_Y X^\nu +\nabla_X Y^\nu\big) F_{\mu\nu}D^\mu{f}\\
& \ \ \ +2\sqrt{-1}r^2X^\nu \big(\mathcal{L}_Y F_{\mu\nu}-\nabla_\mu Y^\delta F_{\delta\nu}-\nabla_\nu Y^\delta F_{\mu\delta}\big)D^\mu{f}\\
&=\underbrace{2\sqrt{-1}Y(r^2)X^\nu F_{\mu\nu}D^\mu{f}}_{\I_{11}}+\underbrace{2\sqrt{-1}r^2 \mathcal{L}_Y X^\nu F_{\mu\nu}D^\mu{f}}_{\I_{12}}\\
&\quad+\underbrace{2\sqrt{-1}r^2X^\nu \mathcal{L}_Y F_{\mu\nu} D^\mu{f}}_{\I_{13}}-2\sqrt{-1}r^2X^\nu \nabla_\mu Y^\delta F_{\delta\nu}D^\mu{f}.
\end{align*}
As for the term $\I_2$, we can further show that
\begin{align*}
\I_2&=\sqrt{-1}Y(r^2)\nabla^\mu(F_{\mu\nu}X^\nu) {f}+\sqrt{-1}r^2\nabla^\mu(F_{\mu\nu}\nabla_Y X^\nu) {f}\\
&\quad +\sqrt{-1}r^2\nabla^\mu(\nabla_Y F_{\mu\nu}X^\nu) {f}+\sqrt{-1}r^2[\nabla_Y, \nabla^\mu](F_{\mu\nu}X^\nu) {f}\\
&=\sqrt{-1}Y(r^2)\nabla^\mu(F_{\mu\nu}X^\nu) {f}+\sqrt{-1}r^2\nabla^\mu\big(F_{\mu\nu}(\mathcal{L}_Y X^\nu +
\nabla_X Y^\nu)\big) {f}-\sqrt{-1}r^2\nabla^\mu Y^\delta  F_{\mu\nu}\nabla_\delta X^\nu {f}\\
& \ \ \ +\sqrt{-1}r^2\nabla^\mu\Big(\big(\mathcal{L}_Y F_{\mu\nu}-\nabla_\mu Y^\delta F_{\delta\nu}-\nabla_\nu Y^\delta F_{\mu\delta}\big)X^\nu\Big) {f}-\sqrt{-1}r^2\nabla^\mu Y^\delta \nabla_\delta F_{\mu\nu}X^\nu {f} \\
&=\underbrace{\sqrt{-1}Y(r^2)\nabla^\mu(F_{\mu\nu}X^\nu) {f}}_{\I_{21}}+\underbrace{\sqrt{-1}r^2\nabla^\mu\big(F_{\mu\nu} \mathcal{L}_Y X^\nu \big) {f}}_{\I_{22}}+\underbrace{\sqrt{-1}r^2\nabla^\mu\big( \mathcal{L}_Y F_{\mu\nu}X^\nu\big) {f}}_{\I_{23}}\\
&\ \ \ -\sqrt{-1}r^2\nabla^\mu \big( \nabla_\mu Y^\delta F_{\delta\nu} X^\nu\big) {f}-\sqrt{-1}r^2\nabla^\mu Y^\delta \nabla_\delta F_{\mu\nu}X^\nu {f} -\sqrt{-1}r^2\nabla^\mu Y^\delta  F_{\mu\nu}\nabla_\delta X^\nu {f}
\end{align*}
We notice that the $\I_{1i}+\I_{2i}$'s can be expressed in terms of the quadratic form $Q$. We therefore derive that
\begin{align*}
&\big[\Dt_{Y},[\Dt_{X}, r^2\Box_A]\big]{f} \\
&= -Y(r^2)Q({f},F;X)-r^2Q({f}, F;[Y,X])-r^2Q({f},\mathcal{L}_Y F;X)+2r^2F_{Y\mu}F_{X}{}^{\mu}{f}\\
&\ \ \ +2\sqrt{-1}r^2X^\nu F_{\mu\nu}\nabla^\mu\big(\frac{Y(r)}{r}\big){f}+4\sqrt{-1}r^2X^\nu \,^{(Y)}\pi^{\delta\mu}F_{\mu\nu}D_\delta{f}\\
& \ \ +\sqrt{-1}r^2\nabla^\mu \big( \nabla_\mu Y^\delta F_{\delta\nu} X^\nu\big) {f}+\sqrt{-1}r^2\nabla^\mu Y^\delta \nabla_\delta F_{\mu\nu}X^\nu {f} +\sqrt{-1}r^2\nabla^\mu Y^\delta  F_{\mu\nu}\nabla_\delta X^\nu {f}\\
&= -Y(r^2)Q({f}, F;X)-r^2Q({f}, F;[Y,X])-r^2Q({f}, \mathcal{L}_Y F;X)+2r^2F_{Y\mu}F_{X}{}^{\mu}{f}\\
&\ \ \ +2\sqrt{-1}r^2X^\nu F_{\mu\nu}\nabla^\mu\big(\frac{Y(r)}{r}\big){f}+4\sqrt{-1}r^2X^\nu \,^{(Y)}\pi^{\delta\mu}F_{\mu\nu}D_\delta{f}\\
& \ \ +\sqrt{-1}r^2\Box Y^\delta F_{\delta\nu} X^\nu  {f}+2\sqrt{-1}r^2 \,^{(Y)}\pi^{\delta\mu}\left( \nabla_\delta F_{\mu\nu}X^\nu {f} + F_{\mu\nu} \nabla_\delta  X^\nu {f}\right).
\end{align*}
Note that the last two terms can be written as
\begin{equation*}
\,^{(Y)}\pi^{\delta\mu} (\nabla_\delta F_{\mu\nu}X^\nu {f} +F_{\mu\nu} \nabla_\delta  X^\nu {f})
=\,^{(Y)}\pi^{\delta\mu} \nabla_\delta\big(F_{\mu\nu}X^\nu \big){f}.
\end{equation*}
We now simplify the previous identity by checking vector fields $Y\in\mathcal{Z}$. We basically have two cases: when $Y=K$, $S$ or $Y=T$, $\Omega_{ij}$. For the latter situation, we notice that $Y$ is Killing and
\[
Y(r)=0,\quad ^{(Y)}\pi=0,\quad \Box Y^{\delta}=0.
\]
Therefore we conclude from the previous identity that
\begin{equation*}
\begin{split}
\big[\Dt_{Y},[\Dt_{X}, r^2\Box_A]\big]{f} &= -r^2Q({f},F;[Y,X])-r^2Q(f,\mathcal{L}_Y F;X)+2r^2F_{Y\mu}F_{X}{}^{\mu}{f}.
\end{split}
\end{equation*}
Now for the first case when $Y=K$ or $S$, note that we can write these two vector fields in a uniform way
\[
Y=u^p\Lb+v^p L,\quad p=1, 2,
\]
in which $p=1$ corresponds to the scaling vector field $S$ while $p=2$ stands for the conformal Killing vector field $K$. We then can compute that
\begin{align*}
  ^{(Y)}\pi=t^{p-1}m,\quad r^{-1}Y(r)=t^{p-1},\quad \Box Y^{\delta}\partial_{\delta}=p(p-1)\partial_t, \quad Y(r^2)=2 rY(r)=2r^2 t^{p-1}.
\end{align*}
We therefore can show that
\begin{equation*}
\begin{split}
&4X^\nu \,^{(K)}\pi^{\delta\mu}F_{\mu\nu}D_\delta{f}  +2X^\nu F_{\mu\nu}\nabla^\mu\big(\frac{K(r)}{r}\big){f}+\Box K^\delta F_{\delta\nu} X^\nu  {f}+2 ^{(K)}\pi^{\delta\mu} \nabla_\delta\big(F_{\mu\nu}X^\nu \big){f}\\
&=4X^\nu t^{p-1}m^{\delta\mu}F_{\mu\nu}D_\delta{f}  +2X^\nu F_{\mu\nu}\nabla^\mu(t^{p-1}){f}+ p(p-1) F_{0\nu} X^\nu  {f}+2 t^{p-1} m^{\delta\mu} \nabla_\delta\big(F_{\mu\nu}X^\nu \big){f}\\
&=4X^\nu t^{p-1}F_{\mu\nu}D^\mu {f}  + (p-2)(p-1) F_{0\nu} X^\nu  {f}+2 t^{p-1}  \nabla^\mu\big(F_{\mu\nu}X^\nu \big){f}\\
&=-\sqrt{-1}r^{-2} Y(r^2)Q({f}, F; X).
\end{split}
\end{equation*}
The last step follows by the definition of $Q({f}, F;X)$ and the fact that $p=1$ or $2$. In particular we have shown that estimate \eqref{commutator twice commutated} holds for all $X$, $Y\in\mathcal{Z}$.
\end{proof}
We are now ready to commute vector fields with MKG equations \eqref{MKG}. First of all, recall that we have defined the discrepancy index $\xi$ for  $Z\in \mathcal{Z}$, that is, the value of $T$, $\Omega_{ij}$, $S$ and $K$ are $-1$, $0$, $0$ and $1$ respectively. Let $\k =(k_0,k_1,k_2)$ be a triplet of nonnegative integers. The number $k_0$, $k_1$ and $k_2$ denote the number of index $-1$, $0$ and $1$ vector fields respectively. For a given $\k$, we define the {\bf discrepancy index} $\xi(\k)$ as
\begin{equation*}
\xi(\k)=k_2-k_0.
\end{equation*}
We also define $|\k|=k_0+k_1+k_2$. For derivatives on forms, for example the Maxwell field $F$ or the charge density $J$, we take the Lie derivative $\cL$. For any given tensor field $\mathcal{T}$, we use the expression $\LZk \mathcal{T}$ to denote the following $\k$-derivatives on forms for $Z \in \mathcal{Z}$:
\begin{equation*}
\LZk\mathcal{T} = \mathcal{L}_{Z_1}  \mathcal{L}_{Z_2}\cdots  \mathcal{L}_{Z_{|k|}} \mathcal{T},
\end{equation*}
where there are exactly $k_0$ degree $-1$ vector fields, exactly $k_1$ degree $0$ vector fields and exactly $k_2$ degree $1$ vector fields in the collection $\{Z_i | 1\leq i \leq |\k|\}$.
In the sequel we only consider situations where $|k|\leq 2$. It corresponds to commute at most two derivatives with the Maxwell-Klein-Gordon equations \eqref{MKG}.

As for derivatives on the complex scalar fields, we use the modified covariant derivative $\Dt$. Define
\begin{equation*}
\DZtk{f} = \Dt_{Z_1} \Dt_{Z_2}\cdots \Dt_{Z_{|\k|}}{f}.
\end{equation*}
For the solution $\phi$, we also define shorthand notations $\phi^{(\k)}= \Dt^{\k}_Z\phi$ and $\psi^{(\k)}= r\Dt^{\k}_Z\phi$. The previous commutator calculations allow us to derive the wave equations for $\phi^{(\k)}$. Define
\[
\Nk = \Box_A \phik.
\]
In particular we have $N^{(\mathbf{0})}=0$. By definition of $Q$, we see that $N^{(\mathbf{1})}=Q(\phi, F;Z)$. For the second order derivative $\phi^{(\mathbf{k})}=\Dt_{Z_1}\Dt_{Z_2}\phi$, Proposition \ref{prop twice commutated} together with the identity \eqref{formula to compute box z z} implies that
\begin{equation}\label{commutated wave equation}
 N^{(\mathbf{2})}=Q(\Dt_{Z_1}\phi,F;Z_2)  + Q(\Dt_{Z_2}\phi,F;Z_1)+Q(\phi, F;[Z_1,Z_2])+Q(\phi, \mathcal{L}_{Z_1} F;Z_2)-2F_{Z_1\mu}F_{Z_2}{}^{\mu}\phi.
\end{equation}
We now turn to the Maxwell part. We first explain our notations. For $r\neq 0$, we shall use the following shorthand notations to derivatives of the chargeless part of $F$:
\begin{equation*}
\alphak =\alpha(\LZk\Fc), \ \ \alphabk =\alphab(\LZk\Fc), \ \ \rhok =\rho(\LZk\Fc), \ \ \sigmak =\sigma(\LZk\Fc).
\end{equation*}
We notice that $\rho^{(\mathbf{0})} \neq \rho$ for $q_0\neq 0$. In most of the cases in this paper,  only the total number of derivatives in $\mathcal{L}^\k_{Z}$ is important. The exact form of $\k$ is usually irrelevant unless it is emphasized. Therefore, we will use shorthand notations ${\mathbf{(1)}}$ and ${\mathbf{(2)}}$ rather than writing down the explicit expression of $\k$, e.g., for $\alpha(\mathcal{L}_T\mathcal{L}_\Omega \Fc)$ we simply write it as $\alpha^{(\mathbf{2})}$.

For a given $\k$, we also define
\begin{equation}
\Jk = \LZk (r^2 J).
\end{equation}
We remark that $J^{(\mathbf{0})} = r^2 J$ which is {\bf not} the current $J$. The null components of $\Jk$ are denoted by $\LJk$, $\LbJk$ and $\sJk$. More precisely, we define
\begin{equation*}
\LJk = -\frac{1}{2} m(\LZk (r^2 J), \Lb), \ \ \LbJk=-\frac{1}{2}m(\LZk (r^2 J),L), \ \ \sJk_A = m(\LZk (r^2 J),e_A) \ \ \text{for} \ A=1,2.
\end{equation*}
In view of \eqref{Maxwell for Fc}, \eqref{Maxwell null Fc} and \eqref{commutator formula 2}, we can commute $\mathcal{L}_Z^{\k}$ to derive
\begin{equation}\label{Maxwell null commuted}
\left\{\begin{aligned}
L(r^2\rho^{(\k)})+\slashed{\Div} (r^2\alpha^{(\k)}) = \LJk, \ \ \Lb(r^2\rho^{(\k)})&-\slashed{\Div} (r^2\alphab^{(\k)}) =-\LbJk,\\
L(r^2\sigma^{(\k)})+\slashed{\Div} (r^2\,^*\alpha^{(\k)}) =0, \ \ \Lb(r^2\sigma^{(\k)})&+\slashed{\Div} (r^2\,^*\alphab^{(\k)}) =0,\\
\snabla_\Lb (r \alpha^{(\k)})_A-\snabla_A(r\rho^{(\k)})-\,^*\snabla_A(r\sigma^{(\k)}) &=r^{-1}\sJk_A,\\
\snabla_L (r \alphab^{(\k)})_A+\snabla_A(r\rho^{(\k)})-\,^*\snabla_A(r\sigma^{(\k)}) &=r^{-1}\sJk_A.
\end{aligned}
\right.
\end{equation}

\subsection{Multiplier vector fields and energy quantities}\label{sectioin multiplier vector fields}
One can associate the so-called energy momentum 2-tensor $T[G, {f}]_{\alpha\beta}$ to a closed 2-form $G$ and any complex scalar field ${f}$:
\begin{equation*}
\begin{aligned}
T[{G}, {f}]_{\alpha\beta} = \underbrace{{G}_{\alpha\mu}{G}_{\beta}{}^{\mu}-\frac{1}{4}m_{\alpha\beta}{G}_{\mu\nu}{G}^{\mu\nu}}_{T[{G}]_{\alpha\beta}}+\underbrace{\Re\big(\overline{D_\alpha{f}}D_\beta{f}\big)-\frac{1}{2}m_{\alpha\beta}\overline{D^\mu{f}}D_\mu {f}}_{T[{f}]_{\alpha\beta}}.
\end{aligned}
\end{equation*}
Given a smooth $\mathbb{R}-$valued function $\chi$ and a vector field $Y^\mu$, for any (multiplier) vector field $X$, we define the associated current as:
\begin{equation}\label{current twisted}
^{(X)}\widetilde{J}[{G}, {f}]_\mu =T[{G}, {f}]_{\mu\nu}X^\nu-\frac{1}{2}\nabla_\mu\chi \cdot |{f}|^2+\frac{1}{2}\chi\cdot \nabla_\mu\big(|{f}|^2\big) + Y_\mu.
\end{equation}
It can be computed that the space-time divergence of $^{(X)}\widetilde{J}[{G}, {f}]$ is given by the following formula:
\begin{equation}\label{divergence of J}
\begin{split}
\Divergence \big(\,^{(X)}\widetilde{J}[{G}, {f}]\big)&=\underbrace{T[{G}, {f}]_{\mu\nu}\,^{(X)} \pi^{\mu\nu}+\chi \overline{D^\mu{f}}D_\mu {f}-\frac{1}{2}\Box \chi \cdot |{f}|^2+\Divergence Y}_{\mathbf{D}_1}\\
&\ \ +\underbrace{\Re\big(\overline{\Box_A{f}} (D_X {f}+\chi {f})\big)+\nabla^\mu {G}_{\mu\nu}\cdot{G}^{\delta\nu} X_\delta+X^\mu F_{\mu\nu} {J}[{f}]^\nu}_{\mathbf{D}_2},
\end{split}
\end{equation}
where the current $J_\mu[{f}] = \Im({f}\cdot \overline{D_\mu{f}})$.

In this paper, we will use two types of vector fields as multipliers. In particular the multiplier $X$ will be chosen as $X=\partial_t$ or $X=r^p L$ ($0\leq p \leq 2$). Their deformation tensors are recorded in the following table:

\ \ \ \ \ \ \ \ \ \ \ \ \ \ \ \ \ \ \ \ \ \ \ \ \ \ \ \ \ \
\begin{tabular}{|c|*{6}{c}|}
\hline
 & $\pi_{LL}$ & $\pi_{L\Lb}$ & $\pi_{\Lb\Lb}$ & $\pi_{L A}$ & $\pi_{\Lb A}$ & $\pi_{AB}$\\
\hline
$X=\partial_t$ & $0$ & $0$ & $0$ & $0$ & $0$ & $0$\\
\hline
$X=r^p L$ & $0$ & $-p r^{p-1}$ & $2 p r^{p-1}$ & $0$ & $0$ & $r^{p-1}\slashed{g}_{AB}$\\
\hline
\end{tabular}

To define energy quantities, we first clarify the measure over different regions or hypersurfaces. In the sequel, the variable $\vartheta$ denotes for a coordinate on the unit sphere $\mathbf{S}^2$. We have
\begin{equation*}
\begin{split}
&\int_{\H_{r_1}^{r_2}} \cdot  = \int_{\frac{r_1}{2}}^{\frac{r_2}{2}}\int_{\mathbf{S}^2} \cdot \, \, \,r^2 dv d\vartheta, \quad\int_{\Hb_{r_2}^{r_1}} \cdot  = \int_{-\frac{r_2}{2}}^{-\frac{r_1}{2}}\int_{ \mathbf{S}^2} \cdot \, \, \,r^2 du d\vartheta, \\
 &\int_{\B_{r_1}^{r_2}} \cdot  = \int_{r_1}^{r_2}\int_{\mathbf{S}^2} \cdot \, \, \,r^2 dr d\vartheta, \quad \int_{\D_{r_1}^{r_2}} \cdot  =\frac{1}{2}\int \int \int_{ \mathbf{S}^2} \cdot \, \, \,r^2 du dv d\vartheta.
\end{split}
\end{equation*}
Given ${G}$ and ${f}$, the energy through $\B_{r_1}^{r_2}$ and the energy flux through $\H_{r_1}^{r_2}$ or $\Hb_{r_2}^{r_1}$ are defined as
\begin{equation*}
\begin{split}
\E[{G},{f}](\B_{r_1}^{r_2})&:=\int_{\B_{r_1}^{r_2}} |\alpha({G})|^2+ |\alphab({G})|^2+ |\rho({G})|^2+ |\sigma({G})|^2 + |D{f}|^2,\\
\F[{G},{f}](\H_{r_1}^{r_2})&:=\int_{\H_{r_1}^{r_2}} |\alpha({G})|^2+ |\rho({G})|^2+ |\sigma({G})|^2 + |D_L{f}|^2 + |\slashed{D}{f}|^2,\\
\Fb[{G},{f}](\Hb_{r_2}^{r_1})&:=\int_{\Hb_{r_2}^{r_1}} |\alphab({G})|^2+ |\rho({G})|^2+ |\sigma({G})|^2 + |D_\Lb{f}|^2 + |\slashed{D}{f}|^2.
\end{split}
\end{equation*}
One can take $X=\partial_t$, $\chi=0$, $Y=0$ and then integrate \eqref{divergence of J} over $\D_{r_1}^{r_2}$. This leads to the classical energy identity:
\begin{lemma}[Classical energy identity]\label{classical energy identity}
For all closed 2-form $G$ and any complex scalar field ${f}$ and all $ 0<r_1 < r_2$, we have
\begin{equation}\label{classical energy inequality}
\F[G,{f}](\H_{r_1}^{r_2}) + \Fb[G,{f}](\Hb_{r_2}^{r_1}) = \E[G,{f}](\B_{r_1}^{r_2})-\int_{D_{r_1}^{r_2}}\Re\big(\overline{\Box_A{f}} \cdot D_{\partial_t} {f} \big)+\nabla^\mu G_{\mu\nu}\cdot G_0{}^{\nu} + F_{0\mu} {J}[{f}]^\mu.
\end{equation}
\end{lemma}
If we choose $X=r^p L$, $\chi = r^{p-1}$ and $Y=\frac{p}{2}r^{p-2} |{f}|^2 L$, this leads to the $r$-weighted energy identity
\begin{lemma}[$r$-weighted energy identity]\label{r weighted energy estimates}
For all closed 2-form ${G}$ and complex scalar field ${f}$, we have
\begin{equation}\label{r wieghted}
\begin{split}
&\ \ \ \ \int_{\B_{r_1}^{r_2}}r^{p-2}\big(|D_L (rf)|^2+|\slashed{D} (rf)|^2\big)+r^{p}\big(|\alpha({G})|^2+|\rho({G})|^2+|\sigma({G})|^2\big)\\
&=\int_{\H_{r_1}^{r_2}}r^{p-2}\big(|D_L (rf)|^2+r^2|\alpha({G})|^2\big)+\int_{\Hb_{r_2}^{r_1}}r^{p-2}\big(|\slashed{D} (rf)|^2+r^2|\rho({G})|^2+r^2|\sigma({G})|^2\big)\\
&\ \ \ +\frac{1}{2}\int_{\D_{r_1}^{r_2}}r^{p-3}\Big(p\big(|D_L (rf)|^2+r^2|\alpha({G})|^2\big)+(2-p)(|\slashed{D}(rf)|^2+r^2|\rho({G})|^2+r^2|\sigma({G})|^2)\Big)\\
&\ \ \ + \underbrace{\int_{\D_{r_1}^{r_2}} r^{p-1}\Re\big(\overline{\Box_A{f}} \cdot D_L (rf)\big)+r^p \nabla^\mu {G}_{\mu\nu}\cdot{G}_{L}{}^{\nu}+r^p F_{L\mu} {J}[{f}]^\mu}_{\text{$r$-weighted error term} \ \mathbf{Err_p}}
\end{split}
\end{equation}
for all $ 0<r_1 < r_2$ and $p \in [0,2]$.
\end{lemma}
One can find the detailed proof in \cite{yangILEMKG}. For reader's interest, we provide the proof here.
\begin{proof}
The identity \eqref{r wieghted} is equivalent to the following one:
\begin{align}
\notag
&\ \ \ \ \underbrace{\int_{r_1}^{r_2}\int_{\mathbf{S}^2}r^{p}\big(|D_L (rf)|^2+|\slashed{D} (rf)|^2\big)+r^{p+2}\big(|\alpha({G})|^2+|\rho({G})|^2+|\sigma({G})|^2\big) dr d\vartheta}_{L_1}\\
\notag
&=\underbrace{\int_{\frac{r_1}{2}}^{\frac{r_2}{2}}\int_{\mathbf{S}^2} r^{p} \big(|D_L (rf)|^2+r^2|\alpha({G})|^2\big)dvd\vartheta}_{R_1}+\underbrace{\int_{-\frac{r_2}{2}}^{-\frac{r_1}{2}}\int_{\mathbf{S}^2}r^{p}\big(|\slashed{D} (rf)|^2+r^2|\rho({G})|^2+r^2|\sigma({G})|^2\big) du d\vartheta}_{R_2}\\
\notag
&\ \ \ +\int_{u}\int_{\vartheta}\int_{v}r^{p-1}\Big(p(|D_L (rf)|^2+r^2|\alpha({G})|^2\big)+(2-p)(|\slashed{D}(rf)|^2+r^2|\rho({G})|^2+r^2|\sigma({G})|^2)\Big) dv d\vartheta du\\
&\ \ \ + \int_{\D_{r_1}^{r_2}} r^{p-1}\Re\big(\overline{\Box_A{f}} \cdot D_L (rf)\big)+r^p \nabla^\mu {G}_{\mu\nu}\cdot{G}_{L}{}^{\delta}+r^p F_{L\mu} {J}[{f}]^\mu. \label{r wieghted'}
\end{align}
We take $X=r^p L$, $\chi = r^{p-1}$ and $Y=\frac{p}{2}r^{p-2} |{f}|^2 L$ in \eqref{divergence of J}. We can compute that
\begin{align*}
T[{G}, {f}]_{\mu\nu}\,^{(X)} \pi^{\mu\nu} &= -\frac{p-2}{2}r^{p-1}\big(\rho({G})^2+\sigma({G})^2\big)-\frac{p}{2}r^{p-1}|\slashed{D}{f}|^2 \\ & \quad\ +\frac{p}{2}r^{p-1}\big(|D_L{f}|^2+|\alpha({G})|^2\big)+r^{p-1}D_L{f} D_\Lb{f}, \\
\chi \overline{D^\mu{f}}D_\mu {f} &=-r^{p-1}D_L{f} D_\Lb{f} +r^{p-1}|\slashed{D}{f}|^2,\ \ \
-\frac{1}{2}\Box \chi \cdot |{f}|^2=-\frac{p(p-1)}{2}r^{p-3}|{f}|^2,\\
\Divergence Y&=\frac{p^2}{2}r^{p-3}|{f}|^2+pr^{p-2}\Re(\overline{D_L {f}}\cdot \phi).
\end{align*}
Since $r^2|D_L f|^2=|D_L (rf)|^2-L(r|{f}|^2)$, we obtain
\begin{equation}\label{C3}
\begin{split}
\mathbf{D}_1&=\frac{2-p}{2}r^{p-3}\big(r^2\rho({G})^2+r^2\sigma({G})^2+|\slashed{D}(rf)|^2\big)+\frac{p}{2}r^{p-3}\big(|\alpha({G})|^2+|D_L(rf)|^2\big),\\
\mathbf{D}_2&= r^{p-1}\Re\big(\overline{\Box_A{f}} \cdot D_L (rf)\big)+r^p \nabla^\mu {G}_{\mu\nu}\cdot{G}_{L}{}^{\nu}+r^p F_{L\mu} {J}[{f}]^\mu,
\end{split}
\end{equation}
where $\mathbf{D}_i$'s are defined in \eqref{divergence of J}. Now we turn to the boundary integrals. On $\B_{r_1}^{r_2}$, the normal $n^\mu$ is $\partial_t$, we have
\begin{equation*}
\begin{split}
^{(X)}\widetilde{J}[{G}, {f}]^\mu n_\mu = &\frac{1}{2}r^{p-2}\big(r^2\alpha({G})^2+r^2\rho({G})^2+r^2\sigma({G})^2+|D_L(rf)|^2+|\slashed{D}(rf)|^2\big)\\
&-\frac{1}{2}\big((p+1)r^{p-2}|{f}|^2+r^{p-1}\partial_r(|{f}|^2)\big).
\end{split}
\end{equation*}
Therefore we derive that
\begin{equation}\label{C5}
\begin{split}
\int_{\B_{r_1}^{r_2}}\,^{(X)}\widetilde{J}[{G}, {f}]^\mu n_\mu &=\frac{1}{2} \underbrace{\int_{\B_{r_1}^{r_2}} r^2\alpha({G})^2+r^2\rho({G})^2+r^2\sigma({G})^2+|D_L(rf)|^2+|\slashed{D}(rf)|^2}_{L_1 \ in \ \eqref{r wieghted'}} \\
&\quad \quad -\frac{1}{2}\int_{r_1}^{r_2}\int_{\mathbf{S}^2}\underbrace{(p+1)r^{p}|{f}|^2+r^{p+1}\partial_r(|{f}|^2)}_{=\partial_r(r^{p+1}|{f}|^2)} d\vartheta dr\\
&=\frac{1}{2}L_1+\frac{1}{2}\int_{\S_{r_1}^{r_1}}r^{p-1}|{f}|^2-\frac{1}{2}\int_{\S_{r_2}^{r_2}}r^{p-1}|{f}|^2.
\end{split}
\end{equation}
On $\H_{r_1}^{r_2}$, the normal $n^\mu$ is $L$. Hence,
\begin{equation*}
^{(X)}\widetilde{J}[{G}, {f}]^\mu n_\mu = r^{p-2}\big(r^2\alpha({G})^2+|D_L(rf)|^2+|\slashed{D}(rf)|^2\big)-\frac{1}{2}\big((p+1)r^{p-2}|{f}|^2+r^{p-1}L(|{f}|^2)\big).
\end{equation*}
Therefore, we have
\begin{equation}\label{C6}
\begin{split}
\int_{\H_{r_1}^{r_2}}\,^{(X)}\widetilde{J}[{G}, {f}]^\mu n_\mu &=\underbrace{\int_{\frac{r_1}{2}}^{\frac{r_2}{2}}\int_{\mathbf{S}^2} r^{p} \big(|D_L (rf)|^2+r^2|\alpha({G})|^2\big)dvd\vartheta}_{R_1 \ in \ \eqref{r wieghted'}} \\ &\quad-\frac{1}{2}\int_{\frac{r_1}{2}}^{\frac{r_2}{2}}\int_{\mathbf{S}^2}\underbrace{(p+1)r^{p}|{f}|^2+r^{p+1}L(|{f}|^2)}_{=L (r^{p+1}|{f}|^2)} d\vartheta dv\\
&=R_1+\frac{1}{2}\int_{\S_{r_1}^{r_1}}r^{p-1}|{f}|^2-\frac{1}{2}\int_{\S_{r_1}^{r_2}}r^{p-1}|{f}|^2.
\end{split}
\end{equation}
On $\Hb_{r_1}^{r_2}$, the normal $n^\mu$ is $\Lb$. Hence,
\begin{equation*}
^{(X)}\widetilde{J}[{G}, {f}]^\mu n_\mu = r^{p-2}\big(r^2\rho({G})^2+r^2\sigma({G})^2+|\slashed{D}(rf)|^2\big)+\frac{1}{2}\big(-(p+1)r^{p-2}|{f}|^2+r^{p-1}\Lb(|{f}|^2)\big).
\end{equation*}
Therefore, we have
\begin{equation}\label{C7}
\begin{split}
\int_{\Hb_{r_1}^{r_2}}\,^{(X)}\widetilde{J}[{G}, {f}]^\mu n_\mu &=\underbrace{\int_{-\frac{r_2}{2}}^{-\frac{r_1}{2}}\int_{\mathbf{S}^2}r^{p}\big(|\slashed{D} (rf)|^2+r^2|\rho({G})|^2+r^2|\sigma({G})|^2\big) du d\vartheta}_{R_2 \ in \ \eqref{r wieghted'}} \\
&\quad +\frac{1}{2}\int_{-\frac{r_2}{2}}^{-\frac{r_1}{2}}\int_{\mathbf{S}^2}\underbrace{-(p+1)r^{p}|{f}|^2+r^{p+1}\Lb(|{f}|^2)}_{=\Lb (r^{p+1}|{f}|^2)} d\vartheta dv\\
&=R_2+\frac{1}{2}\int_{\S_{r_1}^{r_2}}r^{p-1}|{f}|^2-\frac{1}{2}\int_{\S_{r_2}^{r_2}}r^{p-1}|{f}|^2.
\end{split}
\end{equation}
By combining\eqref{C3}-\eqref{C7} we can the use Stokes formula to complete the proof.
\end{proof}

To end this section, we introduce energy norms. For all $r_1>0$, $p \in[0,2]$, $\k \leq 2$ and a given small $\delta>0$, we define the standard energy norms
\begin{equation*}
\begin{split}
\mathcal{E}^{(\k)}(\phi;r_1)&=\F[0, \phi^{(\k)}](\H_{r_1})+\sup_{r_2 \geq r_1}\Fb[0,\phi^{(\k)}](\Hb_{r_2}^{r_1}),\\
\mathcal{E}^{(\k)}(\Fc;r_1)&=\F[\LZk(\Fc), 0](\H_{r_1})+\sup_{r_2 \geq r_1}\Fb[\LZk(\Fc),0](\Hb_{r_2}^{r_1}),
\end{split}
\end{equation*}
and the $r^p$-weighted energy norms
\begin{equation*}
\begin{split}
\mathcal{E}^{(\k)}(\phi;p;r_1)&=\int_{\H_{r_1}}r^{p-2}|D_L \psi^{(\k)}|^2+\sup_{r_2 \geq r_1}\int_{\Hb_{r_2}^{r_1}}r^{p-2}|\slashed{D} \psi^{(\k)}|^2 \\&\qquad  +\int_{\D_{r_1}} r^{p-3}\Big(p|D_L\psi^{(\k)}|^2+(2-p)|\slashed{D}\psi^{(\k)}|^2\big)\Big),\\
\mathcal{E}^{(\k)}(\Fc;p;r_1)&=\int_{\H_{r_1}}r^p|\alphak|^2+\sup_{r_2 \geq r_1}\int_{\Hb_{r_2}^{r_1}}r^p\big(|\rhok|^2+|\sigmak|^2\big) \\ &\qquad+\int_{\D_{r_1}} r^{p-1}\Big(p|\alphak|^2+(2-p)\big(|\rhok|^2+|\sigmak|^2\big)\Big).
\end{split}
\end{equation*}

\section{The analysis in the exterior region 0: set-up and zeroth order energy estimates}

We emphasize again that, till the end of the paper, $(\phi,F)$ is a given solution of \eqref{MKG} associated to a given finite energy smooth initial datum. According to the result of Klainerman-Machedon \cite{MKGkl}, the solution exits globally.

\subsection{The exterior region}
We take a positive number $R_*$ and require that $R_*\geq 1$. The number $R_*$ should be understood as a large number and its size will be determined later on (solely by the initial datum). It determines the so-called exterior region $\mathcal{D}_{R_*}$. It is the grey region in the following picture.

\ \ \ \ \ \ \ \ \ \ \ \ \ \ \ \ \ \ \ \ \ \includegraphics[width=3.8in]{exterior.pdf}

The boundary of the exterior region consists of two pieces: the outgoing null hypersurface $\mathcal{H}_{r_1}^{r_2}$ and its bottom $\mathcal{B}_{R_*}$. The exterior region is also the domain of dependence of $\mathcal{B}_{R_*}$.

According to \eqref{initial data 1}, the following number is the initial energy for $\phi$ and $\Fc$ on $\mathcal{B}_{R_*}$:
\begin{equation}\label{def for mathring E}
\mathring{\E}_{\geq R_*}  = \sum_{k=0}^2 \int_{r\geq R_*}\int_{\mathbf{S}^2}\Big[r^{2k+6+8\varepsilon_0}\big(|D^k D\phi_0|^2+|D^k\phi_1|^2 + |\nabla^k \Fc|^2\big)+r^{4+8\varepsilon_0}|\phi_0|^2\Big]\, r^2dr d\vartheta.
\end{equation}
Since we will eventually take a large $R_*$, we can assume that for a given small positive number $\mathring{\varepsilon}<1$ one has
\begin{equation*}
\mathring{\E}_{\geq R_*}  \leq \mathring{\varepsilon}.
\end{equation*}

Before we proceed to the energy estimates, we prove a key technical lemma. The lemma is indispensable to the estimate on terms with critical decay (coming from the charge term) of the current $J$.
\begin{lemma}\label{lemma key}(Key technical lemma) Let $C_0, C_1,C_2, \gamma_0$ and $\varepsilon_0$ be positive numbers.  The constant $\varepsilon_0$ is small, say $\varepsilon_0=0.001$ and $\gamma_0 > 100 \varepsilon_0$. Let ${f}$ be an arbitrary scalar field satisfying the following two conditions:

1). For all $r_1 \geq R_*$, we have
\begin{equation}\label{condition 1 in lemma}
\int_{\B_{r_1}} r^{-2}|{f}|^2 \leq C_0 \epsilonc r_1^{-\gamma_0}.
\end{equation}

2). For all $r_2>r_1\geq R_*$, we have
\begin{equation}\label{condition 2 in lemma}
\int_{\H_{r_1}^{r_2}} |D_L {f}|^2 \leq C_1 \epsilonc r_1^{-\gamma_0}+C_2\int_{\D_{r_1}^{r_2}}\frac{1}{r^2}|{f}||D_L {f}|.
\end{equation}
Then there exists a constant $C$ depending only on $C_0,C_1,C_2$ and $\varepsilon_0$ such that
\begin{equation}
\int_{\H_{r_1}^{r_2}} |D_L {f}|^2 \leq C\epsilonc \cdot r_1^{-\gamma_0+\varepsilon_0}.
\end{equation}
\end{lemma}
\begin{remark}
As we have mentioned in the introduction that the error term caused by the charge may lead to a logarithmic growth by using the standard Gronwall's inequality. The importance of this lemma is to avoid this log-loss with the price of losing a bit of decay. This technique was introduced by the first author in \cite{yangILEMKG} to derive the energy flux decay. For completeness we summarize it as a Lemma which will also be used to obtain higher order energy estimates.
\end{remark}
\begin{proof}
Recall that $u_+=1+|u|$. By virtue of Cauchy-Schwarz inequality, we have
\begin{align*}
\mathbf{I}&:=\int_{\D_{r_1}^{r_2}}\frac{1}{r^2}|{f}||D_L {f}|\lesssim \underbrace{\int_{u}\int_{v}\int_{\vartheta} u_+^{-1}r^2|D_L {f}|^2 du dv d\vartheta}_{\mathbf{I}_1}+\underbrace{\int_{u}\int_{v}\int_{\vartheta}u_+ r^{-2}|{f}|^2 du dv d\vartheta}_{\mathbf{I}_2}.
\end{align*}
We first deal with $\mathbf{I}_2$. In view of the case $\gamma=4$ in \eqref{inequality hardy on H} of Appendix \ref{Appendix tools}, we have
\begin{align*}
\mathbf{I_2} &\lesssim \int_{\frac{r_1}{2}}^{\frac{r_2}{2}}u_+\Big(\int_{\H_{2u}^{r_2}}r^{-4}|{f}|^2 \Big)du \lesssim \int_{\frac{r_1}{2}}^{\frac{r_2}{2}}u_+\Big( u_+^{-3}\int_{\S_{2u}^{2u}}|{f}|^2+u_+^{-2}\int_{\H_{2u}^{r_2}}\big|D_L{f}\big|^2 \Big)du\\
&=\underbrace{\int_{\frac{r_1}{2}}^{\frac{r_2}{2}}u_+^{-2}\Big(\int_{\S_{2u}^{2u}}|{f}|^2 \Big)du}_{\ \ \ \ \ \ \ \ \ \ \ \ \ \ \displaystyle\approx \int_{\B_{r_1}^{r_2}}|{f}|^2 \lesssim  C_0 r_1^{-\gamma_0}\epsilonc}+\int_{\D_{r_1}^{r_2}}u_+^{-1}\big|D_L{f}\big|^2.
\end{align*}
Here the implicit constant is a universal constant. In particular there exists a universal constant $C$ such that
\begin{equation*}
\begin{split}
\int_{\H_{r_1}^{r_2}} |D_L {f}|^2
&\leq C C_1 r_1^{-\gamma_0}\epsilonc + C C_2 \int_{\D_{r_1}^{r_2}}\big(1+|u|\big)^{-1}|D_L {f}|^2 \\
&= C C_1 r_1^{-\gamma_0}\epsilonc + C C_2 \int_{{r_1}}^{r_2}\frac{1}{s}\Big(\int_{\H_{s}^{r_2}}|D_L {f}|^2\Big) ds.
\end{split}
\end{equation*}
We now apply Gronwall's inequality in Lemma \ref{lemma gronwall} (by setting $f(s) = \int_{\H_{s}^{r_2}}|D_L {f}|^2$) to conclude that
\begin{equation*}\label{A1}
\int_{\H_{r_1}^{r_2}} |D_L {f}|^2 \leq C(C_1+C_2)\epsilonc \cdot r_1^{-\gamma_0} \big(r_2 r_1^{-1})^{CC_2}.
\end{equation*}
For a given $r_1$, define $r_1^* := r_1^{1+\frac{\varepsilon_0}{2CC_2}}$. Then for all $r_2 \leq r_1^*$, we have
\begin{equation}\label{first case losing epsilon}
\int_{\H_{r_1}^{r_2}} |D_L {f}|^2 \lesssim_{C_1,C_2}\epsilonc \cdot r_1^{-\gamma_0+\frac{\varepsilon_0}{2}}, \ \ \ r_2 \leq r_1^*=r_1^{1+\frac{\varepsilon_0}{2CC_2}}.
\end{equation}
Here the implicit constant depends only on $C_1$ and $C_2$.
We now study the case $r_2 > r_1^*$ in a different way. In fact, we take $r_2=\infty$ and we have
\begin{equation}\label{II2}
\begin{split}
\mathbf{I}&=\int_{\D_{r_1}}\frac{1}{r^2}|{f}||D_L {f}| \leq \int_{\D_{r_1}} u_+^{-1-\frac{\varepsilon_0}{2CC_2}}|D_L {f}|^2+u_+^{1+\frac{\varepsilon_0}{2CC_2}}r^{-4}|{f}|^2\\
&=\underbrace{\int_{\frac{r_1}{2}}   u_+^{-1-\frac{\varepsilon_0}{2CC_2}} \int_{\H_{2u}}|D_L{f}|^2 du }_{\mathbf{II}_1}+\underbrace{\int_{\frac{r_1}{2}}^{\infty} u_+^{1+\frac{\varepsilon_0}{2CC_2}}\Big(\overbrace{\int_{v}\int_{\vartheta}r^{-2}|{f}|^2dv d\vartheta}^{\displaystyle \mathbf{II'}_2(u) = \int_{\H_{2u}} r^{-4}|{f}|^2} \Big) du}_{\mathbf{II}_2},
\end{split}
\end{equation}
where $C$ is the constant in the definition of $r_1^*$. Because $u_+^{-1-\frac{\varepsilon_0}{2CC_2}}$ is integrable in $u$, Gronwall's inequality enables us to bound ${\mathbf{II}_1}$ by the righthand side of \eqref{condition 2 in lemma}. It suffices to control ${\mathbf{II}_2}$.

\ \ \ \ \ \ \ \ \ \ \ \ \ \ \ \ \ \ \ \ \ \ \ \ \ \ \ \ \ \ \ \ \ \ \ \ \ \ \ \  \includegraphics[width=2.5 in]{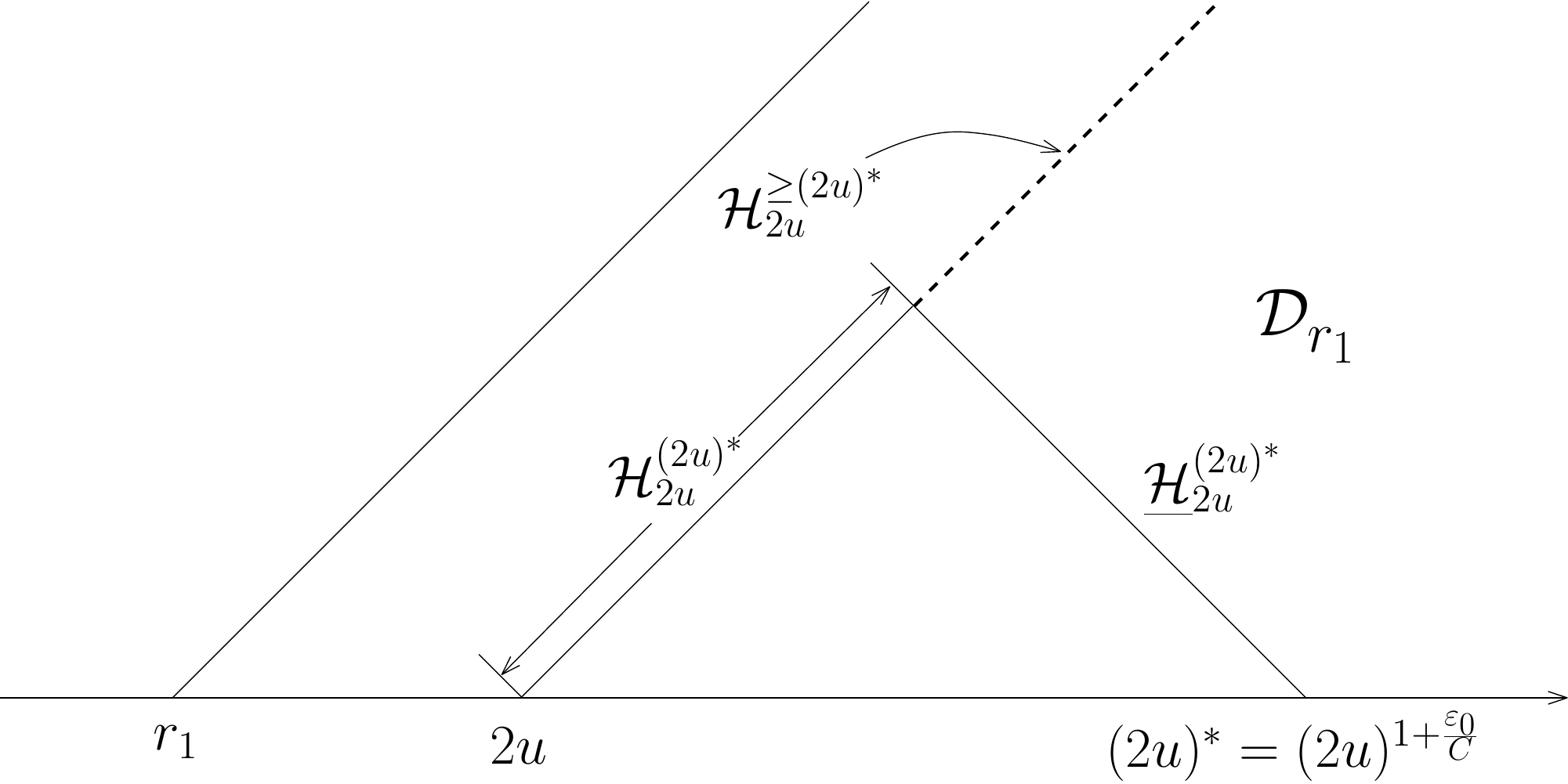}

The cone $\H_{2u}$ is the union of $\H_{2u}^{(2u)^*}$ and $\H_{2u}^{\geq (2u)^*}$ which is the cone emanated from the sphere $\S_{2u}^{(2u)^*}$ ( $(2u)^*= (2u)^{1+\frac{\varepsilon_0}{2CC_2}}$). In the picture, $\H_{2u}^{\geq (2u)^*}$ is denoted by the dashed line. Thus, we have
\begin{align*}
\mathbf{II'}_2(u) = \underbrace{\int_{\H_{2u}^{(2u)^*}} r^{-4} |{f}|^2}_{\mathbf{A}}+\underbrace{\int_{\H_{2u}^{\geq (2u)^*}} r^{-4} |{f}|^2}_{\mathbf{B}}.
\end{align*}
For the term $\mathbf{A}$, we can apply the $\gamma=4$ case of \eqref{inequality hardy on H} and we obtain
\begin{equation*}
\begin{split}
\int_{\H_{2u}^{(2u)^*}}\frac{1}{r^4}|{f}|^2 &\lesssim  \underbrace{u_+^{-3}\int_{\S_{2u}^{2u}}|{f}|^2}_{\mathbf{A}_1}+u_+^{-2}\underbrace{\int_{\H_{2u}^{(2u)^*}}\big|D_L{f}\big|^2}_{\text{$\mathbf{A}_2$, use \eqref{first case losing epsilon}}}\lesssim_{C_1,C_2}  u_+^{-3}\int_{\S_{2u}^{2u}}|{f}|^2+\epsilonc \cdot u_+^{-\gamma_0-2+\frac{\varepsilon_0}{2}}.
\end{split}
\end{equation*}
So the contribution of $\mathbf{A}$ in $\mathbf{II}_2$ is bounded by (we can always assume that $CC_2\geq 1$)
\begin{align*}
\int_{\frac{r_1}{2}}^{\infty}u_+^{1+\frac{\varepsilon_0}{2CC_2}}\mathbf{A}\, du &\lesssim_{C_1,C_2}\underbrace{ \int_{\frac{r_1}{2}}^{\infty}u_+^{-2+\frac{\varepsilon_0}{2CC_2}} \int_{\S_{2u}^{2u}}|{f}|^2 du}_{\text{use \eqref{condition 1 in lemma} and Lemma \ref{lemma changing decay rates}}}+\epsilonc\int_{\frac{r_1}{2}}^{\infty}u_+^{-\gamma_0-1+\varepsilon_0+\frac{\varepsilon_0}{2CC_2}} du\\
&\lesssim_{C_0,C_1,C_2}\epsilonc \cdot u^{-\gamma_0+\varepsilon_0}.
\end{align*}
For the term $\mathbf{B}$, we can apply Lemma \ref{lemma Hardy on H} with $\gamma=4$ and $r_2=\infty$ to obtain that
\begin{equation}
\mathbf{B} \lesssim \underbrace{\big((2u)^*\big)^{-3}\int_{\S_{2u}^{(2u)^*}}|{f}|^2}_{\mathbf{B}_1}+\underbrace{(u^*)^{-2}\int_{\H_{2u}^{\geq (2u)^*}}\big|D_L{f}\big|^2}_{\mathbf{B}_2}.
\end{equation}
The contribution of $\mathbf{B}_2$ in $\mathbf{II}_2$ is bounded by
\begin{align*}
\int_{\frac{r_1}{2}}^{\infty}u_+^{1+\frac{\varepsilon_0}{2CC_2}} \mathbf{B}_2 du
&\lesssim \int_{\frac{r_1}{2}}^{\infty}u_+^{-1-\frac{\varepsilon_0}{2CC_2}} \int_{\H_{2u}}|D_L \psi|^2 du.
\end{align*}
Therefore, this is the same expression as $\mathbf{II_1}$ and we can ignore this term.

For $\mathbf{B}_1$, up to a universal constant, according to Lemma \ref{lemma Hardy on H} for $r_1 =2u$ and $r_2 = (2u)^*$, we have
\begin{align*}
\mathbf{B}_1\leq |u|^{-3}\int_{\S_{2u}^{2u}}|\psi|^2
+\int_{\H_{2u}^{(2u)^*}}\frac{1}{r^2}|D_L\psi|^2.
\end{align*}
Now the first term on the right hand side is $\mathcal{A}_1$ and the second term is $\mathcal{A}_2$ which have already been estimated. In particular we have bounded all the terms and thus complete the proof.
\end{proof}

\subsection{Zeroth order energy estimates} We prove the zeroth order energy estimate.
\begin{proposition}\label{proposition energy estimate 0th order} For $r_1 \geq R_*$ and $1\leq p \leq 2$, we have
\begin{equation}\label{energy estimate 0th order}
\begin{split}
\mathcal{E}^{(0)}(\phi;r_1)+\mathcal{E}^{(0)}(\Fc;r_1) &\leq 2\epsilonc \cdot r_1^{-6-6\varepsilon_0},\\
\mathcal{E}^{(0)}(\phi;p;r_1)+\mathcal{E}^{(0)}(\Fc;p;r_1) &\leq 2\epsilonc \cdot r_1^{p-6-6\varepsilon_0}.
\end{split}
\end{equation}
\end{proposition}
\begin{proof}
 We first prove the second estimate for the endpoint case $p=2$ (This is the only case which has applications in the current work. Indeed, for $p<2$, the proof is exactly the same and one may also see \cite{yangILEMKG}). We set $G=\Fc$, ${f}=\phi$ and $rf=\psi=r\phi$ in Lemma \ref{r weighted energy estimates}. Thus, \eqref{r wieghted} yields
\begin{equation}\label{r weighted for Fc}
\begin{split}
 & \ \ \int_{\H_{r_1}^{r_2}}|D_L \psi|^2+r^2|\alphac|^2+\int_{\Hb_{r_2}^{r_1}}|\slashed{D} \psi|^2+r^2|\rhoc|^2+r^2|\sigmac|^2 +\int_{\D_{r_1}^{r_2}}r^{-1}\big(|D_L \psi|^2+r^2|\alphac|^2\big)+\mathbf{Err_p}\\
 &=\int_{\B_{r_1}^{r_2}}|D_L \psi|^2+|\slashed{D} \psi|^2+r^{2}\big(|\alphac|^2+|\rhoc|^2+|\sigmac|^2\big) \leq \epsilonc \cdot r_1^{-4-8\varepsilon_0}.
\end{split}
\end{equation}
It suffices to bound the term $\mathbf{Err_p}$ of \eqref{r wieghted}. It is straightforward to see that the integrand of $\mathbf{Err_p}$ is $q_0 r^{p-2}J_L$. Hence,
\begin{align*}
\big|\mathbf{Err_p}\big|&\stackrel{p=2}{=}\big| q_0 \int_{\D_{r_1}^{r_2}}J_L\big|=\big|q_0 \int_{\D_{r_1}^{r_2}}\Im(\overline{D_L\phi}\cdot\phi)\big|=\big|q_0 \int_{\D_{r_1}^{r_2}}r^{-2}\Im(\overline{D_L\psi}\cdot\psi)\big|.
\end{align*}
In particular, \eqref{r weighted for Fc} implies
\begin{equation*}
\begin{split}
\int_{\H_{r_1}^{r_2}}|D_L \psi|^2 \lesssim \epsilonc \cdot r_1^{-4-8\varepsilon_0}+ \int_{\D_{r_1}^{r_2}}r^{-2}|\psi||D_L\psi|
\end{split}
\end{equation*}
We now can use Lemma \ref{lemma key} (with $\gamma=-4-8\varepsilon_0$) and we obtain that
\begin{equation*}
\begin{split}
\int_{\H_{r_1}^{r_2}}|D_L \psi|^2 + \int_{\D_{r_1}^{r_2}}r^{-2}|\psi||D_L\psi|\lesssim \epsilonc \cdot r_1^{-4-7\varepsilon_0}.
\end{split}
\end{equation*}
 This leads to the $r$-weighted energy estimates with $p=2$. The case when $p<2$ follows in a similar way. Once we have control on the error term caused by the nonzero charge, the first estimate of the proposition is an immediate consequence of the basic energy identity \eqref{classical energy inequality}. We may always assume that $\R_*$ is large enough so that by afford a factor $r_1^{-\varepsilon_0}$ we beat all the constants from Lemma \ref{lemma key}.This completes the proof.
\end{proof}

\section{The analysis in the exterior region 1: bootstrap ansatz and decay estimates}

\subsection{Bootstrap ansatz}
We make two sets of ansatz on the exterior region $\D_{R_*}$. The first set is on the energy quantities:
\begin{equation*}
\boxed{
\begin{aligned}
\mathcal{E}^{(\k)}(\Fc;r_1)+\mathcal{E}^{(\k)}(\phi;r_1)& \leq 4{\epsilonc}r_1^{-6+2\xi(\k)-(6-2|\k|)\varepsilon_0},\\
\mathcal{E}^{(\k)}(\Fc;p;r_1)+\mathcal{E}^{(\k)}(\phi;p;r_1)&\leq 4{\epsilonc}r_1^{p-6+2\xi(\k)-(6-2|\k|)\varepsilon_0},
\end{aligned}
\ \  \ r_1\geq R_*, \ |\k| = 1, 2 \ \text{and} \ p\in[0,2]. \quad {\mathbf{(B)}}}
\end{equation*}
The second set is on the current terms:
\begin{equation*}
\boxed{
\begin{aligned}
& \ \text{For all} \ r_1\geq R_*, \ |\k| \leq 1, \ \text{we assume}\\
& \int_{\H_{r_1}}\frac{{|\LJk|^2}}{r_1^2}+\int_{\H_{r_1}}\frac{|\sJk|^2}{r^2}+\sup_{r_2 \geq r_1}\int_{\Hb_{r_2}^{r_1}} \frac{|\sJk|^2}{r^2} +\sup_{r_2 \geq r_1}r_1^{\frac{3}{2}}\int_{\Hb_{r_2}^{r_1}} \frac{|\LbJk|^2}{r^{\frac{7}{2}}}\leq 4{\epsilonc}^2 r_1^{-8+2\xi(\k)-4\varepsilon_0}\\
&\text{and for } |\mathbf{k}|=2, \ \text{we assume} \\
& \int_{\H_{r_1}}\frac{|\LJk|^2}{r_1^2}+\int_{\H_{r_1}}\frac{|\sJk|^2}{r^2} +\sup_{r_2 \geq r_1}r_1^{\frac{3}{2}}\int_{\Hb_{r_2}^{r_1}} \frac{|\LbJk|^2}{r^{\frac{7}{2}}}\leq 4{\epsilonc}^2 r_1^{-8+2\xi(\k)-4\varepsilon_0}.
\end{aligned}\quad {\mathbf{(C)}}}
\end{equation*}

We will show that if  $\epsilonc$ is sufficiently small (by setting $R_*$ to be sufficiently large),  the constant $4$ in the ansatz can be improved to be $2$. In the sequel, the bootstrap argument should be understood dynamically (as one does in solving the Cauchy problem): we assume that the solution is defined in the region where $0\leq t \leq T_*$ and $T_*$ is a fixed positive number. Therefore, for  sufficiently small $T_*$, $\mathbf{(B)}$ holds. The bootstrap argument will show that one can indeed replace the constant $4$ by $2$ and this is independent of $T_*$. Therefore, by letting $T_*\rightarrow \infty$, we obtain estimates on the entire spacetime.

Based on these ansatz, we will first derive pointwise estimates on $\Fc$ and $\phi$.
\subsection{Pointwise decay estimates of the Maxwell field}
We use $(\mathbf{B})$ to bound $\alphac$, $\rhoc$, $\sigmac$ and $\alphabc$.
\begin{proposition}\label{Proposition pointwise decay of Maxwell}
We have the following decay estimates:
\begin{equation*}
\begin{split}
|\alphac| &\lesssim \sqrt{{\epsilonc}} r^{-3} u_+^{-1-\varepsilon_0}, \ \ |\rhoc|+|\sigmac| \lesssim \sqrt{{\epsilonc}} r^{-2} u_+^{-2-\varepsilon_0},\ \ |\alphabc| \lesssim  \sqrt{{\epsilonc}} r^{-1} u_+^{-3-\varepsilon_0}.
\end{split}
\end{equation*}
\end{proposition}
\begin{proof}
{\bf Step 1. $L^\infty$ estimate of $\alphabc$.}

In view of \eqref{formula compare Lie and nablaslash}, Lemma \ref{lemma commuting Z with null decomposition}, the last equation in \eqref{Maxwell null commuted} and the fact that $\Lb=2T-L$, we have
\begin{align*}
\int_{\Hb_{r_2}^{r_1}} |\mathcal{L}_\Lb \big(\mathcal{L}_\Omega\alphabc\big)|^2 &\leq\int_{\Hb_{r_2}^{r_1}} |\mathcal{L}_T \big(\mathcal{L}_\Omega\alphabc\big)|^2+\int_{\Hb_{r_2}^{r_1}} |\mathcal{L}_L \big(\mathcal{L}_\Omega\alphabc\big)|^2\\
&= \int_{\Hb_{r_2}^{r_1}} |\alphab^{(\mathbf{2})}|^2 +  \int_{\Hb_{r_2}^{r_1}} |-\snabla \rho^{(\mathbf{1})} +\,^*\snabla \sigma^{(\mathbf{1})} + r^{-2}\slashed{J}^{(\mathbf{1})} + \frac{1}{r}\alphab^{(\mathbf{1})}|^2.
\end{align*}
We remark that in this case $\xi(\mathbf{2})= - 1$ and $\xi(\mathbf{1})=0$. By $\mathbf{(B)}$, we then have
\begin{align*}
\int_{\Hb_{r_2}^{r_1}} |\mathcal{L}_\Lb \big(\mathcal{L}_\Omega\alphabc\big)|^2 &\lesssim {\epsilonc} r_1^{-8-2\varepsilon_0}+ {\epsilonc} r_1^{-6-4\varepsilon_0}+ \int_{\Hb_{r_2}^{r_1}} \frac{|\slashed{J}^{(\mathbf{1})}|^2}{r^4}.
\end{align*}
We can use the first term of $\mathbf{(C)}$ to bound the last term in the above inequality.
 Recall that for forms $\Xi$, we have $ r^2 |\snabla \Xi|^2 \lesssim |\mathcal{L}_\Omega \Xi|^2+|\Xi|^2$. Therefore, $\mathbf{(B)}$ together $(\mathbf{C})$ imply that
\begin{align*}
\int_{\Hb_{r_2}^{r_1}} |\mathcal{L}_\Lb \big(\mathcal{L}_\Omega\alphabc\big)|^2 \lesssim {\epsilonc} r_1^{-6-4\varepsilon_0}.
\end{align*}
By $\mathbf{(B)}$, we also have
\begin{align*}
\int_{\Hb_{r_2}^{r_1}} |\mathcal{L}_\Omega \big(\mathcal{L}_\Omega\alphabc\big)|^2 \lesssim  {\epsilonc} r_1^{-6-2\varepsilon_0}.
\end{align*}
We then can apply \eqref{Sobolve on incoming null hypersurfaces} to derive
\begin{equation*}
\begin{split}
\|\mathcal{L}_\Omega\alphabc\|_{L^4(\S_{r_1}^{r_2})} &\lesssim r_2^{-\frac{1}{2}} \sqrt{{\epsilonc}} r_1^{-3-\varepsilon_0}.
\end{split}
\end{equation*}
We can repeat the above argument by switching $\mathcal{L}_\Omega\alphabc$ to $\alphabc$ and we obtain
\begin{equation*}
\begin{split}
\|\alphabc\|_{L^4(\S_{r_1}^{r_2})} &\lesssim r_2^{-\frac{1}{2}} \sqrt{{\epsilonc}} r_1^{-3-2\varepsilon_0}.
\end{split}
\end{equation*}
Compared to the $L^4$ bound of $\mathcal{L}_\Omega \alphabc$, this bound gains an extra $r^{-\varepsilon_0}$ because we use one less derivative in this case. This is clear from the bootstrap ansatz $\mathbf{(B)}$. We then apply the Sobolev inequality \eqref{Sobolev on sphere} on $\S^{r_1}_{r_2}$. In view of the fact that $\frac{r_1+r_2}{2} \approx r_2$ and $|u|\approx r_1$ on $\S^{r_1}_{r_2}$, we obtain
\begin{equation}
|\alphabc| \lesssim  \sqrt{{\epsilonc}} r^{-1} u_+^{-3-\varepsilon_0}.
\end{equation}

{\bf Step 2. $L^\infty$ estimate of $\rhoc$ and $\sigmac$.}
We only derive the bound on $\rhoc$ since $\sigmac$ can be bounded exactly in the same manner. First of all, for $\mathbf{l}=(0,1,0)$ and $\k = (0,2,0)$, we have
\begin{align*}
\Lb\big(\mathcal{L}_\Omega (r\rhoc)\big)=r^{-1}\Lb(r^2\rho^{(\mathbf{1})})-\rho^{(\mathbf{1})}.
\end{align*}
Thus by using the null equation for $\rhoc$ as well as the bootstrap assumptions we can show that
\begin{align*}
\int_{\Hb_{r_2}^{r_1}}|\Lb \big(\mathcal{L}_\Omega(r \rhoc)\big)|^2 &\leq\int_{\Hb_{r_2}^{r_1}}r^{-2}|\mathcal{L}_\Lb \big(r^2 \rho^{(\mathbf{l})}\big)|^2+|\rho^{(\mathbf{l})}|^2 \stackrel{\eqref{Maxwell null commuted}}{=}\int_{\Hb_{r_2}^{r_1}}r^{-2}|\slashed{\Div} (r^2\alphab^{(\mathbf{l})}) -\LbJl|^2+|\rho^{(\mathbf{l})}|^2\\
&\leq \int_{\Hb_{r_2}^{r_1}}|\alphab^{(\mathbf{k})})|^2+|\rho^{(\mathbf{l})}|^2 + r^{-2}|\LbJl|^2\stackrel{(\mathbf{B}),(\mathbf{C})}{\lesssim}  {\epsilonc} r_1^{-6-2\varepsilon_0}.
\end{align*}
By the $p=2$ case of $(\mathbf{B})$, we also have
\begin{align*}
\int_{\Hb_{r_2}^{r_1}}r^2 |\mathcal{L}_\Omega\rhoc|^2+ \int_{\Hb_{r_2}^{r_1}}r^2 |\mathcal{L}_\Omega \big(\mathcal{L}_\Omega\rhoc\big)|^2 &\lesssim   {\epsilonc} r_1^{-4-2\varepsilon_0}.
\end{align*}
Therefore, we obtain that
\begin{equation*}
\int_{\Hb_{r_2}^{r_1}} |\mathcal{L}_\Omega \big(r \rhoc\big)|^2+\int_{\Hb_{r_2}^{r_1}} \Big|\mathcal{L}_\Lb \Big(\mathcal{L}_\Omega\big(r\rhoc\big)\Big)\Big|^2 +\int_{\Hb_{r_2}^{r_1}}\Big|\mathcal{L}_\Omega \Big(\mathcal{L}_\Omega\big(r\rhoc\big)\Big)\Big|^2 \lesssim  {\epsilonc} r_1^{-4-2\varepsilon_0}.
\end{equation*}
According to \eqref{Sobolve on incoming null hypersurfaces}, the above energy estimate implies that
\begin{equation*}
\begin{split}
\|\mathcal{L}_\Omega \big(r\rhoc\big)\|_{L^4(\S_{r_1}^{r_2})} &\lesssim r_2^{-\frac{1}{2}}  \sqrt{{\epsilonc}} r_1^{-2-\varepsilon_0}.
\end{split}
\end{equation*}
Similarly, we have
\begin{equation*}
\| r\rhoc\|_{L^4(\S_{r_1}^{r_2})} \lesssim r_2^{-\frac{1}{2}}  \sqrt{{\epsilonc}} r_1^{-2-2\varepsilon_0}.
\end{equation*}
We notice that this is a similar bound but with an extra $r_1^{-\varepsilon_0}$ due to one less derivative compared to the previous case.
We then apply \eqref{Sobolev on sphere} on $\S^{r_1}_{r_2}$ and conclude that
\begin{equation}
|\rhoc| \lesssim \sqrt{{\epsilonc}} r^{-2} u_+^{-2-\varepsilon_0}.
\end{equation}
\begin{remark}
By using the flux on $\H_{r_1}^{r_2}$, the same argument yields:
\begin{equation*}
|\alphac| \lesssim \sqrt{{\epsilonc}} r^{-2} u_+^{-2-\varepsilon_0}.
\end{equation*}
This is not optimal and we will obtain a better decay in the next step.
\end{remark}

{\bf Step 3. $L^\infty$ estimate of $\alphac$.} The sharp decay of $\alphac$ relies on the commutator $K$ and the $r^p$-weighted energy estimate. Note that for an arbitrary two form $G$, we have
\begin{equation*}
\alpha(\mathcal{L}_{K} G)_A = v^{-1}\nabla_L (v^3 \alpha(G))_A+u^2\nabla_\Lb \alpha(G)_A+u\alpha(G)_A.
\end{equation*}
Therefore, we have
\begin{equation*}
v\alpha(\mathcal{L}_{K} G)_A = \nabla_L \big(v^3 \alpha(G)\big)_A+u^2\nabla_\Lb \big(r\alpha(G)\big)_A+(u^2+uv)\alpha(G)_A.
\end{equation*}
If we take $G=\mathcal{L}_\Omega \Fc$, in view of the third equation in \eqref{Maxwell null commuted},  we also have
\begin{equation}\label{eq:1}
\nabla_{L}(v^3\mathcal{L}_{\Omega}\alphac)=v\alpha^{(0,1,1)}-(u^2+uv)\alpha^{(0,1,0)}-u^2\big[\snabla(r\rho^{(0,1,0)})+\,^*\snabla(r\sigma^{(0,1,0)})-r^{-1}\slashed{J}^{(0,1,0)}\big].
\end{equation}
By virtue of the bootstrap assumptions $(\textbf{B})$, $(\textbf{C})$ and $|u|\lesssim r$,  especially the $r^p$-weighted energy norms, we have
\begin{align*}
\int_{\H_{r_1}} |\nabla_{L}(v^3\mathcal{L}_{\Omega}\alphac)|^2 & \lesssim \int_{\H_{r_1}} v^2|\alpha^{(0,1,1)}|^2+|u|^2v^2|\alpha^{(0,1,0)}|^2+|u|^4\big(|\rho^{(0,2,0)}|^2+|\sigma^{(0,2,0)}|^2\big) +\frac{|u|^4}{r^2}|\slashed{J}^{(0,1,0)}|^2\\
&\lesssim {\epsilonc} r_1^{-2-2\varepsilon_0}.
\end{align*}
In view of $v=u+r$, we have
\begin{equation}\label{highest weight estimate for alpha}
 \|\nabla_{L}(rv^2\mathcal{L}_{\Omega}\alphac)\|_{L^2(\H_{r_1})}  \lesssim \sqrt{\epsilonc} r_1^{-1-\varepsilon_0}.
\end{equation}
This estimate can be used to get a sharp decay estimates for $\|\mathcal{L}_\Omega\alphac\|_{L^2(\S_{r_1}^{r_2})}$. In fact, we have
\begin{align*}
 \|v^2\mathcal{L}_\Omega\alphac\|_{L^2(\S_{r_1}^{r_2})}^2- \|v^2\mathcal{L}_\Omega\alphac\|_{L^2(\S_{r_1}^{r_1})}^2&= \int_{\frac{r_2}{2}}^{\frac{r_1}{2}}\int_{\mathbf{S}^2}L\big(\big| r v^2 \mathcal{L}_\Omega\alphac\big|^2\big)d\vartheta dv  \\
 &\lesssim  \int_{\frac{r_2}{2}}^{\frac{r_1}{2}}\int_{\mathbf{S}^2} |\nabla_L\big( r v^2 \mathcal{L}_\Omega\alphac \big)||r \mathcal{L}_\Omega\alphac |r^2d\vartheta dv\\
 &\leq \|\nabla_{L}(rv^2\mathcal{L}_{\Omega}\alphac)\|_{L^2(\H_{r_1})}\|r\mathcal{L}_{\Omega}\alphac 	\|_{L^2(\H_{r_1})}.
\end{align*}
Thus,
\begin{align*}
 \|v^2\mathcal{L}_\Omega\alphac\|_{L^2(\S_{r_1}^{r_2})}^2&\lesssim\|v^2\mathcal{L}_\Omega\alphac\|_{L^2(\S_{r_1}^{r_1})}^2+\|\nabla_{L}(rv^2\mathcal{L}_{\Omega}\alphac)\|_{L^2(\H_{r_1})}\|r\mathcal{L}_{\Omega}\alphac\|_{L^2(\H_{r_1})}\\
 &\lesssim {\epsilonc} r_1^{-3-3\varepsilon_0}.
\end{align*}
As a result, we obtain
\begin{equation}\label{L2 bound for LOmegaAlpha on sphere}
  \|\mathcal{L}_\Omega\alphac\|_{L^2(\S_{r_1}^{r_2})} \lesssim \sqrt{\epsilonc} r_2^{-2}r_1^{-\frac{3}{2}-\frac{3}{2}\varepsilon_0}.
\end{equation}
One can also bound $\|\mathcal{L}_\Omega\alphac\|_{L^4(\S_{r_1}^{r_2})}$. We take $\Xi=r\mathcal{L}_\Omega\alphac$ in \eqref{Sobolve on incoming null hypersurfaces} and we obtain
\begin{align*}
  r_2^3 \|\mathcal{L}_\Omega\alphac\|_{L^4(\S_{r_1}^{r_2})}^2 &\lesssim \int_{\H_{r_2}^{r_1}} |r\mathcal{L}_\Omega\alpha|^2+\int_{\H_{r_2}^{r_1}} \frac{1}{r^2}|\mathcal{L}_L (r^2\mathcal{L}_\Omega\alphac)|^2+\int_{\H_{r_2}^{r_1}} {r^2}|\mathcal{L}_\Omega (\mathcal{L}_\Omega\alphac)|^2\\
  & \lesssim \int_{\H_{r_2}^{r_1}} r^2|\alpha^{(0,1,0)}|^2+\underbrace{\int_{\H_{r_2}^{r_1}} \frac{1}{r^2}|\mathcal{L}_L (r^2\mathcal{L}_\Omega\alphac)|^2}_{bounded \ in \ \eqref{highest weight estimate for alpha}}+\int_{\H_{r_2}^{r_1}} {r^2}|\alpha^{(0,2,0)}|^2\\
  & \lesssim \epsilonc r_1^{-4-2\varepsilon_0}.
\end{align*}
In other words, we have
\begin{equation}\label{L4 bound for LOmegaAlpha on sphere}
  \|\mathcal{L}_\Omega\alphac\|_{L^4(\S_{r_1}^{r_2})} \lesssim \sqrt{\epsilonc} r_2^{-\frac{3}{2}}r_1^{-2-\varepsilon_0}.
\end{equation}
For $q\in [2,4]$, by interpolating \eqref{L2 bound for LOmegaAlpha on sphere} and \eqref{L4 bound for LOmegaAlpha on sphere}, we have
\begin{equation}\label{Lq bound for LOmegaAlpha on sphere}
  \|\mathcal{L}_\Omega\alphac\|_{L^q(\S_{r_1}^{r_2})} \lesssim \sqrt{\epsilonc} r_2^{-\big(1+\frac{2}{q}\big)}r_1^{-\big(\frac{5}{2}-\frac{2}{q}+(\frac{1}{2}+\frac{2}{q})\varepsilon_0\big)}, \ 2\leq q \leq 4.
\end{equation}
We now try to improve decay in $r_2$ in \eqref{Lq bound for LOmegaAlpha on sphere} for $2<q<\frac{9}{4}$. For this purpose, we choose $\gamma$ so that
\begin{equation*}
\gamma+\frac{2}{q}=3.
\end{equation*}
Therefore, we have
\begin{align*}
 \|r^\gamma\mathcal{L}_\Omega\alphac\|_{L^q(\S_{r_1}^{r_2})}^q- \|r^\gamma\mathcal{L}_\Omega\alphac\|_{L^q(\S_{r_1}^{r_1})}^q&= \int_{\frac{r_2}{2}}^{\frac{r_1}{2}}\int_{\mathbf{S}^2}L\big(\big| r^3 \mathcal{L}_\Omega\alphac\big|^q\big)d\vartheta dv  \\
 &\lesssim  \int_{\frac{r_2}{2}}^{\frac{r_1}{2}}\int_{\mathbf{S}^2} |\nabla_L\big( r v^2 \mathcal{L}_\Omega\alphac \big)|| r^3 \mathcal{L}_\Omega\alphac\big|^{q-1}d\vartheta dv.
\end{align*}
According to Cauchy-Schwarz inequality, we have
\begin{equation}\label{aux 1}
 \|r^\gamma\mathcal{L}_\Omega\alphac\|_{L^q(\S_{r_1}^{r_2})}^q \lesssim \|r^\gamma\mathcal{L}_\Omega\alphac\|_{L^q(\S_{r_1}^{r_1})}^q+\|\nabla_L\big( r v^2 \mathcal{L}_\Omega\alphac \big)\|_{L^2(\H_{r_1})} \underbrace{\|r^{3q-5}  |\mathcal{L}_\Omega\alphac |^{q-1}\|_{L^2(\H_{r_1})}}_{\mathbf{I}}.
\end{equation}
To bound $\mathbf{I}$, since $q<\frac{9}{4}$, we proceed as follows
\begin{align*}
\mathbf{I}&=\Big(\int_{\H_{r_1}}r^{6q-10}  |\mathcal{L}_\Omega\alphac |^{2q-2}\Big)^{\frac{1}{2}}=\Big(\int_{r_1}^{r_2}r^{6q-10} \|\mathcal{L}_\Omega\alphac \|_{L^{2q-2}({\S_{r_1}^r})}^{2q-2} dr\Big)^{\frac{1}{2}}\\
&\stackrel{\eqref{Lq bound for LOmegaAlpha on sphere}}{\lesssim} \Big(\int_{r_1}^{r_2}r^{6q-10}\cdot \epsilonc^{q-1}\cdot r^{-2q}r_1^{-\big(5q-7+(q+1)\varepsilon_0\big)} dr\Big)^{\frac{1}{2}}\lesssim \epsilonc^{\frac{q-1}{2}} r_1^{-\frac{1}{2}\big(q+2+(q+1)\varepsilon_0\big)} .
\end{align*}
In view of \eqref{conf estimates aux} and \eqref{highest weight estimate for alpha}, we have
\begin{equation*}
 \|r^\gamma\mathcal{L}_\Omega\alphac\|_{L^q(\S_{r_1}^{r_2})}^q \lesssim \epsilonc^{\frac{q}{2}} r_1^{-q(1+\varepsilon_0)} + \epsilonc^{\frac{q}{2}} r_1^{-\frac{1}{2}\big(q+4+(q+3)\varepsilon_0\big)}
\end{equation*}
Therefore, we have
\begin{equation}\label{Lq final estimate 1}
 \|\mathcal{L}_\Omega\alphac\|_{L^q(\S_{r_1}^{r_2})} \lesssim \sqrt{\epsilonc}r_2^{\frac{2}{q}-3} r_1^{-(1+\varepsilon_0)}, \ \text{for} \ 2<q<\frac{9}{4}.
\end{equation}
We remark that, compared to \eqref{Lq bound for LOmegaAlpha on sphere}, the decay in $r_2$ has been improved. Similary, we also have
\begin{equation}\label{Lq final estimate 2}
 \|\alphac\|_{L^q(\S_{r_1}^{r_2})} \lesssim \sqrt{\epsilonc}r_2^{\frac{2}{q}-3} r_1^{-(1+\varepsilon_0)}, \ \text{for} \ 2<q<\frac{9}{4}.
\end{equation}
We can fix a $q \in (2,\frac{9}{4})$ (say $q=\frac{17}{8}$) and apply \eqref{Sobolev on sphere}. Therefore, \eqref{Lq final estimate 1} and \eqref{Lq final estimate 2} together yield
\begin{equation*}
|\alphac| \lesssim  \sqrt{{\epsilonc}} r^{-3} u_+^{-1-\varepsilon_0}.
\end{equation*}
This completes the proof.
\end{proof}

\subsection{Pointwise decay estimates of the scalar field}

We start with the decay estimate of $\phi$ on the initial slice $\B_{R_*}$. By \eqref{Sobolve on incoming null hypersurfaces} and \eqref{Sobolev on sphere}, we have
\begin{equation*}
\begin{split}
\|\phi\|_{L^4(\S_{r_1}^{r_1})} \lesssim \sqrt{{\epsilonc}} r_1^{-\frac{5}{2}-4\varepsilon_0}, \ \ \|D_\Omega \phi\|_{L^4(\S_{r_1}^{r_1})} \lesssim \sqrt{{\epsilonc}} r_1^{-\frac{5}{2}-4\varepsilon_0}.
\end{split}
\end{equation*}
By \eqref{Sobolev on sphere}, we have
\begin{equation*}
 \|\phi\|_{L^\infty(\S_{r_1}^{r_1})}\lesssim \sqrt{{\epsilonc}} r_1^{-3-4\varepsilon_0}.
\end{equation*}

\begin{proposition}\label{Proposition pointwise decay of scalar field}
For the solution $(\phi, F)$ of the MKG equations on the exterior region $\{t+R_*\leq |x|\}$, the scalar field verifies the following decay estimates:
\begin{equation*}
\begin{split}
|\phi| &\lesssim \sqrt{{\epsilonc}} r^{-1} u_+^{-\frac{5}{2}-2\varepsilon_0},\ \ |D_\Lb \phi| \lesssim  \sqrt{{\epsilonc}} r^{-1} u_+^{-3-\varepsilon_0},\\
|\slashed{D}\phi| &\lesssim  \sqrt{{\epsilonc}} r^{-2} u_+^{-2-\varepsilon_0}, \ \ |D_L\psi| \lesssim \sqrt{{\epsilonc}} r^{-2} u_+^{-1-\varepsilon_0}.
\end{split}
\end{equation*}
\end{proposition}
\begin{proof}
{\bf Step 1. $L^\infty$ estimate of $\phi$.} For $k\leq 2$, by Lemma \ref{Sobolev trace estimates on outgoing null hypersurfaces} and $(\mathbf{B})$ we have
\begin{equation}\label{bound on L2 on spheres}
\begin{split}
\|D^{k}_\Omega \phi\|^2_{L^2(\S_{r_1}^{r_2})}&\lesssim \|D^{k}_\Omega \phi\|^2_{L^2(\S_{r_1}^{r_1})}+\frac{1}{r_1}\int_{\H_{r_1}} |D_L D_\Omega^k\psi|^2\stackrel{(\mathbf{B})}{\lesssim} {\epsilonc} r_1^{-5-2\varepsilon_0}.
\end{split}
\end{equation}
We now use \eqref{Sobolev on sphere} to conclude that
\begin{equation*}
\|\phi\|_{L^\infty(\S_{r_1}^{r_2})} {\lesssim}  \sqrt{{\epsilonc}} r^{-1} u_+^{-\frac{5}{2}-\varepsilon_0}.
\end{equation*}
Here note that $u_+=1+\frac{1}{2} |t-r|=1+\frac{1}{2} r_1$.
We can indeed  improve the estimates by gaining a $r_1^{-\varepsilon_0}$. First of all, notice that in \eqref{bound on L2 on spheres}, for $k\leq 1$, we have
\begin{equation*}
\begin{split}
\|D^{k}_\Omega \phi\|^2_{L^2(\S_{r_1}^{r_2})}&{\lesssim} {\epsilonc} r_1^{-5-4\varepsilon_0}.
\end{split}
\end{equation*}
To save one derivative, we can use the second equation in \eqref{Sobolve for scalar filed on incoming null hypersurfaces} to derive that
\begin{equation*}
\begin{split}
\|D_\Omega \phi\|^2_{L^4(\S_{r_1}^{r_2})}&{\lesssim} {\epsilonc} r_2^{-1}r_1^{-4-4\varepsilon_0}.
\end{split}
\end{equation*}
Thus, by \eqref{Sobolev on sphere} again, we have
\begin{equation}\label{pointwise bound on phi}
\|\phi\|_{L^\infty(\S_{r_1}^{r_2})} {\lesssim}  \sqrt{{\epsilonc}} r^{-1} r_1^{-\frac{5}{2}-2\varepsilon_0}{\lesssim}  \sqrt{{\epsilonc}} r^{-1} u_+^{-\frac{5}{2}-2\varepsilon_0}.
\end{equation}
{\bf Step 2. $L^\infty$ estimate of $D_\Lb \phi$.} We first bound $\int_{\Hb_{r_2}^{r_1}} |D_\Lb \big(D_\Omega D_\Lb \phi \big)|^2$. It can be split into:
\begin{align*}
\int_{\Hb_{r_2}^{r_1}} |D_\Lb \big(D_\Omega D_\Lb \phi \big)|^2 &\leq \underbrace{\int_{\Hb_{r_2}^{r_1}}  |D_T \big(D_\Omega D_\Lb \phi \big)|^2}_{\I_1}+\underbrace{\int_{\Hb_{r_2}^{r_1}}|D_L \big(D_\Omega D_\Lb \phi \big)|^2}_{\I_2}.
\end{align*}
To bound $\I_1$, we first commute derivatives to derive
\begin{align*}
D_T D_\Omega D_\Lb \phi &= D_T \big([D_\Omega, D_\Lb] \phi\big) +[D_T,D_\Lb]D_\Omega  \phi+ D_\Lb D_T D_\Omega  \phi\\
&=\sqrt{-1}\mathcal{L}_T F_{\Omega\Lb}\phi +\sqrt{-1}F_{\Omega\Lb}D_T \phi + \sqrt{-1}F_{T\Lb}D_\Omega\phi +D_\Lb D_T D_\Omega  \phi.
\end{align*}
We therefore can bound that
\begin{equation*}
|D_T D_\Omega D_\Lb \phi| \leq r|\alphab^{(\mathbf{1})}||\phi| + r|\alphab||\phi^{(\mathbf{1})}|+r|\rho| |\slashed{D}\phi|+|D_\Lb\phi^{(\mathbf{2})}|,
\end{equation*}
where the discrepancy indices of the $(\mathbf{1})$ and $(\mathbf{2})$ are all equal to $-1$ and we note that $\alpha$, $\alphab$ and $\rho$ are the curvature components for the full Maxwell field $F$. Therefore, we can split $\I_1$ into four terms:
\begin{align*}
\I_1 &\leq \int_{\Hb_{r_2}^{r_1}}   r^2|\alphab^{(\mathbf{1})}|^2|\phi|^2 + \int_{\Hb_{r_2}^{r_1}} r^2|\alphab|^2|\phi^{(\mathbf{1})}|^2+\int_{\Hb_{r_2}^{r_1}} r^2|\rho|^2 |\slashed{D}\phi|^2+\int_{\Hb_{r_2}^{r_1}} |D_\Lb\phi^{(\mathbf{2})}|^2.
\end{align*}
Recall that the full Maxwell field $F$ splits into the chargeless part $\Fc$ which has been bounded in Proposition \ref{Proposition pointwise decay of Maxwell} and the charge part $F[q_0]$ satisfying the trivial bound \eqref{eq:bd4Fq}. Since $F[q_0]$ is stationary, we note that $\alphab^{(\mathbf{1})}=\alphabc^{(\mathbf{1})}$. Therefore we can use \eqref{pointwise bound on phi} to bound $\phi$ in the first term, use $|\alphab|\lesssim r^{-1}u_+^{-2}$ for the second term, use $|\rho|\lesssim r^{-2}$ in the third term and the bootstrap assumption $(\mathbf{B})$ to bound the last term. In particular we can show that
\begin{align*}
\I_1 & \lesssim \int_{\Hb_{r_2}^{r_1}}   |\alphabc^{(\mathbf{1})}|^2r_1^{-5-4\varepsilon_0} +  u_+^{-4}|\phi^{(\mathbf{1})}|^2+ r^{-2} |\slashed{D}\phi|^2+ |D_\Lb\phi^{(\mathbf{2})}|^2 \\
&\lesssim {\epsilonc} r_1^{-8-2\varepsilon_0}+r_1^{-4}\int_{\Hb_{r_2}^{r_1}}   |\phi^{(\mathbf{1})}|^2.
\end{align*}
Since $\phi^{(\mathbf{1})}=D_T\phi$, according to \eqref{Sobolev trace estimates on outgoing null hypersurfaces}, for $k\leq 1$ we have
\begin{equation}\label{L2 of DT DOmega phi on sphere}
\begin{split}
\|D_T D_\Omega^k\phi\|^2_{L^2(\S_{r_1}^{r_2})}&\lesssim \|D_T D_\Omega^k\phi\|^2_{L^2(\S_{r_1}^{r_1})}+\frac{1}{r_1}\int_{\H_{r_1}} |D_L D_T D_\Omega^k\psi|^2\stackrel{(\mathbf{B})}{\lesssim}  {\epsilonc} r_1^{-7-2\varepsilon_0}.
\end{split}
\end{equation}
We now use the case $k=0$ to conclude that
\begin{align*}
\int_{\Hb_{r_2}^{r_1}}   |\phi^{(\mathbf{1})}|^2&\lesssim \int_{-\frac{r_2}{2}}^{-\frac{r_1}{2}}\big(\int_{\S_{-2u}^{r_2}} \epsilonc u_+^{-4}|D_T\phi|^2\big)  du \lesssim \epsilonc \int_{-\frac{r_2}{2}}^{-\frac{r_1}{2}} \epsilonc u_+^{-11-2\varepsilon_0 } du \lesssim {\epsilonc}^{2} r_1^{-10-2\varepsilon_0}.
\end{align*}
Here we keep in mind that $u_+=1+\frac{1}{2}r_1$. In particular we derive that
\begin{align*}
\I_{1} &\lesssim {\epsilonc}r_1^{-8-2\varepsilon_0}.
\end{align*}
Now we turn to the estimate of $\I_2$. By using the null equations for $\phi$, we first can write that
\begin{align*}
D_L D_\Omega D_\Lb \phi &= D_L \big([D_\Omega, D_\Lb] \phi\big) +r^{-1}D_L D_\Lb \big(rD_\Omega  \phi\big)+r^{-1}(D_L D_\Omega \phi-D_\Lb D_\Omega \phi)\\
&\stackrel{\eqref{scalar box null}}{=}\sqrt{-1}D_L\big(F_{\Omega\Lb}\phi \big) -\Box_A D_\Omega\phi +\slashed{D}^2 D_\Omega\phi  - \sqrt{-1} \rho\cdot D_\Omega\phi+\frac{1}{r}(D_L D_\Omega \phi-D_\Lb D_\Omega \phi)\\
&=\sqrt{-1}\mathcal{L}_L F_{\Omega\Lb}\cdot\phi -Q(\phi,F;\Omega)+ \Big(2\sqrt{-1} F_{\Omega\Lb}D_{T} \phi+\frac{2}{r}D_T D_\Omega \phi\Big)\\
&\ \ \ \ \ \ +\Big(\slashed{D}^2\big(D_\Omega\phi \big) - \sqrt{-1} \rho\cdot\big(D_\Omega\phi\big)-\frac{2}{r}D_\Lb D_\Omega \phi-\sqrt{-1} F_{\Omega\Lb} D_\Lb \phi\Big).
\end{align*}
For the integral of the last term, we use the pointwise bounds:
\[
|\rho|\lesssim r^{-2},\quad |F_{\Omega \Lb}|=r|\alphab | \lesssim r^{-2}+ \sqrt{\epsilonc} u_+^{-3-\varepsilon_0}\lesssim u_+^{-2}.
\]
We therefore can bound that
\begin{align*}
  &\int_{\Hb_{r_2}^{r_1}}|\slashed{D}^2\big(D_\Omega\phi \big) - \sqrt{-1} \rho\cdot\big(D_\Omega\phi\big)-\frac{2}{r}D_\Lb D_\Omega \phi-\sqrt{-1} F_{\Omega\Lb} D_\Lb \phi|^2\\
  &\lesssim \int_{\Hb_{r_2}^{r_1}}r^{-2}|\slashed{D} D_\Omega^2\phi |^2+r^{-4}| D_\Omega\phi|^2+r^{-2}|D_\Lb D_\Omega \phi|^2+u_+^{-4}|D_\Lb \phi|^2\\
  &\lesssim {\epsilonc}r_1^{-8-2\varepsilon_0}.
\end{align*}
For the third term in the previous identity, by using the above estimate \eqref{L2 of DT DOmega phi on sphere}, we can show that
\begin{align*}
  &\int_{\Hb_{r_2}^{r_1}} |2\sqrt{-1} F_{\Omega\Lb}D_{T} \phi+\frac{2}{r}D_T D_\Omega \phi|^2 \lesssim \int_{\Hb_{r_2}^{r_1}}r^{-2}|D_T D_\Omega \phi|^2+u_+^{-4}|D_T \phi|^2 \lesssim {\epsilonc}r_1^{-8-2\varepsilon_0}.
\end{align*}
For the first term $\sqrt{-1}\mathcal{L}_L F_{\Omega\Lb}\cdot\phi$, we use the null equation \eqref{Maxwell null commuted} to show that
\begin{align*}
\int_{\Hb_{r_2}^{r_1}} |\sqrt{-1}\mathcal{L}_L F_{\Omega\Lb}\cdot\phi|^2 &\lesssim \int_{\Hb_{r_2}^{r_1}} \big(|\rho^{(\mathbf{1})}|^2+|\sigma^{(\mathbf{1})}|^2 +r^2|\slashed{J}|^2+r^2|\alphab|^2\big)|\phi|^2.
\end{align*}
Now recall that $|\slashed{J}|=|\phi||\slashed{D}\phi|$ and we have the bounds $ |\cL_{\Omega}F[q_0]|\lesssim r^{-3}$.
Then by using the bootstrap assumptions on $\mathring{F}$ as well as the pointwise bound for $\phi$, we indeed can show that
\begin{align*}
  \int_{\Hb_{r_2}^{r_1}} |\sqrt{-1}\mathcal{L}_L F_{\Omega\Lb}\cdot\phi|^2 &\lesssim \int_{\Hb_{r_2}^{r_1}} \big(|\rhoc^{(\mathbf{1})}|^2+|\sigmac^{(\mathbf{1})}|^2 +u_+^{-5}|\slashed{D}\phi|^2+r^2|\alphabc|^2+r^{-4}\big)|\phi|^2 \lesssim {\epsilonc}r_1^{-8-2\varepsilon_0}.
\end{align*}
Finally for the quadratic term $Q(\phi, F; \Omega)$, we use the bound \eqref{null form Omega} in the proof for Proposition \ref{lemma null form} to show that
\begin{align*}
\int_{\Hb_{r_2}^{r_1}} |Q(\phi,F;\Omega)|^2 &\lesssim \int_{\Hb_{r_2}^{r_1}} |D_L\psi|^2|\alphab|^2+|D_\Lb \psi|^2|\alpha|^2+|\sigma|^2|\phi|^2+r^2 |\slashed{J}|^2|\phi|^2+|\sigma|^2|\slashed {D} \psi|^2\\
&\lesssim \int_{\Hb_{r_2}^{r_1}} |D_L\psi|^2r^{-2}u_+^{-4}+|D_\Lb \psi|^2 r^{-6}+r^{-4}u_+^{-2}(|\phi|^2+r^2|\slashed{D} \phi|^2)+u_+^{-10} |\slashed{D}\phi|^2 \\
&\lesssim \int_{\Hb_{r_2}^{r_1}} |D_T\phi|^2 u_+^{-4}+|D_\Lb \phi|^2 u_+^{-4}+r^{-4}u_+^{-2}|\phi|^2+u_+^{-4}|\slashed{D} \phi|^2\\
& \lesssim {\epsilonc}r_1^{-8-2\varepsilon_0}.
\end{align*}
Here we have used the fact that $L=2T-\Lb$ to bound $D_L\psi$ and estimate \eqref{L2 of DT DOmega phi on sphere} to bound the integral of $\phi$ as well as $D_T\phi$.
Combining the above estimate, we have shown that
\begin{equation}\label{eq 1}
\int_{\Hb_{r_2}^{r_1}} |D_\Lb \big(D_\Omega D_\Lb \phi \big)|^2 \lesssim {\epsilonc} r_1^{-8-2\varepsilon_0}.
\end{equation}
The next object is to derive estimate for $\int_{\Hb_{r_2}^{r_1}} |D_\Omega \big(D_\Omega D_\Lb \phi \big)|^2$. First for $\Omega$, $\Omega'$ being angular momentum vector fields,  recall the following commutation formula:
\begin{align*}
D_\Omega \big(D_{\Omega'} D_\Lb \phi \big)
&=\sqrt{-1}\big(\mathcal{L}_\Omega F_{\Omega' \Lb} \phi + F([\Omega,\Omega'], \Lb) \phi+ F_{\Omega' \Lb} D_\Omega \phi+F_{\Omega \Lb} D_{\Omega'}\phi\big) +D_\Lb D_\Omega D_{\Omega'}  \phi
\end{align*}
For the first four terms, we can bound the full Maxwell field by the pointwise bound according to Proposition \ref{Proposition pointwise decay of Maxwell} together with the property of the charge 2-form $F[q_0]$. More precisely we can show that
\begin{align*}
 & \int_{\Hb_{r_2}^{r_1}}|\sqrt{-1}\big(\mathcal{L}_\Omega F_{\Omega' \Lb} \phi + F([\Omega,\Omega'], \Lb) \phi+ F_{\Omega' \Lb} D_\Omega \phi+F_{\Omega \Lb} D_{\Omega'}\phi\big)|^2\\
  &\lesssim \int_{\Hb_{r_2}^{r_1}}(|\mathcal{L}_\Omega \alphabc|^2 +|\alphabc|^2)u_+^{-5} +r^{-4}|\phi|^2+ u_+^{-4}r^2|\slashed{D}\phi|^2 \lesssim {\epsilonc}r_1^{-8-2\varepsilon_0}.
\end{align*}
Then by using the ansatz $\mathbf{(B)}$, we can derive that
\begin{equation*}
\int_{\Hb_{r_2}^{r_1}} |D_\Omega \big(D_\Omega D_\Lb \phi \big)|^2\lesssim {\epsilonc} r_1^{-6-2\varepsilon_0}.
\end{equation*}
Similarly, we also have
\begin{equation*}
\int_{\Hb_{r_2}^{r_1}} |D_\Omega D_\Lb \phi |^2\lesssim  {\epsilonc}r_1^{-6-4\varepsilon_0}.
\end{equation*}
Then using the Sobolev inequality \eqref{Sobolve for scalar filed on incoming null hypersurfaces}, we derive that
\begin{equation*}
\|D_\Omega D_\Lb\phi\|_{L^4(\S_{r_1}^{r_2})} \lesssim r_2^{-\frac{1}{2}}  \sqrt{{\epsilonc}} r_1^{-3-\varepsilon_0}.
\end{equation*}
We then repeat the same argument for $D_\Lb\phi$ to derive
\begin{equation*}
\| D_\Lb\phi\|_{L^4(\S_{r_1}^{r_2})} \lesssim r_2^{-\frac{1}{2}}  \sqrt{{\epsilonc}} r_1^{-3-2\varepsilon_0}.
\end{equation*}
 Finally, by virtue of \eqref{Sobolev on sphere} and the fact that $u_+=1+\frac{1}{2}r_1$, we obtain that
\begin{equation*}
\|D_\Lb\phi\|_{L^\infty(\S_{r_1}^{r_2})} \lesssim  \sqrt{{\epsilonc}} r_1^{-1} u_+^{-3-\varepsilon_0}.
\end{equation*}

{\bf Step 3. $L^\infty$ estimate of $\slashed{D}\phi$.} By the bootstrap ansatz $(\mathbf{B})$, we have
\begin{align*}
\int_{\Hb_{r_2}^{r_1}} |D_\Lb D_{\Omega'}\big( D_\Omega \phi \big)|^2 \lesssim  {\epsilonc}r_1^{-6-2\varepsilon_0},
\end{align*}
We now use the $r^p$-weighted energy estimate with $p=2$ of the bootstrap assumption $(\mathbf{B})$ to show that
\begin{align*}
\int_{\Hb_{r_2}^{r_1}} |D_{\Omega^{''}} \big(D_{\Omega'} D_\Omega \phi \big)|^2 \lesssim \int_{\Hb_{r_2}^{r_1}} |\slashed{\nabla} \big(D_{\Omega'} D_\Omega \psi \big)|^2 &\lesssim {\epsilonc}r_1^{-4-2\varepsilon_0},
\end{align*}
Therefore by using the Sobolev embedding, we have
\begin{equation*}
\|D_{\Omega'} D_\Omega\phi\|_{L^4(\S_{r_1}^{r_2})} \lesssim r_2^{-\frac{1}{2}} \sqrt{{\epsilonc}} r_1^{-2-\varepsilon_0}.
\end{equation*}
Similarly, we can also obtain
\begin{equation*}
\| D_\Omega\phi\|_{L^4(\S_{r_1}^{r_2})} \lesssim r_2^{-\frac{1}{2}} \sqrt{{\epsilonc}} r_1^{-2-2\varepsilon_0}.
\end{equation*}
Therefore, \eqref{Sobolev on sphere} implies that, for all angular momentum vector field $\Omega$, we have
\begin{equation*}
\|D_\Omega\phi\|_{L^\infty(\S_{r_1}^{r_2})} \lesssim  \sqrt{{\epsilonc}} r_2^{-1} u_+^{-2-\varepsilon_0}.
\end{equation*}
Considering that $|D_{\Omega}\phi|=r|\slashed{D}\phi|$, the above estimate implies that
\begin{equation*}
\|\slashed{D}\phi\|_{L^\infty(\S_{r_1}^{r_2})} \lesssim  \sqrt{{\epsilonc}} r^{-2} u_+^{-2-\varepsilon_0}.
\end{equation*}
Here note that on the sphere $\S_{r_1}^{r_2}$ it holds the relation $r=\frac{r_1+r_2}{2}$.

{\bf Step 4. $L^\infty$ estimate of $D_L(r \phi)$.} The idea is to use the highest weight commutator $K=v^2L+u^2\Lb$. According to the bootstrap ansatz, we have
\begin{align*}
\sum\limits_{k\leq 1}\int_{\Hb_{r_2}^{r_1}} r_1^2|D_\Lb D_{\Omega}^k\big( \Dt_K \phi \big)|^2+r^2|\slashed{D}D_{\Omega}^k \big( \Dt_K \phi \big)|^2 \lesssim  {\epsilonc}r_1^{-2-2\varepsilon_0}.
\end{align*}
Here we may note that $D_{\Omega}=\Dt_{\Omega}$. In particular we conclude that
\begin{align*}
\int_{\Hb_{r_2}^{r_1}} |D_\Omega D_{\Omega'}\big( \Dt_K \phi \big)|^2+|D_\Omega \big( \Dt_K \phi \big)|^2 \lesssim \sqrt{\epsilonc} r_1^{-2-2\varepsilon_0}.
\end{align*}
Therefore the Sobolev embedding implies that
\begin{align*}
|\Dt_K \phi |&\lesssim  \sqrt{\epsilonc}r^{-1}u_+^{-1-\varepsilon_0}.
\end{align*}
On the other hand, we have
\begin{align*}
v^2 D_{L}(r\phi)=D_{K}(r\phi)-u^2 D_{\Lb}(r\phi)=r \Dt_{K}\phi-ru^2 D_{\Lb}\phi+u^2 \phi.
\end{align*}
Then by using the bounds for $\phi$ and $D_{\Lb}\phi$, we derive that
\begin{equation*}
v^2|D_L(r\phi)| \lesssim \sqrt{{\epsilonc}} (1+ |u|)^{-1-\varepsilon_0}.
\end{equation*}
This completes the proof.
\end{proof}

\section{The analysis in the exterior region 2: energy estimates}
\subsection{Energy estimates on Maxwell field}
For an multi-index $\k$ with $1\leq |\k| \leq 2$, we can take $G=\mathcal{L}_{Z}^{\k}\Fc$ and ${f}=0$ in \eqref{classical energy inequality} and \eqref{r wieghted} to deduce:
\begin{equation*}
\begin{split}
\mathcal{E}^{(\k)}(\Fc;r_1) &\leq \E[\mathcal{L}_Z^\k \Fc](\B_{r_1})+\int_{\D_{r_1}}r^{-2}\big|\Jk_\nu\cdot\mathcal{L}_{Z}^{\k}\Fc_0{}^{\nu}\big|\\
&\leq {\epsilonc}r_1^{-6+2\xi(\k)-8\varepsilon_0}+C\int_{\D_{r_1}}\underbrace{\frac{ |\Jk_L||\rho^{(\k)}|}{r^2}}_{\I_1}+\underbrace{\frac{ |\Jk_\Lb||\rho^{(\k)}|}{r^2}}_{\I_2}+\underbrace{\frac{|\sJk||\alpha^{(\k)}|}{r^2}}_{\I_3}+\underbrace{\frac{|\sJk||\alphab^{(\k)}|}{r^2}}_{\I_4},
\end{split}
\end{equation*}
and
\begin{equation*}
\begin{split}
\mathcal{E}^{(\k)}(\Fc;p=2;r_1) &\leq \int_{\B_{r_1}}r^{2}\big(|\alpha^{(\k)}|^2+|\rho^{(\k)}|^2+|\sigma^{(\k)}|^2\big)+ \int_{\D_{r_1}} \big|\Jk_\nu\cdot\mathcal{L}_{Z}^{\k}\Fc_L{}^{\nu}\big|,\\
&\leq {\epsilonc}r_1^{-4+2\xi(\k)-8\varepsilon_0}+C \int_{\D_{r_1}} \underbrace{|\Jk_L||\rho^{(\k)}|}_{\I_5}+\underbrace{|\sJk||\alpha^{(\k)}|}_{\I_6},
\end{split}
\end{equation*}
where $C$ is a universal constant. In this section, the constant $C$ may change but they all denote universal constants. We now bound the $\I_i$'s one by one.

For $\I_1$ and $\I_5$,  we have
\begin{align*}
\int_{\D_{r_1}}\I_5 &\lesssim \Big(\int_{\D_{r_1}} |\Jk_L|^2 \Big)^{\frac{1}{2}} \Big(\int_{\D_{r_1}} |\rho^{(\k)}|^2\Big)^{\frac{1}{2}}\\
&\lesssim \Big( \int_{r\geq r_1}  \big(\int_{\H_r} |\Jk_L|^2\big) dr\Big)^{\frac{1}{2}} \Big( \int_{r_1}^{\infty} \big(\int_{\H_{r}}|\rho^{(\k)}|^2\big) dr\Big)^{\frac{1}{2}}\\
&\lesssim \epsilonc r_1^{-\frac{5}{2}+\xi(\k)-2\varepsilon_0}\cdot {\epsilonc}^{\frac{1}{2}} r_1^{-\frac{5}{2}+\xi(\k)-(3-|\k|)\varepsilon_0} \lesssim {\epsilonc}^
{\frac{3}{2}}r_1^{-4+2\xi(\k)-(6-2|\k|)\varepsilon_0}.
\end{align*}	
The last step follows from the bootstrap assumption $\mathbf{(C)}$ as well as the bootstrap assumption $\mathbf{(B)}$. Similarly,
\begin{align*}
\int_{\D_{r_1}}\I_1
\lesssim {\epsilonc}^{\frac{3}{2}}r_1^{-6+2\xi(\k)-(6-2|\k|)\varepsilon_0}.
\end{align*}
For $\I_2$, we have
\begin{align*}
\int_{\D_{r_1}}\I_2&\les \Big(\int_{\D_{r_1}}\frac{|\Jk_\Lb|^2}{r^{\frac{19}{4}}}\Big)^{\frac{1}{2}}\Big(\int_{\D_{r_1}} r^{\frac{3}{4}}|\rho^{(\k)}|^2\Big)^{\frac{1}{2}}\\
&= \Big(\int_{\frac{r_1}{2}}^{\infty} \frac{1}{v^{\frac{5}{4}}}\big(\int_{\Hb_{2v}^{r_1}}\frac{|\Jk_\Lb|^2}{r^{\frac{7}{2}}}\big) dv \Big)^{\frac{1}{2}} \Big( \int_{\frac{r_1}{2}}^{\infty} \frac{1}{v^{\frac{5}{4}}}\big(\int_{\Hb_{2v}^{r_1}}r^2|\rho^{(\k)}|^2\big) dv\Big)^{\frac{1}{2}}\\
& \lesssim \epsilonc r_1^{-\frac{39}{8}+\xi(\k)-2\varepsilon_0}\cdot {\epsilonc}^{\frac{1}{2}} r_1^{-\frac{17}{8}+\xi(\k)-(3-|\k|)\varepsilon_0} \lesssim {\epsilonc}^
{\frac{3}{2}}r_1^{-6+2\xi(\k)-(6-2|\k|)\varepsilon_0}.
\end{align*}
We remark that in the last step we have used the bootstrap assumption $\mathbf{(B)}$ since $\int_{\Hb^{r_1}_{r}} r^2|\rho^{(\k)}|^2$ appears in the $r^p$-weighted energy. Another key point is that $v^{-\frac{5}{4}}$ is integrable on $[\frac{r_1}{2},\infty)$.

For $\I_3$ and $\I_6$, we have
\begin{align*}
\int_{\D_{r_1}}\I_6&\les \Big(\int_{\D_{r_1}}\frac{|\sJk |^2}{r^2}\Big)^{\frac{1}{2}}\Big(\int_{\D_{r_1}}r^2  |\alpha^{(\k)}|^2\Big)^{\frac{1}{2}}\\
&=\Big(\int_{r_1}^\infty\big(\int_{\H_{r_2}}\frac{|\sJk|^2}{r^2}\big)dr_2\Big)^{\frac{1}{2}} \Big( \int_{r_1}^{\infty}  \big(\int_{\H_{r}}r^2|\alpha^{(\k)}|^2\big) dr\Big)^{\frac{1}{2}} \\
&\lesssim \epsilonc r_1^{-\frac{7}{2}+\xi(\k)-2\varepsilon_0}\cdot {\epsilonc}^{\frac{1}{2}} r_1^{-\frac{3}{2}+\xi(\k)-(3-|\k|)\varepsilon_0} \lesssim {\epsilonc}^
{\frac{3}{2}}r_1^{-6+2\xi(\k)-(4-2|\k|)\varepsilon_0}.
\end{align*}
Similarly, we have
\begin{align*}
\int_{\D_{r_1}}\I_3
\lesssim {\epsilonc}^{\frac{3}{2}}r_1^{-6+2\xi(\k)-(6-2|\k|)\varepsilon_0}.
\end{align*}
For $\I_4$, we have
\begin{align*}
\int_{\D_{r_3}}\I_4&\les \Big(\int_{\D_{r_1}} \frac{ |\sJk|^2}{r^2}\Big)^{\frac{1}{2}} \Big(\int_{\D_{r_1}}\frac{|\alphab^{(\k)}|^2}{r^2}\Big)^{\frac{1}{2}}\\
&\les\Big( \int_{r_1}^\infty\big( \int_{\H_{r_2}}\frac{ |\sJk|^2}{r^2}\big)dr_2\Big)^{\frac{1}{2}} \Big( \int_{\frac{r_1}{2}}^\infty \frac{1}{v^2}\big(\int_{\Hb^{r_1}_{2v}} |\alphab^{(\k)}|^2\big) dv \Big)^{\frac{1}{2}}\\
&\lesssim \epsilonc r_1^{-\frac{7}{2}+\xi(\k)-2\varepsilon_0}\cdot {\epsilonc}^{\frac{1}{2}} r_1^{-\frac{7}{2}+\xi(\k)-(3-|\k|)\varepsilon_0} \lesssim {\epsilonc}^
{\frac{3}{2}}r_1^{-6+2\xi(\k)-(6-2|\k|)\varepsilon_0}.
\end{align*}

As a conclusion and by our convention on the implicit constant, we derive that
\begin{equation*}
\begin{split}
\mathcal{E}^{(\k)}(\Fc;r_1)&\leq {\epsilonc}r_1^{-6+2\xi(\k)-(6-2|\k|)\varepsilon_0}\big(1+C{\epsilonc}^{\frac{1}{2}}\big),\\
\mathcal{E}^{(\k)}(\Fc;p=2;r_1) &\leq {\epsilonc}r_1^{-6+2\xi(\k)-(6-2|\k|)\varepsilon_0}\big(1+C{\epsilonc}^{\frac{1}{2}}\big)
\end{split}
\end{equation*}
for some universal constant $C$.

For sufficiently small $\epsilonc$, we then has closed the bootstrap argument for the Maxwell fields in $\mathbf{(B)}$:
\begin{equation}\label{improved bootstrap maxwell}
\begin{split}
\mathcal{E}^{(\k)}(\Fc;r_1) & \leq 2 {\epsilonc}r_1^{-6+2\xi(\k)-(6-2|\k|)\varepsilon_0},\ \
\mathcal{E}^{(\k)}(\Fc;p=2;r_1) \leq 2{\epsilonc}r_1^{-6+2\xi(\k)-(6-2|\k|)\varepsilon_0}.
\end{split}
\end{equation}

\subsection{Energy estimates on scalar field}
For all multi-index $\k$ such that $1\leq |\k| \leq 2$, we take ${f}=\Dt_{Z}^{\k}\phi$ and $G=0$ in \eqref{classical energy inequality} and \eqref{r wieghted}. Let $\psi^{(\k)}=r\phi^{(\k)}$. We deduce the following energy estimates
\begin{equation}
\label{eq:EE:scal:00}
\begin{split}
\mathcal{E}^{(\k)}(\phi;r_1) &\leq \E[\phi^{(\k)}](\B_{r_1})+\int_{\D_{r_1}}\big| \Box_A\phi^{(\k)}  \cdot D_{\partial_t} \phi^{(\k)} \big|+|F_{0\mu} {J}[\phi^{(\k)}]^\mu|\\
&\leq  {\epsilonc} r_1^{-6+2\xi(\k)-8\varepsilon_0} +\underbrace{\int_{\D_{r_1}}\big| \Box_A\phi^{(\k)}
|\big(| D_L \phi^{(\k)}|+|D_\Lb\phi^{(\k)}|\big)}_{\R_1} \\
&\ \ \ \ \ \ \  \ \ \ \ \ \ \ \ \ \ \ \ \ \ \ \ \ +\underbrace{\int_{\D_{r_1}}(|\alpha|+|\alphab|)|\slashed{D}\phi^{(\k)}||\phi^{(\k)}|}_{\SS_1}+\underbrace{\int_{\D_{r_1}}|\rho|\big(|D_L\phi^{(\k)}|+|D_\Lb\phi^{(\k)}|\big)|\phi^{(\k)}|}_{\T_1}
\end{split}
\end{equation}
as well as the $r$-weighted energy estimates
\begin{equation}
\label{eq:EE:scal:000}
\begin{split}
\mathcal{E}^{(\k)}(\phi;p=2;r_1) &\leq \int_{\B_{r_1}}|D_L \psi^{(\k)}|^2+|\slashed{D} \psi^{(\k)}|^2+ \int_{\D_{r_1}} r\big|\Box_A\phi^{(\k)} \cdot D_L \psi^{(\k)}\big|+r^2 \big|F_{L\mu} {J}[\phi^{(\k)}]^\mu\big|\\
&\leq {\epsilonc}r_1^{-4+2\xi(\k)-8\varepsilon_0}+ \underbrace{\int_{\D_{r_1}} r\big|\Box_A\phi^{(\k)}\big||D_L \psi^{(\k)}|}_{\R_2}+\underbrace{\int_{\D_{r_1}} r^2|\alpha||\slashed{D}\phi^{(\k)}||\phi^{(\k)}|}_{\SS_2}\\
&\qquad\qquad\qquad\qquad+\underbrace{\int_{\D_{r_1}}|\rho||D_L\psi^{(\k)}||\psi^{(\k)}|}_{\T_2}.
\end{split}
\end{equation}
We remark that for the term $\T_2$, we have used the following structure of current term:
\begin{equation*}
r^2{J}[\phi^{(\k)}]= r^2 \Im(\phi^{(\k)}\cdot \overline{D \phi^{(\k)}})= \Im(\psi^{(\k)}\cdot \overline{D \psi^{(\k)}})={J}[\psi^{(\k)}].
\end{equation*}
This will be crucial for the estimate of $\T_2$. We first bound the $\SS_i$'s  which rely on the following lemma:
\begin{lemma}\label{lemma technical on phik}
Under the bootstrap ansatz, for $\gamma_2\geq 0$, $\gamma_1 > 1$, we have
\begin{equation*}
\int_{\D_{r_1}}\frac{|\phi^{(\k)}|^2}{r^{\gamma_1}|u|^{\gamma_2}} \lesssim {\epsilonc} r_1^{-3-\gamma_1-\gamma_2+2\xi(\k)-(6-2|\k|)\varepsilon_0}.
\end{equation*}
\end{lemma}
\begin{proof}
Let $\S_{u,v}$ be the intersection of $\H_{u}$ and $\Hb_{v}$. By \eqref{Sobolev trace estimates on outgoing null hypersurfaces}, we then have
\begin{align*}
\int_{\D_{r_1}}\frac{|\phi^{(\k)}|^2}{r^{\gamma_1}|u|^{\gamma_2}} &=\int_{u}\int_{v}\frac{\int_{\S_{u,v}}|\phi^{(\k)}|^2}{r^{\gamma_1}|u|^{\gamma_2}}\lesssim \int_{u}\int_{v}\frac{\int_{\S_{u,u}}|\phi^{(\k)}|^2+|u|^{-1}\int_{\H_u}|D_L\psi^{(\k)}|^2}{r^{\gamma_1}|u|^{\gamma_2}}\\
&\lesssim \int_{u}\frac{\int_{\S_{u,u}}|\phi^{(\k)}|^2}{|u|^{\gamma_1+\gamma_2-1}}+\int_{u}\frac{\int_{\H_u}|D_L\psi^{(\k)}|^2}{|u|^{\gamma_1+\gamma_2}}.
\end{align*}
The first term is from the initial data and it is bounded by  ${\epsilonc} r_1^{-3-\gamma_1-\gamma_2+2\xi(\k)-8\varepsilon_0}$. We can control the second term by the bootstrap ansatz and it is bounded by $C{\epsilonc} r_1^{-3-\gamma_1-\gamma_2+2\xi(\k)-(6-2|\k|)\varepsilon_0}$. This completes the proof.
\end{proof}
For $\mathbf{S}_1$, according to Proposition \ref{Proposition pointwise decay of Maxwell} and the decay properties of the charge part $\alpha(F[q_0])$, $\alphab(F[q_0])$, we in particular have the following bounds
\[
|\alpha|\les \sqrt{{\epsilonc}} r^{-3} u_+^{-1-\varepsilon_0 }+r^{-3} \les r^{-3}, \quad |\alphab| \les \sqrt{{\epsilonc}} r^{-1}u_+^{-3-\varepsilon_0}+r^{-3}\les r^{-1}u_+^{-2}.
\]
We have used the fact that $\epsilonc$ is sufficiently small. Therefore we can show that
\begin{align*}
\SS_1&\lesssim  \int_{\D_{r_1}}\frac{|\slashed{D}\phi^{(\k)}||\phi^{(\k)}|}{r |u|^{2}}
\lesssim \Big(\int_{\D_{r_1}}\frac{|\slashed{D}\phi^{(\k)}|^2}{|u|}\Big)^{\frac{1}{2}} \Big(\int_{D_{r_1}}\frac{|\phi^{(\k)}|^2}{r^{2}|u|^{3}}\Big)^{\frac{1}{2}}\\
&= \Big( \int_{|u| \geq \frac{r_1}{2}} \frac{\int_{\H_u}|\slashed{D}\phi^{(\k)}|^2}{|u|} du\Big)^\frac{1}{2}\Big(\int_{D_{r_1}}\frac{|\phi^{(\k)}|^2}{r^{2}|u|^{3}}\Big)^{\frac{1}{2}}.
\end{align*}
We use the bootstrap ansatz to bound the first term and use Lemma \ref{lemma technical on phik} to bound the second term. Therefore, we obtain
\begin{equation*}
\SS_1 \lesssim {\epsilonc} r_1^{-6.5+2\xi(\k)-(6-2|\k|)\varepsilon_0}.
\end{equation*}
We can also derive in the same manner that
\begin{equation*}
\SS_2 \lesssim  {\epsilonc} r_1^{-4.5+2\xi(\k)-(6-2|\k|)\varepsilon_0}.
\end{equation*}
By our convention the implicit constant is independent of $R_*$. Since $r_1\geq R_*$, by choosing $R_*$ sufficiently large, we derive the following estimates
\begin{equation}\label{energy identities to be bounded}
\begin{split}
 \mathcal{E}^{(\k)}(\phi;r_1) \leq \frac{5}{4}\epsilonc r_1^{-6+2\xi(\k)-8\varepsilon_0}+\mathbf{R}_1 
 +\mathbf{T}_1,\\
\mathcal{E}^{(\k)}(\phi;p=2;r_1)  \leq \frac{5}{4}\epsilonc r_1^{-4+2\xi(\k)-8\varepsilon_0}+ \mathbf{R}_2 
 +\mathbf{T}_2
\end{split}
\end{equation}
with $\mathbf{R}_i$, $\mathbf{T}_i$ defined in \eqref{eq:EE:scal:00} and \eqref{eq:EE:scal:000}.

\subsubsection{Energy estimates on one derivatives of the scalar field}\label{section on 1 derivatives on scalar}
We consider the case where $|\k|=1$. The multi-index $\k$ then represents a vector field $Z\in \Gamma$. In view of \eqref{null form estimate} and the pointwise bounds in Proposition \ref{Proposition pointwise decay of Maxwell} and Proposition \ref{Proposition pointwise decay of scalar field}, we have
\begin{equation*}
|u|^{-\xi(Z)}|Q(\phi,F;Z)| \lesssim \frac{1}{r|u|}|D_L\psi|+\frac{1}{r^2|u|}|\slashed{D}\psi| +\frac{|u|}{r^3}|D_\Lb\psi|+\frac{1}{r^2}|\phi|.
\end{equation*}
Thus, we have
\begin{equation}\label{the first pointwise bound on Q}
\begin{split}
r^2|\Box_A\phi^{({\bf 1})}|^2 &\lesssim r^2 |Q(\phi,F;Z)|^2\\
&\lesssim |u|^{2\xi({\bf 1})-2}|D_L\psi|^2+|u|^{2\xi({\bf 1})-2}|\slashed{D}\phi|^2+\frac{|u|^{2\xi({\bf 1})+2}}{r^2}|D_\Lb\phi|^2+\frac{|u|^{2\xi({\bf 1})}}{r^2}|\phi|^2.
\end{split}
\end{equation}
Since $r^{-2}\les |u|^{-2}$, according to the bounds on the zeroth order energy estimates, we have
\begin{equation*}
\begin{split}
\int_{\D_{r_1}}  |u|^{2\xi({\bf 1})-2}|D_L\psi|^2  &\leq \int_{u}{|u|^{2\xi({\bf 1})-2}}\Big(\int_{\H_u}|D_L \psi|^2\Big)du\lesssim {\epsilonc} r_1^{2\xi({\bf 1})-5-6\varepsilon_0},
\end{split}
\end{equation*}
\begin{align*}
\int_{\D_{r_1}} |u|^{2\xi({\bf 1})-2}|\slashed{D}\phi|^2 &\lesssim \int_{u}{|u|^{2\xi({\bf 1})-4-2\varepsilon_0}}\Big(\int_{\H_u}|\slashed{D}\phi|^2\Big)du \lesssim {{\epsilonc} r_1^{2\xi({\bf 1})-7-6\varepsilon_0}},
\end{align*}
and
\begin{align*}
\int_{\D_{r_1}}\frac{|u|^{2\xi({\bf 1})+2}}{r^2}|D_\Lb\phi|^2 &\lesssim \int_{v}{|u|^{2\xi({\bf 1})}}|v|^{-2}\Big(\int_{\Hb_v}|D_\Lb\phi|^2\Big)du \lesssim {\epsilonc} r_1^{2\xi({\bf 1})-5-6\varepsilon_0}.
\end{align*}
By Lemma \ref{lemma technical on phik}, we also have
\begin{align*}
\int_{\D_{r_1}} \frac{|u|^{2\xi({\bf 1})}}{r^2}|\phi|^2 &\lesssim {\epsilonc} r_1^{2\xi({\bf 1})-5-6\varepsilon_0}.
\end{align*}
Thus, we have
\begin{equation}\label{the first null form bound on spacetime slab}
\int_{\D_{r_1}} r^2|Q(\phi,F;Z)|^2 \lesssim r_1^{2\xi({Z})-5-6\varepsilon_0}.
\end{equation}
Let $(\mathbf{2})$ denotes two vector fields $Z_1$ and $Z_2$. If we replace $\phi$ by $\Dt_{Z_1}$ in the proof between \eqref{the first pointwise bound on Q} and \eqref{the first null form bound on spacetime slab}, we obtain
\begin{equation}\label{the first null form bound on spacetime slab with one more derivatives}
\int_{\D_{r_1}} r^2|Q(\Dt_{Z_1}\phi,F;Z_2)|^2 \lesssim r_1^{2\xi(\mathbf{2})-5-6\varepsilon_0}.
\end{equation}
Similarly, we have
\begin{align*}
&\int_{\D_{r_1}}r^{-2}\big(|D_L\phi^{({\bf 1})}|+|D_\Lb\phi^{({\bf 1})}| \big)^2 \\
&\lesssim \int_{|u|\geq r_1} u_+^{-2}\big(\int_{\H_u}|D_L \phi^{({\bf 1})}|^2\big)du+\int_{v}v^{-2}\Big(\int_{\Hb_v}|D_\Lb \phi^{({\bf 1})}|^2\Big)dv\lesssim {\epsilonc} r_1^{2\xi({\bf 1})-5-4\varepsilon_0}.
\end{align*}
Therefore, we can bound $\mathbf{R}_1$ as follows
\begin{align*}
\R_1 &\leq \Big(\int_{\D_{r_1}} r^2|\Box_A\phi^{({\bf 1})}|^2 \Big)^\frac{1}{2} \Big(\int_{\D_{r_1}} r^{-2}\big(|D_L\phi^{({\bf 1})}|+|D_\Lb\phi^{({\bf 1})}| \big)^2\big)\Big)^\frac{1}{2} \lesssim {\epsilonc} r_1^{-6+2\xi({\bf 1})-5\varepsilon_0}.
\end{align*}
One can also proceed exactly in the same manner to prove that
\begin{align*}
\R_2 &\lesssim \Big(\int_{\D_{r_1}}r^2 |\Box_A\phi^{({\bf 1})}|^2 \Big)^\frac{1}{2} \Big(\int_{\D_{r_1}} |D_L\psi^{({\bf 1})}|^2\Big)^\frac{1}{2}\lesssim {\epsilonc} r_1^{-4+2\xi({\bf 1})-5\varepsilon_0}
\end{align*}
by using the $r$-weighted energy estimates.
Therefore, for sufficiently large $R_*$, since $r_1\geq R_*$, we have
\begin{equation}\label{energy identities final}
\begin{split}
 \mathcal{E}^{{({\bf 1})}}(\phi;r_1) &\lesssim \epsilonc r_1^{-6+2\xi({{\bf 1}})-5\varepsilon_0}+\T_1 \\ 
\mathcal{E}^{({\bf 1})}(\phi;p=2;r_1) & \lesssim \epsilonc r_1^{-4+2\xi({{\bf 1}})-5\varepsilon_0}+\T_2 .
\end{split}
\end{equation}
At this stage, we need to first control $\T_2$ in the second equation. In view of the definition of $\mathcal{E}^{(\k)}(\phi;p=2;r_1)$ and the fact that $|\rho|\lesssim r^{-2}$, the second inequality gives
\begin{equation*}
\int_{\H_{r_1}} |D_L\psi^{({\bf 1})}|^2 \lesssim \epsilonc r_1^{-4+2\xi{({\bf 1})}-5\varepsilon_0}+ \int_{\D_{r_1}} \frac{|D_L\psi^{({\bf 1})}||\psi^{({\bf 1})}|}{r^2}.
\end{equation*}
When we apply Lemma \ref{lemma key} in this case, we change $\varepsilon_0$ to $\frac{1}{2}\varepsilon_0$. This leads to
\begin{equation*}
\int_{\H_{r_1}} |D_L\psi^{({\bf 1})}|^2 \lesssim \epsilonc r_1^{-4+2\xi{({\bf 1})}-4.5\varepsilon_0}.
\end{equation*}
The gain of $r^{-0.5\varepsilon_0}$ can be used to improve the estimates in Lemma \ref{lemma technical on phik}. This gives
\begin{equation}\label{improved est}
\int_{\D_{r_1}} \frac{|\psi^{({\bf 1})}|^2}{r^4} =\int_{\D_{r_1}} \frac{|\phi^{({\bf 1})}|^2}{r^2}  \lesssim {\epsilonc} r_1^{-5+2\xi({\bf 1})-4.5\varepsilon_0}.
\end{equation}
Hence,
\begin{equation*}
\T_2 \lesssim \Big(\int_{\D_{r_1}} |D_L\psi^{({\bf 1})}|^2\big)^\frac{1}{2} \Big(\int_{\D_{r_1}} \frac{|\psi^{({\bf 1})}|^2}{r^4}\Big)^\frac{1}{2}\lesssim \epsilonc r_1^{-4+2\xi{({\bf 1})}-4.5\varepsilon_0}.
\end{equation*}
This improved estimate \eqref{improved est} also allows us to bound $\T_1$ as follows:
\begin{align*}
\T_1 &\lesssim\big( \underbrace{\int_{\D_{r_1}} r^{-2}|D_L\phi^{({\bf 1})}|^2 +\int_{\D_{r_1}} r^{-2}|D_\Lb\phi^{({\bf 1})}|^2 }_{\lesssim  r_1^{-7+2\xi{({\bf 1})}-4\varepsilon_0} \text{by \bf{(B)}} }\big)^\frac{1}{2}\big(\underbrace{\int_{\D_{r_1}} \frac{|\phi^{({\bf 1})}|^2}{r^2}}_{\lesssim  r_1^{-5+2\xi{({\bf 1})}-4.5\varepsilon_0}}\Big)^\frac{1}{2}\\
&\lesssim \epsilonc r_1^{-6+2\xi{({\bf 1})}-4.25\varepsilon_0}.
\end{align*}
Thus, the estimate \eqref{energy identities final} implies
\begin{equation*}
\begin{split}
 \mathcal{E}^{{({\bf 1})}}(\phi;r_1) &\lesssim \epsilonc r_1^{-6+2\xi({{\bf 1}})-4.25\varepsilon_0},\ \
\mathcal{E}^{({\bf 1})}(\phi;p=2;r_1) \lesssim \epsilonc r_1^{-4+2\xi({{\bf 1}})-4.5\varepsilon_0}.
\end{split}
\end{equation*}
For sufficiently large $R_*$, we then have closed the bootstrap argument for first order energy quantities on scalar field in $\mathbf{(B)}$:
\begin{equation}\label{improved bootstrap first order scalar}
\begin{split}
 \mathcal{E}^{{({\bf 1})}}& \leq 2 {\epsilonc}r_1^{-6+2\xi({\bf 1})-4\varepsilon_0},\ \
\mathcal{E}^{{({\bf 1})}}(\phi;p=2;r_1) \leq 2{\epsilonc}r_1^{-4+2\xi({\bf 1})-4\varepsilon_0}.
\end{split}
\end{equation}

\subsubsection{Energy estimates on second derivatives of the scalar field}
We now fix a $\k$ so that $|\k|=2$ and the first objective is to bound the $\R_1$ and $\R_2$ term in \eqref{energy identities to be bounded}. For this purpose, we first recall that, for $({\bf 2})$ representing $\Dt_{Z_1}\Dt_{Z_2}$, we have
\begin{equation*}
\Box_A \phi^{(\mathbf{2})} =Q(\Dt_{Z_1}\phi,F;Z_2)  + Q(\Dt_{Z_2}\phi,F;Z_1)+Q(\phi,  F;[Z_1,Z_2])+Q(\phi,  \mathcal{L}_{Z_1} F;Z_2)-2F_{Z_1\mu}F_{Z_2}{}^{\mu}\phi.
\end{equation*}

For $\R_1$, according to the above expression, we split it into three parts:
\begin{equation*}
\begin{split}
\R_1&\lesssim\int_{\D_{r_1}} \underbrace{\Big(\big|Q(\Dt_{Z_1}\phi,F;Z_2)\big|+\big|Q(\Dt_{Z_2}\phi,F;Z_1)\big|+\big|Q(\phi,  F;[Z_1,Z_2])\big|\Big)\big(| D_L \phi^{(\mathbf{2})}|+|D_\Lb\phi^{(\mathbf{2})}|\big)}_{\R_{11}}\\
& \ \ \ + \int_{\D_{r_1}}\underbrace{\big|Q(\phi, \mathcal{L}_{Z_1} F;Z_2)\big|\big(| D_L \phi^{(\mathbf{2})}|+|D_\Lb\phi^{(\mathbf{2})}|\big)}_{\R_{12}}+\int_{\D_{r_1}}\underbrace{\big|F_{Z_1\mu}F_{Z_2}{}^{\mu}\phi\big|\big(| D_L \phi^{(\mathbf{2})}|+|D_\Lb\phi^{(\mathbf{2})}|\big)}_{\R_{13}}
\end{split}
\end{equation*}
All the three $Q$-terms in $\R_{11}$ can be schematically written as either $Q(\phi^{({\bf 1})},F;Z)$ or $Q(\phi^{({\bf 0})},F;Z)$ due to the observation that the linear span of $\mathcal{Z}$ is closed under commutations.  These terms resemble the terms in $\R_1$ in Section \ref{section on 1 derivatives on scalar}. Thanks to \eqref{the first null form bound on spacetime slab with one more derivatives}, they can be bounded exactly in the same manner:
\begin{align*}
\R_{1}
&\lesssim \Big(\int_{\D_{r_1}} r^2\big|Q(\Dt_{Z_1}\phi,F;Z_2)\big|^2+r^2\big|Q(\Dt_{Z_2}\phi,F;Z_1)\big|^2+r^2\big|Q(\phi,  F;[Z_1,Z_2])\big|^2 \Big)^\frac{1}{2} \\
&\ \ \ \times\Big(\int_{\D_{r_1}} r^{-2}\big(|D_L\phi^{({\bf 2})}|+|D_\Lb\phi^{({\bf 2})}| \big)^2\big)\Big)^\frac{1}{2} \lesssim {\epsilonc} r_1^{-6+2\xi({\bf 2})-3\varepsilon_0}.
\end{align*}
Here we remark that compared with the estimate of $\R_1$ in the last subsection we lose a decay power of $\varepsilon_0$ is due to the weaker decay of second order energy estimates in the bootstrap assumption.

For $\R_{12}$, we use $(\mathbf{1})$ to denote the vector field $Z_1$, according to \eqref{null form estimate} and the pointwise bounds on the scalar field, we have
\begin{equation*}
\begin{split}
|u|^{-\xi(Z_2)}|Q(\phi,\mathcal{L}_{Z_1}F;Z_2)| &\lesssim \big(\frac{r}{|u|}|\rho(\mathcal{L}_{Z_1}F)|+|\alphab(\mathcal{L}_{Z_1}F)|\big)|D_L\psi|\\
&\ \ +\big(\frac{r}{|u|}|\alpha(\mathcal{L}_{Z_1}F)| +\frac{|u|}{r}|\alphab(\mathcal{L}_{Z_1}F)|+|\sigma(\mathcal{L}_{Z_1}F)|\big)|\slashed{D}\psi| \\
& \ \ + \big(|\alpha(\mathcal{L}_{Z_1}F)|+\frac{|u|}{r}|\rho(\mathcal{L}_{Z_1}F)|\big)|D_\Lb\psi|+\big(|\rho(\mathcal{L}_{Z_1}F)|+|\sigma(\mathcal{L}_{Z_1}F)|\big)|\phi|\\
& \ \ +\big(\frac{|u|}{r^2}|J(\mathcal{L}_{Z_1}F)_\Lb|+\frac{1}{|u|}|J(\mathcal{L}_{Z_1}F)_L|+\frac{1}{r}|\slashed{J}(\mathcal{L}_{Z_1}F)|\big)|\phi|.
\end{split}
\end{equation*}
Since $F=\Fc+F[q_0]$ and $F[q_0]$ solves the linear Maxwell equations, accroding to \eqref{commutator formula 2}, we have
\begin{equation*}
J(\mathcal{L}_{Z_1}F)_\Lb=J^{({\bf 1})}_\Lb, \ J(\mathcal{L}_{Z_1}F)_L=J^{({\bf 1})}_L, \ \slashed{J}(\mathcal{L}_{Z_1}F)=\slashed{J}^{({\bf 1})}.
\end{equation*}
Therefore, according to the pointwise decay for the scalar field, we have
\begin{equation*}
\begin{split}
&\ \ \ \ \ |u|^{-\xi(Z_2)}|Q(\phi,\mathcal{L}_{Z_1}F;Z_2)|\\
&\lesssim  \big(\frac{r}{|u|}|\slashed{D}\psi|+|D_\Lb\psi|\big) |\alpha(\mathcal{L}_{Z_1}F)|+  \big(\frac{r}{|u|}|D_L\psi|+\frac{|u|}{r}|D_\Lb\psi|+|\phi|\big)|\rho(\mathcal{L}_{Z_1}F)|+(|\slashed{D}\psi|+|\phi|\big)|\sigma(\mathcal{L}_{Z_1}F)|   \\
&\ \ \ + \big(|D_L\psi|+\frac{|u|}{r}|\slashed{D}\psi|\big) |\alphab(\mathcal{L}_{Z_1}F)|+\big(\frac{|u|}{r^2}|J^{({\bf 1})}_\Lb|+\frac{1}{|u|}|J^{({\bf 1})}_L|+\frac{1}{r}|\slashed{J}^{({\bf 1})}|\big)|\phi| \\
&\lesssim \underbrace{\frac{\sqrt{{\epsilonc}}}{|u|^{3+\varepsilon_0}} |\alpha(\mathcal{L}_{Z_1}F)|}_{\A_0} + \underbrace{\frac{\sqrt{{\epsilonc}}}{r|u|^{2+\varepsilon_0}}|\rho(\mathcal{L}_{Z_1}F)|+\frac{\sqrt{{\epsilonc}}}{r|u|^{2+\varepsilon_0}}|\sigma(\mathcal{L}_{Z_1}F)|}_{\A_1}   +\underbrace{\frac{\sqrt{{\epsilonc}}}{r^2|u|^{1+\varepsilon_0}}|\alphab(\mathcal{L}_{Z_1}F)|}_{\A_2} \\
&\ \ \ +\underbrace{\frac{\sqrt{{\epsilonc}}}{r|u|^{\frac{7}{2}+2\varepsilon_0}}|J^{({\bf 1})}_L|+\frac{\sqrt{{\epsilonc}}}{r^2|u|^{\frac{5}{2}+2\varepsilon_0}}|\slashed{J}^{({\bf 1})}|}_{\A_{3}}+\underbrace{\frac{\sqrt{{\epsilonc}}}{r^3|u|^{\frac{3}{2}+2\varepsilon_0}}|J^{({\bf 1})}_\Lb|}_{\A_{4}}.
\end{split}
\end{equation*}

On the other hand, according to Lemma \ref{lemma commuting Z with null decomposition}, we have
\begin{align*}
\big|\alpha(\mathcal{L}_{Z_1}F[q_0])\big|&\leq \big|\mathcal{L}_{Z_1}\big(\alpha(F[q_0])\big)\big|+r^{\xi(Z_1)}\big|\alpha(F[q_0])\big|\lesssim r^{-3+\xi(Z_1)}.
\end{align*}
Hence,
\begin{align*}
\big|\alpha(\mathcal{L}_{Z_1}F)\big|\leq \big|\alpha(\mathcal{L}_{Z_1}\Fc)\big|+\big|\alpha(\mathcal{L}_{Z_1}F[q_0])\big|\leq |\alphaone| + r^{-3+\xi(Z_1)}.
\end{align*}
Similarly, since we have
\begin{align*}
\big|\alphab(\mathcal{L}_{Z_1}F[q_0])\big|&\lesssim r^{-3+\xi(Z_1)},\ \big|\rho(\mathcal{L}_{Z_1}F[q_0])\big|\lesssim r^{-3+\xi(Z_1)},\ \sigma(\mathcal{L}_{Z_1}F[q_0])=0,
\end{align*}
We notice that the estimate on $\rho(\mathcal{L}_{Z_1}F[q_0])$ is as good as the other components. This is due to the fact that $\mathcal{L}_Z \big(\frac{1}{r^2}dt\wedge dr\big)=0$ for all $Z \in \mathcal{Z}$. We conclude that
\begin{equation}\label{onederivative of F}
\begin{split}
\big|\alpha(\mathcal{L}_{Z_1}F)&\lesssim	|\alphaone| + r^{-3+\xi(Z_1)}, \ \big|\alphab(\mathcal{L}_{Z_1}F)\lesssim	|\alphabone| + r^{-3+\xi(Z_1)},\\
\big|\rho(\mathcal{L}_{Z_1}F)&\lesssim	|\rhoone| + r^{-3+\xi(Z_1)}, \ \big|\sigma(\mathcal{L}_{Z_1}F)\lesssim	|\sigmaone|.
\end{split}
\end{equation}
We notice that for $Z_1=K$ we lose decay in $r$. For the $\alpha$ component, we can improve the decay in $r$:
\begin{lemma}
\begin{equation}\label{improved estimate for Fq0}
|\alpha(\mathcal{L}_{K}F[q_0])\big|\lesssim r^{-3}|u|.
\end{equation}
\end{lemma}
\begin{proof}We recall the definition for $F[q_0]$:
\[
F[q_0]_{0i}=\partial_{i}V(x),\quad F[q_0]_{ij}=0, \text{for} i,j=1,2,3,
\]
where the potential $V(x)$ is given by
\[
V(x)=\frac{1}{4\pi}\int_{\mathbb{R}^3}(\underbrace{\frac{1}{r}}_{V_1}+\underbrace{\frac{x\cdot y}{r^3}}_{V_2}+\underbrace{\frac{1}{2} \frac{(3|x|^{-2}(x\cdot y)^2-|y|^2)}{r^3}}_{V_3})\Im(\phi_0\cdot \bar \phi_1)dy,\quad |x|>0.
\]
The contribution from $V_3$ is of order $r^{-3}$ so that we can ignore it. The contribution from $V_1$ gives the charge part $\frac{1}{r^2}dt\wedge dr$ and it will vanish when one takes $\mathcal{L}_K$ derivative. Thus, we consider
\[
F^{(2)}[q_0]_{0i} = \partial_i\Big(\frac{1}{4\pi}\int_{\mathbb{R}^3} \frac{x\cdot y}{r^3}\Im\big((\phi_0\cdot \bar \phi_1)(y)\big)dy\Big), \ F^{(2)}[q_0]_{ij}=0.
\]
Thus, we have
\begin{equation*}
\alpha(F^{(2)}[q_0])_A=\frac{1}{4\pi}\int_{\mathbb{R}^3} \frac{e_A\cdot y}{r^3}\Im\big((\phi_0\cdot \bar \phi_1)(y)\big)dy.
\end{equation*}
By virtue of the formula for $\mathcal{L}_K$ in Lemma \ref{lemma commuting Z with null decomposition}, we obtain
\begin{equation*}
\alpha(\mathcal{L}_K F^{(2)}[q_0])_A=\frac{1}{4\pi}\int_{\mathbb{R}^3} \frac{(r-t)e_A\cdot y}{r^3}\Im\big((\phi_0\cdot \bar \phi_1)(y)\big)dy.
\end{equation*}
This completes the proof of the lemma.
\end{proof}
As a corollary, we have
\begin{equation}\label{precise bound on alpha F Fq0}
\big|\alpha(\mathcal{L}_{Z_1}F)\lesssim	|\alphaone| + r^{-3}|u|^{\xi(Z_1)}.
\end{equation}
\begin{lemma}We have the following spacetime estimates:
\begin{equation}\label{bound on r box phi 2}
\|rQ(\phi,\mathcal{L}_{Z_1}F;Z_2)\|_{L^2(\D_{r_1})} \lesssim \sqrt{\epsilonc} r_1^{\xi({\bf 2})-3-\varepsilon_0}.
\end{equation}
\end{lemma}
\begin{proof}With the help of \eqref{onederivative of F} and \eqref{improved estimate for Fq0}, we can bound the terms $\int_{\D_{r_1}} r^2|u|^{2\xi(Z_2)}|\A_i|^2$ one by one. This will prove the lemma:
\begin{equation*}
\begin{split}
\int_{\D_{r_1}} r^2|u|^{2\xi(Z_2)}|\A_0|^2  &\lesssim {\epsilonc}\int_{\D_{r_1}}{|u|^{2\xi({Z_2})-6-2\varepsilon_0}}\big(r^2|\alphaone|^2 +r^{-4}|u|^{2\xi(Z_1)}\big) \\
&\lesssim {\epsilonc}\int_{r_1}^\infty{|u|^{2\xi({Z_2})-6-2\varepsilon_0}}\Big(\int_{\H_{r_2}}\big(r^2|\alphaone|^2 +r^{-4}|u|^{2\xi(Z_1)}\big)\Big)dr_2.
\end{split}
\end{equation*}

For $\A_1$, we have
\begin{equation*}
\begin{split}
\int_{\D_{r_1}} r^2|u|^{2\xi(Z_2)}|\A_1|^2  &\lesssim {\epsilonc}\int_{\D_{r_1}}{|u|^{2\xi({Z_2})-4-2\varepsilon_0}}\big(|\rhoone|^2 +|\sigmaone|^2 + r^{-6+2\xi(Z_1)}\big) \\
&\lesssim {\epsilonc}\int_{r_1}^\infty{|u|^{2\xi({Z_2})-4-2\varepsilon_0}}\Big(\int_{\H_{r_2}}\big(|\rhoone|^2 +|\sigmaone|^2 + r^{-6+2\xi(Z_1)}\big)\Big)dr_2.
\end{split}
\end{equation*}

For $\A_2$, we have
\begin{equation*}
\begin{split}
\int_{\D_{r_1}} r^2|u|^{2\xi(Z_2)}|\A_2|^2  &\lesssim {\epsilonc}\int_{\D_{r_1}}{|u|^{2\xi({Z_2})-2-2\varepsilon_0}}r^{-2}\big(|\alphabone|^2 + r^{-6+2\xi(Z_1)}\big) \\
&\lesssim {\epsilonc}\int_{\frac{r_1}{2}}^\infty{r_1^{2\xi({Z_2})-4-2\varepsilon_0}}v^{-2}\Big(\int_{\Hb_{2v}^{r_1}}\big(|\alphabone|^2  + r^{-6+2\xi(Z_1)}\big)\Big)dr_2. 
\end{split}
\end{equation*}

For $\A_{3}$, based on the ansatz $\mathbf{(C)}$, we can proceed in the same manner to obtain
\begin{align*}
\int_{\D_{r_1}} r^2|u|^{2\xi(Z_2)}|\A_3|^2& \lesssim {\epsilonc}\int_{\D_{r_1}}{|u|^{2\xi({Z_2})-7-4\varepsilon_0}} |J^{({\bf 1})}_L|^2+|u|^{2\xi({Z_2})-5-4\varepsilon_0}\frac{|\slashed{J}^{({\bf 1})}|^2}{r^2}\\
&={\epsilonc}\int_{r_1}^\infty \Big({|u|^{2\xi({Z_2})-7-4\varepsilon_0}} \int_{\H_{r_2}}|J^{({\bf 1})}_L|^2+|u|^{2\xi({Z_2})-5-4\varepsilon_0}\int_{\H_{r_2}}\frac{|\slashed{J}^{({\bf 1})}|^2}{r^2}\Big)dr_2.
\end{align*}
All the terms on the righthand sides of the above four inequalities now can be integrated. They are all bounded by ${\epsilonc} r_1^{2\xi({\bf 2})-6-2\varepsilon_0}$.

For $\A_4$, let the vector field $Z$ represent the index $(\mathbf{1})$, we have
\begin{align*}
J^{(\mathbf{1})} &= \mathcal{L}_Z (r^2 J)=\mathcal{L}_Z(\Im(\overline{\psi}\cdot D\psi )).
\end{align*}
Since
\begin{equation*}
\mathcal{L}_{Z}\big(\overline{\psi}\cdot D\psi\big)_\mu = \overline{D_Z\psi} \cdot D_\mu\psi + (D_\mu \log(r))\overline{\psi}\cdot D_Z \psi+r\overline{\psi} \cdot D_\mu \big(\Dt_Z\phi\big) +i F_{Z\mu}|\psi|^2,
\end{equation*}
we have
\begin{equation}\label{formula and bound on J1}
{|J^{(\mathbf{1})}_\mu|^2}\lesssim   r^4|\phi|^2|D_\mu\big(\Dt_Z\phi\big)|^2 +  \big(|D_\mu \psi|^2+|\phi|^2\big)|D_Z\psi|^2    +    r^4|F_{Z\mu}|^2|\phi|^4.
\end{equation}
In particular, we have
\begin{align*}
\int_{\D_{r_1}} r^2|u|^{2\xi(Z_2)}|\A_4|^2& \lesssim {\epsilonc}\int_{\D_{r_1}}{|u|^{2\xi({Z_2})-3-4\varepsilon_0}}\big(|\phi|^2|D_\Lb \big(\Dt_{Z_1}\phi\big)|^2  + \frac{\big(|D_\Lb \psi|^2+|\phi|^2\big)|D_{Z_1}\psi|^2}{r^4}+ |F_{{Z_1}\Lb}|^2|\phi|^4\big).
\end{align*}
In view of the pointwise bounds, we can then use the following crude bound for $D_{Z_1}\psi$ and $F_{Z_1\Lb}$:
\begin{equation*}
|D_{Z_1}\psi|\lesssim |u|^{\xi(Z_1)-1-\varepsilon_0},\ \   |F_{Z_1\Lb}| \lesssim r^{\xi(Z_1)-1}|u|^{-1-\varepsilon_0}.
\end{equation*}
Therefore, we obtain
\begin{align*}
\int_{\D_{r_1}} r^2|u|^{2\xi(Z_2)}|\A_4|^2& \lesssim {\epsilonc}\int_{\D_{r_1}}{|u|^{2\xi({Z_2})-5-4\varepsilon_0}}\big(\frac{|D_\Lb \big(\Dt_{Z_1}\phi\big)|^2+|D_\Lb \phi|^2}{r^2}  +r^{-4}|u|^{-1}\epsilonc^2\big)\\
&\lesssim \epsilonc^2 r_1^{2\xi({\bf 2})-6-2\varepsilon_0}，
\end{align*}
where we bound $D_\Lb(\Dt_{Z_1}\phi)$ and $D_\Lb \phi$ on $\Hb_{r_2}$ as before.

We complete the proof by putting the estimates of the $\A_i$'s all together.
\end{proof}
The term $\R_{12}$ can be easily bounded by the lemma:
\begin{align*}
\int_{\D_{r_1}} \R_{12}  & \lesssim \|rQ(\phi,\mathcal{L}_{Z_1}F;Z_2)\|_{L^2(\D_{r_1})} \Big(\int_{\D_{r_1}}r^{-2}(| D_L \phi^{(\mathbf{2})}|^2+|D_\Lb\phi^{(\mathbf{2})}|^2)\Big)^{\frac{1}{2}}\\
&\lesssim {\epsilonc}  r_1^{2\xi({\bf 2})-6.5-2\varepsilon_0}.
\end{align*}

To bound $\R_{13}$, we need the following lemma:
\begin{lemma}We have the following estimates:
\begin{equation}\label{bound on r box phi 2 2}
\|r F_{Z_1 \mu}F_{Z_2}{}^\mu\|_{L^2(\D_{r_1})} \lesssim \sqrt{\epsilonc} r_1^{\xi({\bf 2})-2.5-2\varepsilon_0}.
\end{equation}
\end{lemma}
\begin{proof}
According to the different choices of $Z_1$ and $Z_2$, we have
\begin{itemize}
\item[Case 1]$(Z_1,Z_2) =(\Omega, \Omega')$. We have
\begin{align*}
r |F_{\Omega \mu} F_{\Omega}{}^\mu \phi|&\lesssim r^3|\phi|\big( |\alpha||\alphab|+|\sigma|^2\big)\lesssim\frac{\sqrt{{\epsilonc}}}{r^{4}|u|^{\frac{5}{2}+2\varepsilon_0}}.
\end{align*}
\item[Case 2]$Z_1=\Omega$ and $Z_2 =v^{1+\xi(Z_2)}L+u^{1+\xi(Z_2)}\Lb$. Thus,
\begin{align*}
r|F_{\Omega \mu} F_{Z_2}{}^\mu\phi|&\lesssim r^2|\phi|\big(|\sigma|+|\rho|\big)\big(v^{1+\xi(Z_2)}|\alpha|+u^{1+\xi(Z_2)}|\alphab|\big)\lesssim \frac{\sqrt{{\epsilonc}}}{r^{3-\xi(Z_2)}|u|^{\frac{5}{2}+2\varepsilon_0}}
\end{align*}
\item[Case 3]$Z_1 =v^{1+\xi(Z_1)}L+u^{1+\xi(Z_1)}\Lb$ and $Z_2 =v^{1+\xi(Z_2)}L+u^{1+\xi(Z_2)}\Lb$. We have
\begin{align*}
|F_{Z_1 \mu} F_{Z_2}{}^\mu|&\lesssim v^{2+\xi(Z_1)+\xi(Z_2)}|\alpha|^2+|u|^{2+\xi(Z_1)+\xi(Z_2)}|\alphab|^2\\
&+\big( |u|^{1+\xi(Z_1)} v^{1+\xi(Z_2)}+ |u|^{1+\xi(Z_2)} v^{1+\xi(Z_1)}\big)\big(|\alpha||\alphab|+|\rho|^2\big).
\end{align*}
Thus, we have
\begin{align*}
r|F_{Z_1 \mu} F_{Z_2}{}^\mu \phi|&\lesssim \sqrt{\epsilonc}|u|^{-\frac{5}{2}-2\varepsilon_0}\big(r^{-4+\xi{(\mathbf{2})}}+|u|^{-4+\xi{(\mathbf{2})}}r^{-2}\epsilonc+|u|^{1+\xi(Z_1)} r^{-3+\xi(Z_2)}+ |u|^{1+\xi(Z_2)} r^{-3+\xi(Z_1)}\big).
\end{align*}
\end{itemize}
Then we can simply integrate the above pointwise bounds to conclude.
\end{proof}
This lemma leads to the estimate of $\R_{13}$:
\begin{equation*}
\begin{split}
\int_{\D_{r_1}} \R_{13}  & \lesssim \|r F_{Z_1 \mu}F_{Z_2}{}^\mu\|_{L^2(\D_{r_1})} \Big(\int_{\D_{r_1}}r^{-2}(| D_L \phi^{(\mathbf{2})}|^2+|D_\Lb\phi^{(\mathbf{2})}|^2)\Big)^{\frac{1}{2}}\\
&\lesssim {\epsilonc} r_1^{2\xi({\bf 2})-6-1.5\varepsilon_0}.
\end{split}
\end{equation*}
Finally, from the estimates of $\R_{11},\R_{12}$ and $\R_{13}$, we conclude that
\begin{align*}
\int_{\D_{r_1}}\R_1 & \lesssim {\epsilonc} r_1^{-6+2\xi({\bf 2})-1.5\varepsilon_0}.
\end{align*}
Based on \eqref{bound on r box phi 2} and \eqref{bound on r box phi 2 2}, one can also proceed exactly in the same manner to prove that
\begin{align*}
\int_{\D_{r_1}}\R_2 &\lesssim {\epsilonc} r_1^{-4+2\xi({\bf 1})-1.5\varepsilon_0}.
\end{align*}
Therefore, for sufficiently large $R_*$, since $r_1\geq R_*$, we have
\begin{equation*}
\begin{split}
 \mathcal{E}^{{({\bf 2})}}(\phi;r_1) &\lesssim \epsilonc r_1^{-6+2\xi({{\bf 2}})-1.5\varepsilon_0}+\int_{\D_{r_1}}\underbrace{|\rho|\big(|D_L\phi^{({\bf 2})}|+|D_\Lb\phi^{({\bf 2})}|\big)|\phi^{({\bf 2})}|}_{\T_1},\\
\mathcal{E}^{({\bf 2})}(\phi;p=2;r_1) & \lesssim \epsilonc r_1^{-4+2\xi({{\bf 2}})-1.5\varepsilon_0}+ \int_{\D_{r_1}}\underbrace{|\rho||D_L\psi^{({\bf 2})}||\psi^{({\bf 2})}|}_{\T_2}.
\end{split}
\end{equation*}
For $\T_1$ and $\T_2$ we can the proceed exactly in the same manner as in the previous subsection. Finally, for sufficiently large $R_*$, we can close the bootstrap argument for second order energy quantities on scalar field in $\mathbf{(B)}$:
\begin{equation}\label{improved bootstrap second order scalar}
\begin{split}
 \mathcal{E}^{{({\bf 2})}}& \leq 2 {\epsilonc}r_1^{-6+2\xi({\bf 1})-2\varepsilon_0},\ \
\mathcal{E}^{{({\bf 2})}}(\phi;p=2;r_1) \leq 2{\epsilonc}r_1^{-4+2\xi({\bf 1})-2\varepsilon_0}.
\end{split}
\end{equation}

\subsubsection{The estimates on the current terms}

 We now recover the estimates for the current terms in $(\mathbf{C})$.

\medskip

\underline{$\blacktriangleright$~~Zeroth order estimates}

\medskip

For $\k=(\mathbf{0})$, $J^{(\mathbf{0})} = r^2 J = \Im(\overline{\psi}\cdot D\psi )$, according to the pointwise bound on $\phi$,  we have
\begin{align*}
|J^{(\mathbf{0})}_L|&\lesssim \epsilonc r^{-2}|u|^{-\frac{7}{2}-3\varepsilon_0},\ \ |\slashed{J}^{(\mathbf{0})}|\lesssim \epsilonc r^{-1}|u|^{-\frac{9}{2}-3\varepsilon_0}, \ \ |J^{(\mathbf{0})}_L| \lesssim \epsilonc |u|^{-\frac{11}{2}-3\varepsilon_0}.
\end{align*}
We then can directly integrate these bounds and we obtain
\begin{align*}
& r_1^{-2}\int_{\H_{r_1}}{|J^{(\mathbf{0})}_L|^2}+\int_{\H_{r_1}}\frac{|\slashed{J}^{(\mathbf{0})}|^2}{r^2}+\sup_{r_2 \geq r_1}\int_{\Hb_{r_2}^{r_1}} \frac{|\slashed{J}^{(\mathbf{0})}|^2}{r^2} +\sup_{r_2 \geq r_1}r_1^{\frac{3}{2}}\int_{\Hb_{r_2}^{r_1}} \frac{|J^{(\mathbf{0})}_\Lb|^2}{r^{\frac{7}{2}}}\lesssim {\epsilonc}^2 r_1^{-10+2\xi(\k)-6\varepsilon_0}.
\end{align*}
Thus, for sufficiently large $R_*$, it immediately closes $(\mathbf{C})$ for $\k=(\mathbf{0})$.

\medskip

\underline{$\blacktriangleright$~~First order estimates}

\medskip

For $\k=(\mathbf{1})$, we use a vector field $Z$ to represent this index. Notice that we have
\begin{align*}
J^{(\mathbf{1})}_\mu&=\Im\Big(\overline{D_Z\psi} \cdot D_\mu\psi + \overline{\psi}\cdot D_\mu(r\Dt_Z\phi) +i F_{Z\mu}|\psi|^2\Big)\\
&=\Im\Big(\overline{\psi^{(\mathbf{1})}} \cdot D_\mu\psi + \overline{\psi}\cdot D_\mu \psi^{(\mathbf{1})} \Big)+ F_{Z\mu}|\psi|^2.
\end{align*}

$\blacklozenge$~For $J^{(\mathbf{1})}_L$, we have
\begin{align*}
|J^{(\mathbf{1})}_L|&\lesssim |\psi^{(\mathbf{1})}||D_L\psi| + |\psi||D_L \psi^{(\mathbf{1})}|+ |F_{ZL}||\psi|^2\\
&\lesssim \underbrace{\sqrt{\epsilonc}r^{-1}|u|^{-1-\varepsilon_0}|\phi^{(\mathbf{1})}|}_{\I_{L1}} + \underbrace{\sqrt{\epsilonc}|u|^{-\frac{5}{2}-2\varepsilon_0}|D_L \psi^{(\mathbf{1})}|}_{\I_{L2}}+\underbrace{{\epsilonc} |u|^{-5-4\varepsilon_0}|F_{ZL}|}_{\I_{L3}}.
\end{align*}
For $k\leq 2$, by Lemma \ref{Sobolev trace estimates on outgoing null hypersurfaces}, we have
\begin{equation}\label{L^2 estimates on spheres for k=1 2}
\begin{split}
\|\phi^{(\mathbf{k})}\|^2_{L^2(\S_{r_1}^{r_2})}&\lesssim \|\phi^{(\mathbf{k})}\|^2_{L^2(\S_{r_1}^{r_1})}+\frac{1}{r_1}\int_{\H_{r_1}} |D_L \psi^{(\mathbf{k})}|^2 \lesssim  {\epsilonc} r_1^{-5+2\xi(\mathbf{k})-2\varepsilon_0}.
\end{split}
\end{equation}
For $\I_{L1}$, we have
\begin{align*}
r_1^{-2}\int_{\H_{r_1}}  |\I_{L1}|^2&\lesssim \epsilonc r_1^{-2}  |u|^{-2-2\varepsilon_0} \int_{r_1}^{r_2}r^{-2}\big(\int_{\S_{r_1}^r}|\phi^{(\mathbf{1})}|^2\big) dr 
\lesssim \epsilonc^2r_1^{-10+2\xi(\mathbf{1})-4\varepsilon_0}.
\end{align*}
For $\I_{L2}$, we have
\begin{align*}
r_1^{-2}\int_{\H_{r_1}}  |\I_{L2}|^2&\lesssim \epsilonc r_1^{-2}  |u|^{-5-4\varepsilon_0} \int_{\H_{r_1}}|D_L \psi^{(\mathbf{1})}|^2\lesssim \epsilonc^2r_1^{-11+2\xi(\mathbf{1})-6\varepsilon_0}.
\end{align*}
For $\I_{L3}$, we have two cases:
\begin{equation*}
\I_{L3} \leq \begin{cases}{\epsilonc} |u|^{-5-4\varepsilon_0}r|\alpha|\lesssim {\epsilonc} |u|^{-5-4\varepsilon_0}r^{-2}, \ \text{if} \ Z=\Omega;\\
{\epsilonc} |u|^{-5-4\varepsilon_0}|u|^{1+\xi(\mathbf{1})}|\rho|\lesssim {\epsilonc} |u|^{-4+\xi(\mathbf{1})-4\varepsilon_0}r^{-2}, \ \text{if} \ Z=v^{1+\xi(\mathbf{1})}L+u^{1+\xi(\mathbf{1})}\Lb.
\end{cases}
\end{equation*}
In both cases, we can simply directly integrate the pointwise bounds on $\H_{r_1}$. Therefore, the contribution from $\I_{L3}$ is also bounded by $\epsilonc^2r_1^{-11+\xi(\mathbf{1})-8\varepsilon_0}$.
Hence, we conclude that
\begin{align*}
r_1^{-2}\int_{\H_{r_1}}{|J^{(\mathbf{1})}_L|^2} \lesssim {\epsilonc}^2 r_1^{-10+2\xi(\mathbf{1})-4\varepsilon_0}.
\end{align*}

$\blacklozenge$~For $\slashed{J}^{(\mathbf{1})}$, we have
\begin{align*}
|\slashed{J}^{(\mathbf{1})}|&\lesssim |\psi^{(\mathbf{1})}||\slashed{D}\psi| + |\psi||\slashed{D} \psi^{(\mathbf{1})}|+ |F_{ZA}||\psi|^2\\
&\lesssim \underbrace{\sqrt{\epsilonc}|u|^{-2-\varepsilon_0}|\phi^{(\mathbf{1})}|}_{\slashed{\I}_{1}} + \underbrace{\sqrt{\epsilonc}r|u|^{-\frac{5}{2}-2\varepsilon_0}|\slashed{D} \phi^{(\mathbf{1})}|}_{\slashed{\I}_{2}}+\underbrace{{\epsilonc} |u|^{-5-4\varepsilon_0}|F_{ZA}|}_{\slashed{\I}_{3}}.
\end{align*}
For $\slashed{\I}_{1}$, according to \eqref{L^2 estimates on spheres for k=1 2}, we have
\begin{align*}
\int_{\H_{r_1}}  \frac{|\slashed{\I}_{1}|^2}{r^2}+\int_{\Hb_{r_1}^{r_2}}  \frac{|\slashed{\I}_{1}|^2}{r^2}&\lesssim \epsilonc   \int_{r_1}^{r_2}r^{-2}|u|^{-4-2\varepsilon_0} \big(\int_{\S_{r_1}^r}|\phi^{(\mathbf{1})}|^2\big) dr \lesssim \epsilonc^2r_1^{-10+2\xi(\mathbf{1})-2\varepsilon_0}.
\end{align*}
For $\slashed{\I}_{2}$, we have
\begin{align*}
\int_{\H_{r_1}}  \frac{|\slashed{\I}_{2}|^2}{r^2}+\int_{\Hb_{r_1}^{r_2}}  \frac{|\slashed{\I}_{2}|^2}{r^2}&\lesssim \epsilonc   |u|^{-5-4\varepsilon_0} \Big(\int_{\H_{r_1}}|\slashed{D} \phi^{(\mathbf{1})}|^2+ \int_{\Hb_{r_1}^{r_2}}|\slashed{D} \phi^{(\mathbf{1})}|^2\Big)
\lesssim \epsilonc^2r_1^{-11+2\xi(\mathbf{1})-8\varepsilon_0}.
\end{align*}
For $\slashed{\I}_{3}$, we have two cases:
\begin{equation*}
|\slashed{\I}_{3}| \leq \begin{cases}&{\epsilonc} |u|^{-5-4\varepsilon_0}r|\sigmac|\lesssim {\epsilonc}^\frac{3}{2} |u|^{-7-6\varepsilon_0}r^{-1},  \ \text{if} \ Z=\Omega;\\
&
\lesssim {\epsilonc} |u|^{-5-4\varepsilon_0}\big(r^{-2+\xi(\mathbf{1})}+\sqrt{\varepsilon_0}r^{-1}|u|^{-2+\xi(\mathbf{1})-\varepsilon_0}\big), \ \text{if} \ Z=v^{1+\xi(\mathbf{1})}L+u^{1+\xi(\mathbf{1})}\Lb.
\end{cases}
\end{equation*}
In both cases, the contribution of $\slashed{I}_3$ can be estimated directly by integrating the above bounds and it is bounded by $\epsilonc^2r_1^{-13+2\xi(\mathbf{1})-8\varepsilon_0}$.
Hence, we conclude that
\begin{align*}
\int_{\H_{r_1}}  \frac{|\slashed{J}^{(\mathbf{1})}|^2}{r^2}+\sup_{r_2\geq r_1}\int_{\Hb_{r_1}^{r_2}}  \frac{|\slashed{J}^{(\mathbf{1})}|^2}{r^2}&\lesssim \epsilonc^2r_1^{-10+2\xi(\mathbf{1})-2\varepsilon_0}.
\end{align*}

$\blacklozenge$~For $J^{(\mathbf{1})}_\Lb$, we have
\begin{align*}
|J^{(\mathbf{1})}_\Lb|&\lesssim |\psi^{(\mathbf{1})}||D_\Lb\psi| + |\psi||D_\Lb \psi^{(\mathbf{1})}|+ |F_{Z\Lb}||\psi|^2\\
&\lesssim \underbrace{\sqrt{\epsilonc}r|u|^{-\frac{5}{2}-2\varepsilon_0}|\phi^{(\mathbf{1})}|}_{\I_{\Lb1}} + \underbrace{\sqrt{\epsilonc}|u|^{-\frac{5}{2}-2\varepsilon_0}|D_\Lb \psi^{(\mathbf{1})}|}_{\I_{\Lb2}}+\underbrace{{\epsilonc} |u|^{-5-4\varepsilon_0}|F_{Z\Lb}|}_{\I_{\Lb3}}.
\end{align*}
For $\I_{\Lb1}$, according to \eqref{L^2 estimates on spheres for k=1 2}, we have
\begin{align*}
r_1^{\frac{3}{2}}\int_{\Hb_{r_1}^{r_2}}  \frac{|{\I}_{\Lb1}|^2}{r^{\frac{7}{2}}}&\lesssim \epsilonc   r_1^{\frac{3}{2}}\int_{r_1}^{r_2}r^{-\frac{3}{2}}r_1^{-5-4\varepsilon_0} \big(\int_{\S_{r_1}^r}|\phi^{(\mathbf{1})}|^2\big) dr \lesssim \epsilonc^2r_1^{-10+2\xi(\mathbf{1})-6\varepsilon_0}.
\end{align*}
For $\I_{\Lb2}$, we first notice that $|D_\Lb \psi^{(\mathbf{1})}|\lesssim r |D_\Lb \phi^{(\mathbf{1})}|+|\phi^{(\mathbf{1})}|$. The contribution from $|\phi^{(\mathbf{1})}|$ can be ignored since it has been already treated in $\I_{\Lb1}$. Thus, modulo this term, we have
\begin{align*}
r_1^{\frac{3}{2}}\int_{\Hb_{r_1}^{r_2}}  \frac{|{\I}_{\Lb2}|^2}{r^{\frac{7}{2}}}&\lesssim \epsilonc   r_1^{\frac{3}{2}}\int_{\Hb_{r_1}^{r_2}} r^{-\frac{3}{2}}|u|^{-5-4\varepsilon_0}|D_\Lb\phi^{(\mathbf{1})}|^2 \lesssim \epsilonc^2r_1^{-11+2\xi(\mathbf{1})-6\varepsilon_0}.
\end{align*}
For $\I_{\Lb3}$, we have two cases:
\begin{equation*}
|\I_{L3}| \leq\begin{cases} {\epsilonc} |u|^{-5-4\varepsilon_0}r|\alphab|\lesssim {\epsilonc} |u|^{-5-4\varepsilon_0}r^{-2}+{\epsilonc}^\frac{3}{2} |u|^{-8-5\varepsilon_0},  \ \text{if} \ Z=\Omega;\\
 {\epsilonc} |u|^{-5-4\varepsilon_0}|v|^{1+\xi(\mathbf{1})}|\rho|\lesssim |u|^{-5-4\varepsilon_0}r^{-1+\xi(\mathbf{1})} ,\ \text{if} \ Z=v^{1+\xi(\mathbf{1})}L+u^{1+\xi(\mathbf{1})}\Lb.
\end{cases}
\end{equation*}
In both cases, we can simply directly integrate the pointwise bounds on $\Hb_{r_1}^{r_2}$ to obtain bound $r_1^{\frac{3}{2}}\int_{\Hb_{r_1}^{r_2}}$ by $\epsilonc^2r_1^{-11+2\xi(\mathbf{1})-2\varepsilon_0}$.
Hence, we conclude that
\begin{align*}
\sup_{r_2\geq r_1}r_1^{\frac{3}{2}}\int_{\Hb_{r_1}^{r_2}}  \frac{|J^{(\mathbf{1})}_\Lb|^2}{r^{\frac{7}{2}}} &\lesssim \epsilonc^2r_1^{-10+2\xi(\mathbf{1})-2\varepsilon_0}.
\end{align*}

Putting all the estimates together, we obtain that
\begin{align*}
& r_1^{-2}\int_{\H_{r_1}}{|J^{(\mathbf{1})}_L|^2}+\int_{\H_{r_1}}\frac{|\slashed{J}^{(\mathbf{1})}|^2}{r^2}+\sup_{r_2 \geq r_1}\int_{\Hb_{r_2}^{r_1}} \frac{|\slashed{J}^{(\mathbf{1})}|^2}{r^2} +\sup_{r_2 \geq r_1}r_1^{\frac{3}{2}}\int_{\Hb_{r_2}^{r_1}} \frac{|J^{(\mathbf{1})}_\Lb|^2}{r^{\frac{7}{2}}}\lesssim {\epsilonc}^2 r_1^{-10+2\xi(\k)-2\varepsilon_0}.
\end{align*}
Thus, for sufficiently large $R_*$, it immediately closes $(\mathbf{C})$ for $\k=(\mathbf{1})$.

\medskip

\underline{$\blacktriangleright$~~Second order estimates}

\medskip

For $\k=(\mathbf{2})$, we assume that the vector fields $Z_1$ and $Z_2$ represent this index. We also use $(\mathbf{1})$ to denote $Z_1$ and $(\mathbf{1'})$ to denote $Z_2$. We first derive the a formula for $J^{(\mathbf{2})} =\mathcal{L}_{Z_1}\mathcal{L}_{Z_2}(\Im(\overline{\psi}\cdot D\psi ))$. Indeed, we have
\begin{align*}
\mathcal{L}_{Z_1}\mathcal{L}_{Z_2}(\overline{\psi}\cdot D\psi )_\mu&= \overline{\psi} \cdot D_\mu D_{Z_1} D_{Z_2}\psi+ \overline{ D_{Z_1} D_{Z_2}\psi}\cdot D_\mu\psi + \overline{D_{Z_2}\psi} \cdot D_\mu D_{Z_1} \psi+ \overline{D_{Z_1}\psi} \cdot D_\mu D_{Z_2} \psi\\
&+ 2i\Re\big(\overline{D_{Z_{2}}\psi}\cdot \psi\big) F_{Z_1 \mu}+i\big(\mathcal{L}_{Z_1}F\big)_{Z_2 \mu}|\psi|^2+iF_{[Z_2,Z_1]\mu}|\psi|^2+iF_{Z_2\mu}Z_1(|\psi|^2).
\end{align*}
Hence,
\begin{align*}
J^{(\mathbf{2})}_\mu&= \Im\Big(\overline{\psi} \cdot D_\mu \psi^{(\mathbf{2})}	+ \overline{ \psi^{(\mathbf{2})}}\cdot D_\mu\psi + \overline{\psi^{(\mathbf{1'})}} \cdot D_\mu \psi^{(\mathbf{1})}+ \overline{\psi^{(\mathbf{1})}} \cdot D_\mu \psi^{(\mathbf{1'})} \Big)\\
& \ \ + 2\Re\big(\overline{\psi^{(\mathbf{1'})}}\cdot \psi\big) F_{Z_1 \mu}+2\Re\big(\overline{\psi^{(\mathbf{1})}}\cdot \psi\big) F_{Z_2 \mu}+\big(\mathcal{L}_{Z_1}F\big)_{Z_2 \mu}|\psi|^2+F_{[Z_2,Z_1]\mu}|\psi|^2.
\end{align*}
In view of the symmetry of the indices $(\mathbf{1})$ and $(\mathbf{1'})$, we may drop the terms with similar structures and bound $J^{(\mathbf{2})}_\mu$ as follows:
\begin{align*}
|J^{(\mathbf{2})}_L|&\lesssim |\psi| |D_\mu \psi^{(\mathbf{2})}|	+ |\psi^{(\mathbf{2})}||D_\mu\psi| +|\psi^{(\mathbf{1'})}||D_\mu \psi^{(\mathbf{1})}|+ |\psi||\psi^{(\mathbf{1'})}||F_{Z_1 \mu}|+\big(|\big(\mathcal{L}_{Z_1}F\big)_{Z_2 \mu}|+|F_{[Z_2,Z_1]\mu}|\big)|\psi|^2.
\end{align*}
In order to bound $|J^{(\mathbf{2})}_L|$ in an efficient way, we first bound $\psi^{(\mathbf{1'})}$ in $L^\infty$. We have
\begin{equation*}
|\psi^{(\mathbf{1'})}|\leq\begin{cases} r^2|\slashed{D}\varphi|,  \ &\text{if} \ Z=\Omega;\\
 |v^{1+\xi(\mathbf{1'})}||D_L \psi|+ |u^{1+\xi(\mathbf{1'})}|\big(|rD_\Lb \phi|+|\phi|\big),\ &\text{if} \ Z=v^{1+\xi(\mathbf{1})}L+u^{1+\xi(\mathbf{1})}\Lb.
\end{cases}
\end{equation*}
By virtue of the pointwise bounds on $D\phi$, we have	
\begin{equation*}
|\psi^{(\mathbf{1'})}| \lesssim \sqrt{\epsilonc}|u|^{-\frac{3}{2}+\xi(\mathbf{1'})-2\varepsilon_0}.
\end{equation*}
We can now bound $\phi$ and $\psi^{(\mathbf{1'})}$ in $J^{(\mathbf{2})}_\mu$ to derive
\begin{equation}\label{J2 pointwise bound}
\begin{split}
|J^{(\mathbf{2})}_\mu|&\lesssim \Big(\sqrt{\epsilonc}|u|^{-\frac{5}{2}-2\varepsilon_0} |D_\mu \psi^{(\mathbf{2})}|	 +\sqrt{\epsilonc}|u|^{-\frac{3}{2}+\xi(\mathbf{1'})-2\varepsilon_0}|D_\mu \psi^{(\mathbf{1})}|\Big)
 +|\psi^{(\mathbf{2})}||D_\mu\psi|\\
 &\ \ \ \ \ \ \ \ \ \ +\Big( {\epsilonc}|u|^{-4+\xi(\mathbf{1'})-4\varepsilon_0}|F_{Z_1 \mu}|+{\epsilonc}|u|^{-5-4\varepsilon_0}\big(|\big(\mathcal{L}_{Z_1}F\big)_{Z_2 \mu}|+|F_{[Z_2,Z_1]\mu}|\big)\Big).
 \end{split}
\end{equation}

$\blacklozenge$~For $J^{(\mathbf{2})}_L$, we can use the pointwise bound for $D_L\psi$ in \eqref{J2 pointwise bound} (where $\mu = L$) and we obtain
\begin{align*}
|J^{(\mathbf{2})}_L|&\lesssim \underbrace{\Big(\sqrt{\epsilonc}|u|^{-\frac{5}{2}-2\varepsilon_0} |D_L \psi^{(\mathbf{2})}|	 +\sqrt{\epsilonc}|u|^{-\frac{3}{2}+\xi(\mathbf{1'})-2\varepsilon_0}|D_L \psi^{(\mathbf{1})}|\Big)}_{\II_{L2}}
 +\underbrace{\sqrt{\epsilonc}r^{-2}|u|^{-1-\varepsilon_0}|\psi^{(\mathbf{2})}|}_{\II_{L1}}\\
 &\ \ \ \ \ \ \ \ \ \ +\underbrace{\Big( {\epsilonc}|u|^{-4+\xi(\mathbf{1'})-4\varepsilon_0}|F_{Z_1 L}|+{\epsilonc}|u|^{-5-4\varepsilon_0}|F_{[Z_2,Z_1]L}|\Big)}_{\II_{L3}}+\underbrace{{\epsilonc}|u|^{-5-4\varepsilon_0}|\big(\mathcal{L}_{Z_1}F\big)_{Z_2 L}|}_{\II_{L4}}.
\end{align*}
For $\II_{L1}$, in view of \eqref{L^2 estimates on spheres for k=1 2}, we have
\begin{align*}
r_1^{-2}\int_{\H_{r_1}}  |\II_{L1}|^2&\lesssim \epsilonc r_1^{-2}  |u|^{-2-2\varepsilon_0} \int_{r_1}^{r_2}r^{-2}\big(\int_{\S_{r_1}^r}|\phi^{(\mathbf{2})}|^2\big) dr \lesssim \epsilonc^2r_1^{-10+2\xi(\mathbf{2})-4\varepsilon_0}.
\end{align*}
For $\II_{L2}$, by the $r^p$-weighted energy estimates, we have
\begin{align*}
r_1^{-2}\int_{\H_{r_1}}  |\II_{L2}|^2&\lesssim \epsilonc r_1^{-2}  \int_{\H_{r_1}}|u|^{-5-4\varepsilon_0} |D_L \psi^{(\mathbf{2})}|^2+ |u|^{-3+2\xi(\mathbf{1'})-4\varepsilon_0} |D_L \psi^{(\mathbf{1})}|^2\lesssim \epsilonc^2r_1^{-9+2\xi(\mathbf{2})-6\varepsilon_0}.
\end{align*}
For $\II_{L3}$, we will need the following two inequalities:
\begin{equation*}
|F_{Z_1L}|\lesssim r^{-2}|u|^{1+\xi(\mathbf{1})}, \ \ |F_{[Z_1,Z_2]L}|\lesssim r^{-2}|u|^{1+\xi(\mathbf{2})}.
\end{equation*}
The first one can be checked by a direction computation. For the second, we notice that the only non-vanishing $[Z_1,Z_2]$'s for $Z_1,Z_2\in \mathcal{Z}$ are $[T,S]=T$, $[T,K]=2S$ and $[S,K]=K$. For those vector fields, it is clear that $\xi([Z_1,Z_2])=\xi(Z_1)+\xi(Z_2)$. Therefore, the second inequality follows from the first one. In particular, we have
\begin{align*}
|\II_{L3}| \lesssim {\epsilonc} r^{-2}|u|^{-3+\xi(\mathbf{2})-4\varepsilon_0}.
\end{align*}
We can integrate this pointwise bound on $\H_{r_1}$ to bound $r_1^{-2}\int_{\H_{r_1}}|\II_{L3}|^2$ by $\lesssim \epsilonc^2r_1^{-9+2\xi(\mathbf{2})-8\varepsilon_0}$.

\noindent For $\II_{L4}$,  we have two cases
\begin{equation*}
|\big(\mathcal{L}_{Z_1}F\big)_{\Omega L}|\leq\begin{cases} r|\alphaone| + r^{-2}|u|^{\xi(\mathbf{1})},  \ &\text{if} \ Z_2=\Omega;\\
 	|u|^{1+\xi(\mathbf{1'})}\big(|\rhoone| + r^{-3}|u|^{\xi(\mathbf{1})}\big),\ &\text{if} \ Z_2=v^{1+\xi(\mathbf{1'})}L+u^{1+\xi(\mathbf{1'})}\Lb.
\end{cases}
\end{equation*}
For the first case, we have \begin{align*}
r_1^{-2}\int_{\H_{r_1}}  |\II_{L4}|^2&\lesssim r_1^{-2} \epsilonc^2|u|^{-10-4\varepsilon_0}\int_{\H_{r_1}} r^2|\alphaone|^2+r^{-4}|u|^{2\xi(\mathbf{1})}\lesssim \epsilonc^2r_1^{-13+2\xi(\mathbf{1})-4\varepsilon_0}.
\end{align*}
For the second case,  we have \begin{align*}
r_1^{-2}\int_{\H_{r_1}}  |\II_{L4}|^2&\lesssim r_1^{-2} \epsilonc^2|u|^{-8+2\xi(\mathbf{1'})-4\varepsilon_0}\int_{\H_{r_1}} |\rhoone|^2+r^{-6}|u|^{2\xi(\mathbf{1})}\lesssim \epsilonc^2r_1^{-13+2\xi(\mathbf{2})-4\varepsilon_0}.
\end{align*}
Hence, $r_1^{-2}\int_{\H_{r_1}}  |\II_{L4}|^2$ is bounded by $\epsilonc^2r_1^{-13+2\xi(\mathbf{2})-4\varepsilon_0}$
Together with previous estimates, we obtain
\begin{align*}
r_1^{-2}\int_{\H_{r_1}}  |J^{(\mathbf{2})}_L|^2&\lesssim \epsilonc^2r_1^{-9+2\xi(\mathbf{2})-6\varepsilon_0}.
\end{align*}

$\blacklozenge$~For $\slashed{J}^{(\mathbf{2})}$, we bound $\slashed{D}\psi$ in \eqref{J2 pointwise bound} in $L^\infty$ (where $\mu = e_A$) and we obtain
\begin{align*}
|J^{(\mathbf{2})}_\mu|&\lesssim \underbrace{\Big(\sqrt{\epsilonc}r|u|^{-\frac{5}{2}-2\varepsilon_0} |\slashed{D} \phi^{(\mathbf{2})}|	 +\sqrt{\epsilonc}r|u|^{-\frac{3}{2}+\xi(\mathbf{1'})-2\varepsilon_0}|\slashed{D}\phi^{(\mathbf{1})}|\Big)}_{\slashed{\II}_{2}}
 +\underbrace{\sqrt{\epsilonc}|u|^{-2-\varepsilon_0}|\phi^{(\mathbf{2})}|}_{\slashed{\II}_{1}}\\
 &\ \ \ \ \ \ \ \ \ \ +\underbrace{\Big( {\epsilonc}|u|^{-4+\xi(\mathbf{1'})-4\varepsilon_0}|F_{Z_1 A}|+{\epsilonc}|u|^{-5-4\varepsilon_0}|F_{[Z_2,Z_1]A}|\Big)}_{\slashed{\II}_{3}}+\underbrace{{\epsilonc}|u|^{-5-4\varepsilon_0}|\big(\mathcal{L}_{Z_1}F\big)_{Z_2 A}|}_{\slashed{\II}_{4}}.
\end{align*}
For $\slashed{\II}_{1}$, according to \eqref{L^2 estimates on spheres for k=1 2}, we have
\begin{align*}
\int_{\H_{r_1}}  \frac{|\slashed{\II}_{1}|^2}{r^2}&\lesssim \epsilonc   \int_{r_1}^{r_2}r^{-2}|u|^{-4-2\varepsilon_0} \big(\int_{\S_{r_1}^r}|\phi^{(\mathbf{2})}|^2\big) dr \lesssim \epsilonc^2r_1^{-10+2\xi(\mathbf{2})-2\varepsilon_0}.
\end{align*}
For $\slashed{\II}_{2}$, we have
\begin{align*}
\int_{\H_{r_1}}  \frac{|\slashed{\II}_{2}|^2}{r^2}&\lesssim \epsilonc   |u|^{-5-4\varepsilon_0} \int_{\H_{r_1}}|\slashed{D} \phi^{(\mathbf{2})}|^2+ \epsilonc   |u|^{-5-4\varepsilon_0} \int_{\Hb_{r_1}^{r_2}}|\slashed{D} \phi^{(\mathbf{2})}|^2\\
& \ \ + \epsilonc   |u|^{-3+2\xi(\mathbf{1'})-4\varepsilon_0} \int_{\H_{r_1}}|\slashed{D} \phi^{(\mathbf{1})}|^2+ \epsilonc   |u|^{-3+2\xi(\mathbf{1'})-4\varepsilon_0} \int_{\Hb_{r_1}^{r_2}}|\slashed{D} \phi^{(\mathbf{1})}|^2\lesssim \epsilonc^2r_1^{-9+2\xi(\mathbf{2})-8\varepsilon_0}.
\end{align*}
For $\slashed{\II}_{3}$, since
\begin{equation*}
|F_{Z_1A}|\lesssim \sqrt{\epsilonc}r^{-1}|u|^{-2+\xi(\mathbf{1})-\varepsilon_0}+r^{-2+\xi(\mathbf{1})},\ \
|F_{[Z_1,Z_2]A}|\lesssim \sqrt{\epsilonc}r^{-1}|u|^{-2+\xi(\mathbf{2})-\varepsilon_0}+r^{-2+\xi(\mathbf{2})},
\end{equation*}
we can integrate these pointwise bounds to derive
\begin{align*}
\int_{\H_{r_1}}  \frac{|\slashed{\II}_{3}|^2}{r^2}
&\lesssim \epsilonc^2r_1^{-11+2\xi(\mathbf{2})-8\varepsilon_0}.
\end{align*}
For $\slashed{\II}_{4}$,  we have two cases
\begin{equation*}
|\big(\mathcal{L}_{Z_1}F\big)_{\Omega A}|\leq\begin{cases} r|\sigmaone|,  \ &\text{if} \ Z_2=\Omega;\\
 	r^{1+\xi(\mathbf{1'})}|\alphaone| + |u|^{1+\xi(\mathbf{1'})}|\alphabone|+r^{-2+\xi(\mathbf{1'})}|u|^{\xi(\mathbf{1})},\ &\text{if} \ Z_2=v^{1+\xi(\mathbf{1'})}L+u^{1+\xi(\mathbf{1'})}\Lb.
\end{cases}
\end{equation*}
We claim that in both cases we all have
\begin{align*}
	\int_{\H_{r_1}}  \frac{|\slashed{\II}_{4}|^2}{r^2}&\lesssim \epsilonc^2r_1^{-16+2\xi(\mathbf{2})-8\varepsilon_0}.
\end{align*}
The proof for the first case is straightforward.
For the second case,  we need the following bound on $\alphaone$:
\begin{equation}\label{L2 sphere bound on alphabone}
\|\alphabone\|_{L^2(\S_{r_1}^{r_2})}\lesssim  \sqrt{\epsilonc}r_1^{-3+\xi(\mathbf{1})-\varepsilon_0}.
\end{equation}
In fact,
\begin{align*}
 \|\alphabone\|_{L^2(\S_{r_1}^{r_2})}^2- \|\alphabone\|_{L^2(\S_{r_1}^{r_1})}^2&= \int_{\frac{r_2}{2}}^{\frac{r_1}{2}}\int_{\mathbf{S}^2}L\big(\big| r \alphabone \big|^2\big)d\vartheta dv  \lesssim  \int_{\frac{r_2}{2}}^{\frac{r_1}{2}}\int_{\mathbf{S}^2} |\slashed{\nabla}_L\big(r\alphabone \big)|| r\alphabone|d\vartheta dv\\
 &\lesssim  \|\slashed{\nabla}_L(r\alphabone)\|_{L^2(\H_{r_1}^{r_2})} \Big(\int_{\frac{r_2}{2}}^{\frac{r_1}{2}}\frac{1}{r^2}\big(\int_{\S_{r_1}^{r}} |\alphabone|^2\big)dr\Big)^\frac{1}{2}.
 \end{align*}
 To bound $\|\slashed{\nabla}_L(r\alphabone)\|_{L^2(\H_{r_1}^{r_2})}$, according to \eqref{Maxwell null commuted} and the facts that $r|\slashed{\nabla}\rho(\mathcal{L}_{Z_1} \Fc)|\simeq |\rho(\mathcal{L}_{\Omega}\mathcal{L}_{Z_1}\Fc )|$ and $r|\slashed{\nabla}\sigma(\mathcal{L}_{Z_1} \Fc)|\simeq |\sigma(\mathcal{L}_{\Omega}\mathcal{L}_{Z_1}\Fc )|$, we have
 \begin{align*}
 \|\slashed{\nabla}_L(r\alphabone)\|_{L^2(\H_{r_1}^{r_2})}^2 &\leq \int_{\H_{r_1}^{r_2}}|\rho(\mathcal{L}_{\Omega}\mathcal{L}_{Z_1}\Fc )|^2+|\sigma(\mathcal{L}_{\Omega}\mathcal{L}_{Z_1}\Fc )|^2+\frac{|\slashed{J}^{(\mathbf{1})}|^2}{r^2}\lesssim \epsilonc r_1^{-6+2\xi(\mathbf{1})-4\varepsilon_0}.
 \end{align*}
Similar to the proof of Lemma \ref{lemma A7}, we use Gronwall's inequality to obtain \eqref{L2 sphere bound on alphabone}. Thus,
\begin{align*}
\int_{\H_{r_1}}  \frac{|\slashed{\II}_{4}|^2}{r^2}
&\lesssim {\epsilonc}^2\int_{\H_{r_1}} |u|^{-10-8\varepsilon_0}\Big(r^{2\xi(\mathbf{1'})} {|\alphaone|^2} + |u|^{2+2\xi(\mathbf{1'})}\frac{|\alphabone|^2}{r^2}+r^{-6+2\xi(\mathbf{1'})}|u|^{2\xi(\mathbf{1})}\Big)	
\end{align*}
The first term can be bounded by the $r^p$-weighted energy estimates. The last term can bounded directly. For the second term, we use \eqref{L2 sphere bound on alphabone} to get its $L^2$ bound on $\S_{r_1}^{r_2}$ then integrate over $r$. This proves the estimate for $\slashed{\II}_{4}$. Together with other estimates, we obtain
\begin{align*}
\int_{\H_{r_1}}  \frac{|\slashed{\II}|^2}{r^2}&\lesssim \epsilonc^2r_1^{-9+2\xi(\mathbf{2})-6\varepsilon_0}.
\end{align*}

$\blacklozenge$~For $J^{(\mathbf{2})}_\Lb$, we can use the pointwise bound for $D_L\psi$ in \eqref{J2 pointwise bound} (where $\mu = L$) and we obtain
\begin{align*}
|J^{(\mathbf{2})}_\mu|&\lesssim \underbrace{\Big(\sqrt{\epsilonc}|u|^{-\frac{5}{2}-2\varepsilon_0} |D_\Lb \psi^{(\mathbf{2})}|	 +\sqrt{\epsilonc}|u|^{-\frac{3}{2}+\xi(\mathbf{1'})-2\varepsilon_0}|D_\Lb \psi^{(\mathbf{1})}|\Big)}_{\II_{\Lb2}}
 +\underbrace{\sqrt{\epsilonc}|u|^{-3-\varepsilon_0}|\psi^{(\mathbf{2})}|}_{\II_{\Lb1}}\\
 &\ \ \ \ \ \ \ \ \ \ +\underbrace{\Big( {\epsilonc}|u|^{-4+\xi(\mathbf{1'})-4\varepsilon_0}|F_{Z_1 \Lb}|+{\epsilonc}|u|^{-5-4\varepsilon_0}|F_{[Z_2,Z_1]\Lb}|\Big)}_{\II_{\Lb3}}+\underbrace{{\epsilonc}|u|^{-5-4\varepsilon_0}|\big(\mathcal{L}_{Z_1}F\big)_{Z_2 \Lb}|}_{\II_{\Lb4}}.
\end{align*}
For $\II_{\Lb1}$, according to \eqref{L^2 estimates on spheres for k=1 2}, we have
\begin{align*}
r_1^{\frac{3}{2}}\int_{\Hb_{r_1}^{r_2}}  \frac{|{\II}_{\Lb1}|^2}{r^{\frac{7}{2}}}&\lesssim \epsilonc   r_1^{\frac{3}{2}}\int_{r_1}^{r_2}r^{-\frac{3}{2}}r_1^{-6-2\varepsilon_0} \big(\int_{\S_{r_1}^r}|\phi^{(\mathbf{2})}|^2\big) dr \lesssim \epsilonc^2r_1^{-11+2\xi(\mathbf{2})-6\varepsilon_0}.
\end{align*}
For $\II_{\Lb2}$, we first notice that $|D_\Lb \psi^{(\mathbf{k})}|\lesssim r |D_\Lb \phi^{(\mathbf{k})}|+|\phi^{(\mathbf{k})}|$. The contribution from $|\phi^{(\mathbf{1})}|$ and  $|\phi^{(\mathbf{2})}|$ can be ignored since they have been already treated in $\I_{\Lb1}$ and $\II_{\Lb1}$. Thus, modulo those terms, we have
\begin{align*}
r_1^{\frac{3}{2}}\int_{\Hb_{r_1}^{r_2}}  \frac{|{\II}_{\Lb2}|^2}{r^{\frac{7}{2}}}&\lesssim \epsilonc   r_1^{\frac{3}{2}}\int_{\Hb_{r_1}^{r_2}} r^{-\frac{3}{2}}\big(|u|^{-5-4\varepsilon_0}|D_\Lb\phi^{(\mathbf{2})}|^2+ |u|^{-3+2\xi(\mathbf{1'})-4\varepsilon_0}|D_\Lb\phi^{(\mathbf{1})}|^2\big)	\lesssim \epsilonc^2r_1^{-9+2\xi(\mathbf{1})-6\varepsilon_0}.
\end{align*}
For $\II_{\Lb3}$, we notice that
\begin{equation*}
|F_{Z_1\Lb}|\lesssim \begin{cases} \sqrt{\epsilonc}|u|^{-3-\varepsilon_0}+r^{-2}, \ \ &\text{if} \ \ Z_1=\Omega;\\
r^{-1+\xi(\mathbf{1})},\ \ &\text{if} \ \ Z_1=v^{1+\xi(\mathbf{1})}L+u^{1+\xi(\mathbf{1})}\Lb.
\end{cases}.
\end{equation*}
Since $\Omega$ can not be a commutator $[Z_1,Z_2]$, we have
\begin{equation*}
|F_{[Z_1,Z_2]\Lb}|\lesssim r^{-1}|u|^{1+\xi(\mathbf{2})}.
\end{equation*}
We remark that the $\xi(\mathbf{2})$ in the above formula is at most $1$.
Thus, we can directly integrate those pointwise bounds and we obtain
\begin{align*}
r_1^{\frac{3}{2}}\int_{\Hb_{r_1}^{r_2}}  \frac{|\II_\Lb|^2}{r^{\frac{7}{2}}} &\lesssim \epsilonc^2r_1^{-9+2\xi(\mathbf{1})-2\varepsilon_0}.
\end{align*}
For $\II_{\Lb4}$,  we have two cases
\begin{itemize}
\item If $Z_2=\Omega$, by \eqref{onederivative of F}, we have
\begin{equation*}
|\big(\mathcal{L}_{Z_1}F\big)_{\Omega \Lb}|\lesssim	r|\alphabone| + r^{-2+\xi(\mathbf{1})}.
\end{equation*}
Thus, \begin{align*}
r_1^{\frac{3}{2}}\int_{\Hb_{r_1}^{r_2}}  \frac{|\II_{\Lb4}|^2}{r^{\frac{7}{2}}} &\lesssim \epsilonc^2r_1^{\frac{3}{2}}\int_{\Hb_{r_1}^{r_2}}   r^{-\frac{3}{2}}|u|^{-10-8\varepsilon_0}|\alphabone|^2+r^{-\frac{13}{2}+2\xi(\mathbf{1})}|u|^{-10-8\varepsilon_0}.
\end{align*}
\item If $Z_2=v^{1+\xi(\mathbf{1'})}L+u^{1+\xi(\mathbf{1'})}\Lb$, by \eqref{onederivative of F}, we have
\begin{align*}
|\big(\mathcal{L}_{Z_1}F\big)_{Z_2 \Lb}|\lesssim	r^{1+\xi(\mathbf{1'})}|\rhoone| + r^{-2+\xi(\mathbf{2})}.
\end{align*}
Thus, \begin{align*}
r_1^{\frac{3}{2}}\int_{\Hb_{r_1}^{r_2}}  \frac{|\II_{\Lb4}|^2}{r^{\frac{7}{2}}} &\lesssim \epsilonc^2r_1^{\frac{3}{2}}\int_{\Hb_{r_1}^{r_2}}   r^{2\xi(\mathbf{1'})-\frac{7}{2}}|u|^{-10-8\varepsilon_0}r^2|\rhoone|^2+r^{-\frac{15}{2}+2\xi(\mathbf{2})}|u|^{-10-8\varepsilon_0}.
\end{align*}
\end{itemize}
In both cases, we have
\begin{align*}
r_1^{\frac{3}{2}}\int_{\Hb_{r_1}^{r_2}}  \frac{|{\II}_{\Lb4}|^2}{r^{\frac{7}{2}}} &\lesssim \epsilonc^2r_1^{-12+2\xi(\mathbf{1})-4\varepsilon_0}.
\end{align*}
Hence, we conclude that
\begin{align*}
\sup_{r_2\geq r_1}r_1^{\frac{3}{2}}\int_{\Hb_{r_1}^{r_2}}  \frac{|J^{(\mathbf{2})}_\Lb|^2}{r^{\frac{9}{2}}} &\lesssim \epsilonc^2r_1^{-10+2\xi(\mathbf{1})-2\varepsilon_0}.
\end{align*}

Putting all the estimates together, we obtain that
\begin{align*}
& r_1^{-2}\int_{\H_{r_1}}{|J^{(\mathbf{2})}_L|^2}+\int_{\H_{r_1}}\frac{|\slashed{J}^{(\mathbf{2})}|^2}{r^2}+\sup_{r_2 \geq r_1}r_1^{\frac{3}{2}}\int_{\Hb_{r_2}^{r_1}} \frac{|J^{(\mathbf{2})}_\Lb|^2}{r^{\frac{7}{2}}}\lesssim {\epsilonc}^2 r_1^{-9+2\xi(\k)-2\varepsilon_0}.
\end{align*}
Thus, for sufficiently large $R_*$, it immediately closes $(\mathbf{C})$ for $\k=(\mathbf{2})$.

\section{The analysis in the interior region}
\subsection{The conformal theory of Maxwell-Klein-Gordon equations}
We review the conformal theory of the Maxwell-Klein-Gordon equations.  We refer the readers to Chapter 4 of the the booklet \cite{Christodoulou:Book:GR1} of Christodoulou for more details. Let $\mathbf{L}$ be a line bundle over a four dimension Lorentzian manifold $(M,g)$ with a given hermitian metric $h$. Let $D_A$ be a $\mathbf{U}(1)$-connection compatible with $h$ where $A$ is the corresponding connection 1-form and let $\phi$ be a section of $\mathbf{L}$. The action for Maxwell-Klein-Gordon equations is
\begin{equation*}
\mathcal{L}(A,\phi;g,h)=\int_{M} \underbrace{\frac{1}{2}g^{\mu\nu}h\big((D_A)_{\partial_\mu} \phi,(D_A)_{\partial_\nu}\phi\big)}_{L_s(A,\phi;g,h)} +\underbrace{\frac{1}{4}F^{\mu\nu}F_{\mu\nu}}_{L_m(A,\phi;g,h)} d\text{vol}_g.
\end{equation*}
Let $\Lambda,\lambda$ be two positive smooth functions on $M$. We conformally change the metrics $g$ and $h$ by the following rules:
\begin{equation*}
\widetilde{g} = \Lambda^2 g, \ \ \widetilde{h} = \lambda^2 h.
\end{equation*}
The $\widetilde{D_A}\phi=D_A \phi +(d\log\gamma)\phi$ is a connection compatible with $\widetilde{h}$ ($A$ is viewed as a given $1$-form which gives a connection compatible with $\widetilde{h}$ via this formula). We remark that this is not a gauge transformation. The action in the new conformal settings is related to the old one by the following formulae:
\begin{align*}
L_s(A,\phi;g,h)&=\frac{\Lambda^2}{\lambda^2}\Big(L_s(A,\phi;\widetilde{g},\widetilde{h})+\frac{1}{2}|\phi|_{\widetilde{h}}^2\gamma^{-1}\Box_{\widetilde{g}}\gamma-\frac{1}{2}\text{Div}_{\widetilde{g}}\big(|\phi|^2_{\widetilde{h}}\cdot\text{grad}_{\widetilde{g}}(\log\gamma)\big)\Big),\\
L_m(A,\phi;g,h)&=\Lambda^4 L_m(A,\phi;\widetilde{g},\widetilde{h}).
\end{align*}
We also notice that $d\text{vol}_g = \Lambda^{-4} d\text{vol}_{\widetilde{g}}$. By setting $\lambda=\Lambda^{-1}$, we obtain
\begin{equation*}
\mathcal{L}(A,\phi;g,h)=\mathcal{L}(A,\phi;\widetilde{g},\widetilde{h})+\int_M \frac{1}{2}|\phi|_{\widetilde{h}}^2\gamma^{-1}\Box_{\widetilde{g}}\gamma-\frac{1}{2}\text{Div}_{\widetilde{g}}\big(|\phi|^2_{\widetilde{h}}\cdot\text{grad}_{\widetilde{g}}(\log\gamma)\big)\Big).
\end{equation*}
We now impose a condition on $\gamma$: $\Box_{\widetilde{g}}\gamma =0$. In applications, we will take $\widetilde{g}$ to be the Minkowski metric and $\gamma(t,x)=\big((t+C)^2-|x|^2\big)^{-1}$ where $C$ is a constant so that this condition is always satisfied. Thus, we conclude that the difference between the two actions  $\mathcal{L}(A,\phi;g,h)$ and $\mathcal{L}(A,\phi;\widetilde{g},\widetilde{h})$ is simply a divergence term. In particular, this implies that the two actions give the same Euler-Lagrange equations. Therefore $(A,\phi)$ being a solution of \eqref{MKG} with $(g,h)$ is equivalent to being a solution of \eqref{MKG} with $(\widetilde{g},\widetilde{h})$.  We point out that the two sets of (identical) solutions need to be measured in different metrics.

We study a special case. Let ${\Phi}:(U,m) \rightarrow (\widetilde{U},\widetilde{m})$ be a conformal mapping between two domains of Minkowski space. We assume that
\begin{equation*}
{\Phi}^*\widetilde{m}_{\mu\nu} = \Lambda^{-2} m_{\mu\nu}.
\end{equation*}
where $\Lambda(t,x) = (t+R_*+1)^2-|x|^2$ is a smooth function on $U$. By setting $\widetilde{g}=m$ and $g={\Phi}^*\widetilde{m}$, we apply the previous constructions. This yields the following lemma:
\begin{lemma}
Let $\widetilde{\mathbf{L}}$ be a complex line bundle on $\widetilde{U}$ and ${\Phi}:U \rightarrow \widetilde{U}$ be a conformal diffeomorphism described above. If $(\widetilde{\phi},\widetilde{A})$ is a solution of \eqref{MKG} for $\widetilde{\mathbf{L}}$, then $(\phi,A):=({\Phi}^*\widetilde{\phi},{\Phi}^*\widetilde{A})$ is solution of \eqref{MKG} for $\mathbf{L}={\Phi}^*\widetilde{\mathbf{L}}$ with respect to $(m,\Lambda^{-2}{\Phi}^*\widetilde{h})$. In particular, one takes $h= {\Phi}^*\widetilde{h}$ on $\mathbf{L}$ (this is always the case since we usually identify scalar fields with a $\mathbb{C}-$valued functions by fixing a unit section in $\mathbf{L}$ or $\widetilde{\mathbf{L}}$). We can reformulate the statement as follows: If $(\widetilde{\phi},\widetilde{A})$ ($\phi$ is a complex function) a solution of \eqref{MKG} on $\widetilde{U}$, then $(\Lambda^{-1} {\Phi}^*\widetilde{\phi}, {\Phi}^*\widetilde{A})$ is also a solution of  \eqref{MKG} on $U$.
\end{lemma}
We can also reverse the direction:
\begin{corollary}
 If $(\phi,A)$ is a solution of \eqref{MKG} on ${U}$, then $(\big({\Phi}^{-1}\big)^*\big(\Lambda \cdot \phi\big), \big({\Phi}^{-1}\big)^*{A})$ is also a solution of  \eqref{MKG} on $\widetilde{U}$.
\end{corollary}

\subsection{The conformal picture}\label{section conformal picture}
From now on until the end of the paper, the radius $R_*$ is fixed. And in the sequel we allow the implicit constant depends also on $R_*$ and the size of the data $C_0$. More precisely, $B\les P$ means that there is a constant $C$ depending only on $C_0$ such that $B\leq CP$. Here we point out that the radius $R_*$ as well as the charge $q_0$ only depends on the size of the data $C_0$. We define a hyperboloid $\Sigma$ in the Minkowski space by the following equation:
\begin{equation*}
\Sigma =\Big\{ (t,x) \Big| -\big(t+R_*+1-\frac{2R_*+1}{2(R_*+1)}\big)^2+|x|^2=-\Big( \frac{2R_*+1}{2(R_*+1)}\Big)^2\Big\}.
\end{equation*}
Geometrically, $\Sigma$ (drawn as the bold black curve in the left figure) is the unique hyperboloid passing through $\S_{R_*}^{R_*}$ and asymptotic to $\H_{R_*+\frac{1}{2R_*+2}}$ (the dash-dot-dot line in the left figure). We denote its causal future by $\mathcal{J}^+(\Sigma)$ and it is the grey region in the left figure.

\ \ \ \ \ \ \ \ \ \ \ \ \ \ \ \ \ \ \includegraphics[width=5in]{conformal.pdf}

We now define a map ${\Phi}: \mathcal{J}^+(\Sigma) \rightarrow \mathbf{R}^{3+1}$. To distinguish the domain and the target of the map, we will use $\widetilde{t}$ and $\widetilde{x_i}$'s as coordinate system on the target Minkowski space $(\mathbb{R}^{3+1},\widetilde{m}_{\alpha\beta})$ where $\widetilde{m}_{\alpha\beta}$ is the Minkowski metric on the target. The map ${\Phi}$ is given by the following formula:
\begin{equation*}
{\Phi}: (t,x)\mapsto (\widetilde{t},\widetilde{x})=\Big(-\frac{t+R_*+1}{(t+R_*+1)^2-|x|^2}+\frac{R_*+1}{2R_*+1},\frac{x}{(t+R_*+1)^2-|x|^2}\Big)
\end{equation*}
Geometrically, it is the composition of a time translation with the standard inversion map centered at $(-R_*-1,0,0,0)$.  It is straightforward to see that the image of $\Sigma$ is given by $\widetilde{t}=0$ and $|\widetilde{x}|<\frac{R_*+1}{2R_*+1}$. It is a 3-dimensional open ball denoted by $\widetilde{\B}_{\frac{R_*+1}{2R_*+1}}$ (see the bold straight line segment in the right figure). We also define $\Sigma_\pm = \Sigma \cap \{\pm t\geq 0\}$ and their images under ${\Phi}$ are denoted by $\B_{\pm}$. We remark that the $\Sigma_+$ are entirely inside the exterior region where we have already obtained good control on the solutions. The image of $\mathcal{J}^+(\Sigma)$ is indeed the future domain of dependence of $\widetilde{\B}_{\frac{R_*+1}{2R_*+1}}$ and denoted by $\mathcal{D}^+(\widetilde{\B}_{\frac{R_*+1}{2R_*+1}})$. It is depicted as the grey region in the right figure. As a result, we obtain a diffeomorphism:
\begin{equation*}
{\Phi}: \mathcal{J}^+(\Sigma) \rightarrow \mathcal{D}^+(\widetilde{\B}_{\frac{R_*+1}{2R_*+1}}).
\end{equation*}
With the naturally induced flat metrics, ${\Phi}$ is indeed a conformal map:
\begin{equation*}
{\Phi}^*\widetilde{m}_{\mu\nu} = \frac{1}{\Lambda(t,x)^2}m_{\mu\nu}  \ \ \text{with}\ \ \Lambda(t,x) = (t+R_*+1)^2-|x|^2.
\end{equation*}
We define functions on the domain of definition on the target of ${\Phi}$：
\begin{align*}
&u_{*}=u+\frac{R_*+1}{2}, \ v_*=v+\frac{R_*+1}{2},\ \widetilde{u}=\frac{1}{2}(\widetilde{t}-\frac{R_*+1}{2R_*+1}-\widetilde{r}),\\ &\widetilde{v}=\frac{1}{2}(\widetilde{t}-\frac{R_*+1}{2R_*+1}+\widetilde{r}),\quad
\widetilde{\Lambda}(\widetilde{t},\widetilde{x})=\big(\widetilde{t}-\frac{R_*+1}{2R_*+1}\big)^2-|\widetilde{x}|^2.
\end{align*}
It is straightforward to check that
\begin{align*}
\Lambda &= 4u_*v_*, \ \ \widetilde{\Lambda}=4\widetilde{u}\widetilde{v},\ \
{\Phi}^*(\widetilde{u})=-\frac{1}{4u_*}, \ \ {\Phi}^*(\widetilde{v})=-\frac{1}{4v_*}, \ \ \big({\Phi}^{-1}\big)^*\widetilde{\Lambda}=\Lambda^{-1}.
\end{align*}
We can also define two principal null vector fields on the target:
\begin{equation*}
\widetilde{L}=\partial_{\widetilde{t}}+\partial_{\widetilde{r}}, \ \ \widetilde{\Lb}=\partial_{\widetilde{t}}-\partial_{\widetilde{r}}.
\end{equation*}
We can compute the tangent map ${\Phi}_*$ as follows:
\begin{equation*}
\begin{split}
{\Phi}_* L&= 4\widetilde{v}^2\widetilde{L}, \ \ {\Phi}_* \Lb= 4\widetilde{u}^2\widetilde{\Lb}, \ \ \ {\Phi}_* (x_i\partial_{x_j}-x_j\partial_{x_i}) = \widetilde{x}_i\partial_{\widetilde{x}_j}-\widetilde{x}_j\partial_{\widetilde{x}_i}.
\end{split}
\end{equation*}
Now define $\widetilde{e}_1$ and $\widetilde{e}_2$ via the following formula:
\begin{equation*}
{\Phi}_* e_A = \widetilde{\Lambda}(\widetilde{t},\widetilde{x})\widetilde{e}_A.
\end{equation*}
Thus, $(\widetilde{e}_1,\widetilde{e}_2,\widetilde{L}, \widetilde{\Lb})$ consists of a null frame for the target space $\mathcal{D}^+(\widetilde{\B}_{\frac{R_*+1}{2R_*+1}})$.
The following set of formulae gives the image of $\mathcal{Z}$ under ${\Phi}_*$:
\begin{equation*}
{\Phi}_* T= 2\widetilde{v}^2\widetilde{L}+ 2\widetilde{u}^2\widetilde{\Lb}, \ \ {\Phi}_* (x_i\partial_{x_j}-x_j\partial_{x_i}) = \widetilde{x}_i\partial_{\widetilde{x}_j}-\widetilde{x}_j\partial_{\widetilde{x}_i}, \ \ {\Phi}_* K = \frac{1}{2}\partial_{\widetilde{t}}.
\end{equation*}

The next objective is to define the fields on the target of ${\Phi}$ corresponding to $(G = dA,{f})$ (on the domain of definition). The correspondence is given by the following formulae: \footnote{In the sequel of the section, when we write $(t,x)$ and $(\widetilde{t},\widetilde{x})$, it is always understood as ${\Phi}: (t,x)\mapsto (\widetilde{t},\widetilde{x})$.}
\begin{align*}
\widetilde{{f}}(\widetilde{t},\widetilde{x})&:=\Big(\big({\Phi}^{-1}\big)^*{(\Lambda \cdot {f})}\Big)\big(\widetilde{t},\widetilde{x}\big)=\Lambda(t,x){{f}}(t,x), \\  \widetilde{A}(\widetilde{t},\widetilde{x})&:=\Big(\big({\Phi}^{-1}\big)^*{A}\Big)\big(\widetilde{t},\widetilde{x}\big) \ \ (\text{hence} \  \widetilde{\Omega}(\widetilde{t},\widetilde{x}):=\Big(\big({\Phi}^{-1}\big)^*{\Omega}\Big)\big(\widetilde{t},\widetilde{x}\big)).
\end{align*}
In view of the conformal theory presented at the beginning of the section, if we take $(G,{f})=(F,\phi)$ the solution of \eqref{MKG}, the pair $(\widetilde{\phi},\widetilde{F})$ is a solution of \eqref{MKG} on $\mathcal{D}^+(\widetilde{\B}_{\frac{R_*+1}{2R_*+1}})$.

In the rest of this subsection (Section \ref{section conformal picture}), we will use the following shorthand notations for the null components of $G$ and $\widetilde{G}$:
\begin{align*}
\alpha &= \alpha(G),\ \ \rho = \rho(G), \ \ \sigma=\sigma(G), \ \ \alphab =\alphab(G),\\
\alphat &= \alpha(\widetilde{G}),\ \ \rhot = \rho(\widetilde{G}), \ \ \sigmat=\sigma(\widetilde{G}), \ \ \alphabt =\alphab(\widetilde{G}).
\end{align*}
Since ${\Phi}_*$ and ${\Phi}^*$ behave well in the functorial way, the following formulae are immediate consequences of the previous computations:
\begin{equation}\label{conformal formulae for null components}
\begin{split}
&\alphat_A(\widetilde{t},\widetilde{x}) = 16u_*v_*^3\alpha_A(t,x), \ \ \rhot(\widetilde{t},\widetilde{x})= 16u_*^2v_*^2\rho(t,x), \\
 &\sigmat(\widetilde{t},\widetilde{x})= 16u_*^2v_*^2\sigma(t,x),\ \
 \alphabt_A(\widetilde{t},\widetilde{x}) = 16u_*^3v_*\alphab_A(t,x),\\
\widetilde{D}_{\widetilde{L}}\widetilde{\phi}(\widetilde{t},\widetilde{x})&=4v_*^2 D_L(\Lambda \phi)(t,x), \ \ \widetilde{D}_{\widetilde{e}_A}\widetilde{\phi}=4u_*v_* D_{e_A}(\Lambda \phi)(t,x), \ \ \widetilde{D}_{\widetilde{\Lb}}\widetilde{\phi}(\widetilde{t},\widetilde{x})=4u_*^2 D_\Lb(\Lambda \phi)(t,x).
\end{split}
\end{equation}
The correspondence also behaves well with respect to taking derivatives for $Z \in \mathcal{Z}$:\footnote{We also use $\widetilde{D}$ to denote the covariant derivatives corresponding to $\widetilde{A}$ on the target. This should not be confused with the $\widetilde{D}$ on the domain of definition of ${\Phi}$.}
\begin{equation}\label{commutator for derivative and conformal compactification}
\begin{split}
\widetilde{\mathcal{L}_Z G}&=\mathcal{L}_{{\Phi}_* Z}\widetilde{G},\ \ \forall Z \in \mathcal{Z},\ \ \ \
\widetilde{\Dt_{\Omega_{ij}} {f}} =\widetilde{D}_{{\Phi}_*\Omega_{ij}}\widetilde{{f}}, \\ \widetilde{\Dt_T {f}}&=\widetilde{D}_{{\Phi}_*T}\widetilde{{f}}+\big(\widetilde{t}-\frac{R_*+1}{2R_*+1}\big)\widetilde{{f}}, \ \ \widetilde{\Dt_K {f}}=\widetilde{D}_{{\Phi}_*K}\widetilde{{f}}+(R_*+1)^2\big(\widetilde{t}-\frac{R_*^2}{(R_*+1)(2R_*+1)}\big)\widetilde{{f}}.
\end{split}
\end{equation}

We will study the energy flux through the space-like hypersurface $\Sigma$. For this purpose, we need to study the geometry of $\Sigma$. It is more convenient to use a new coordinate system to characterize $\Sigma$. We define
\begin{equation*}
U=\sqrt{\big(t+R_*+\frac{1}{2R_*+2}\big)^2-r^2}, \ \ V=\sqrt{\big(t+R_*+\frac{1}{2R_*+2}\big)^2+r^2}.
\end{equation*}
The new coordinates system is $(U,V,\vartheta)$ where $\vartheta \in \mathbf{S}^2$ is the standard spherical coordinates. According to the definition, $\Sigma$ is defined by $U= \frac{2R_*+1}{2(R_*+1)}$. Thus, $(V,\vartheta)$ should be regarded as local coordinate system on $\Sigma$. To simplify notations, we also define
\begin{equation*}
t_*=t+R_*+\frac{1}{2R_*+2}.
\end{equation*}
In the new coordinate system, the volume form of the Minkowski metric $m$ can be written as
\begin{equation}\label{volume form spacetime}
d\text{vol}_m = r^2 dt dr d\vartheta =\frac{rUV}{2t_*}dUdV d\vartheta.
\end{equation}
The hypersurface $\Sigma$ can also be viewed as the graph of the function $g$ over $\mathbb{R}^3$, where $g$ is defined as
\begin{equation*}
t=g(x)=\sqrt{|x|^2+\Big( \frac{2R_*+1}{2(R_*+1)}\Big)^2}-\frac{2R_*^2+2R_*+1}{2(R_*+1)}.
\end{equation*}
Therefore the surface measure on $\Sigma_+$ is given by (using $r$ and $\vartheta$ as coordinate function):
\begin{equation*}
d\mu_{\Sigma}=\sqrt{1+|\nabla g|^2}dx= \sqrt{\frac{2r^2+\big(\frac{2R_*+1}{2(R_*+1)}\big)^2}{r^2+\big(\frac{2R_*+1}{2(R_*+1)}\big)^2}}r^2dr d\vartheta.
\end{equation*}
In view of the defining equation of $\Sigma$, one can express it in terms of  $(V,\vartheta)$ on $\Sigma$:
\begin{equation}\label{volume form of Sigma}
d\mu_\Sigma = \frac{rV^2}{t_*}dV d\vartheta.
\end{equation}
The following lemma gives a formula to compute the contraction of a vector field with the spacetime volume form on $\Sigma$. It will play a key role in the computations of the local energy density on $\Sigma$.
\begin{lemma}\label{lemma contraction of volume form}Let $J$ be a smooth vector field on $\mathbb{R}^{3+1}$ and  $\iota: \Sigma \hookrightarrow \mathbb{R}^{3+1}$ be the canonical embedding. We use $i_J d\text{vol}_m$ to denote the contraction of $J$ and the spacetime volume form. Then we have
\begin{align*}
\iota^*\big(i_J d\text{vol}_m\big)
&=-\frac{rV}{4t_*}\big((t_*-r)J_\Lb+(t_*+r)J_L\big)dV d\vartheta.
\end{align*}
\end{lemma}
\begin{proof}
Since we have already derived the formulae for volume/surface measures in terms of $(U,V,\vartheta)$, it just remains to relate $L$ and $\Lb$ to $\partial_U$ and $\partial_V$. This is recorded in the following formulae:
\begin{equation}\label{L in terms of U and V}
L=\frac{t_*-r}{U}\partial_U+\frac{t_*+r}{V}\partial_V, \ \ \Lb=\frac{t_*+r}{U}\partial_U+\frac{t_*-r}{V}\partial_V.
\end{equation}
\end{proof}

We turn to the study of energy quantities. Given fields $(\widetilde{{f}},\widetilde{G})$ on $\widetilde{\B}_{\frac{R_*+1}{2R_*+1}}$, the standard energy is defined as
\begin{equation*}
\E[\widetilde{{f}},\widetilde{G}](\widetilde{\B}_{\frac{R_*+1}{2R_*+1}})=\int_{0}^{\frac{R_*+1}{2R_*+1}}\int_{\mathbf{S}^2} \Big(|\alphat|^2+|\rhot|^2+|\sigmat|^2+|\alphabt|^2+|\widetilde{D}_{\widetilde{L}}\widetilde{\phi}|^2+\sum_{A=1}^2|\widetilde{D}_{\widetilde{e}_A}\widetilde{\phi}|^2+|\widetilde{D}_{\widetilde{\Lb}}\widetilde{\phi}|^2\Big)\widetilde{r}^2d\widetilde{r} d\widetilde{\vartheta}.
\end{equation*}
In view of our analysis in the exterior region, the more relevant part of the energy is as follows:
\begin{equation*}
\E[\widetilde{{f}},\widetilde{G}](\widetilde{\B}_{+})=\int_{\frac{R_*}{2R_*+1}}^{\frac{R_*+1}{2R_*+1}}\int_{\mathbf{S}^2} \Big(|\alphat|^2+|\rhot|^2+|\sigmat|^2+|\alphabt|^2+|\widetilde{D}_{\widetilde{L}}\widetilde{\phi}|^2+\sum_{A=1}^2|\widetilde{D}_{\widetilde{e}_A}\widetilde{\phi}|^2+|\widetilde{D}_{\widetilde{\Lb}}\widetilde{\phi}|^2\Big)\widetilde{r}^2d\widetilde{r} d\widetilde{\vartheta}.
\end{equation*}
The main objective is to rewrite this energy in terms of $({f}, G)$ on $\Sigma_+$:
\begin{proposition}\label{proposition conformal energy on Sigma}
Given the conformal correspondence $({f},G)\mapsto (\widetilde{{f}},\widetilde{G})$, we have
\begin{equation}\label{conformal energy estimate on Sigma plus}
\E[\widetilde{{f}},\widetilde{G}](\widetilde{\B}_{+})\lesssim \int_{\Sigma_+} r^2|\alphat|^2+|\rhot|^2+|\sigmat|^2+\frac{|\alphabt|^2}{r^2}+|D_L \Psi|^2+|D_L{f}|^2+|\slashed{D} {f}|^2+\frac{|D_\Lb{f}|^2}{r^2}+\frac{|{f}|^2}{r^2}.
\end{equation}
where we use the hypersurface measure $d\mu_\Sigma$ for the integration and $\Psi = r{f}$.
\end{proposition}
\begin{proof}
We start to express the volume form $\widetilde{r}^2d\widetilde{r} d\widetilde{\vartheta}$ in terms of $dVd\vartheta$. We notice that $\big({\Phi}^{-1}\big)^*(U) = \frac{2R_*+1}{2R_*+2}$ on $\Sigma_+$. In terms of $V$ and $U$, we have
\begin{equation*}
\widetilde{r}=\frac{\sqrt{\frac{V^2-U^2}{2}}}{U^2+\big(\frac{2R_*+1}{2R_*+2}\big)^2+\frac{2R_*+1}{R_*+1}\sqrt{\frac{V^2+U^2}{2}}}=\frac{\sqrt{V^2-U^2}}{2^{\frac{3}{2}}U^2+2U\sqrt{V^2+U^2}}.
\end{equation*}
Thus,
\begin{equation*}
d\widetilde{r}=\frac{2^{\frac{3}{2}}U^2\sqrt{V^2+4U^3}V}{\big(2^{\frac{3}{2}}U^2+2U\sqrt{V^2+U^2}\big)^2\sqrt{V^2-U^2}\sqrt{V^2+U^2}}dV.
\end{equation*}
Since $U\approx 1$ and $V\approx r$ on $\Sigma_+$ (provided $R_*\geq 1$ which always holds). Thus, $\widetilde{r}\approx 1$ and we have
\begin{equation}
\widetilde{r}^2d\widetilde{r} d\widetilde{\vartheta} \approx r^{-2}dV d\vartheta.
\end{equation}
In view of \eqref{volume form of Sigma} and the facts that $t_* \approx r \approx v_*$ on $\Sigma_+$, we conclude that
\begin{equation}\label{vol form estimate on target}
\widetilde{r}^2d\widetilde{r} d\widetilde{\vartheta} \approx v_*^{-4}d\mu_{\Sigma_{+}}.
\end{equation}
Thus, combing with \eqref{conformal formulae for null components}, the above equation yields
\begin{equation}\label{conformal energy on Sigma plus}
\begin{split}
\E[\widetilde{{f}},\widetilde{G}](\widetilde{\B}_{+})=\int_{\Sigma_+} & u_*^2v_*^2|\alphat|^2+u_*^4\big(|\rhot|^2+|\sigmat|^2\big)+u_*^6v_*^{-2}|\alphabt|^2+|D_L(\Lambda{f})|^2\\
&+u_*^2v_*^{-2}|\slashed{D}(\Lambda{f})|^2+u_*^4v_*^{-4}|D_\Lb(\Lambda{f})|^2.
\end{split}
\end{equation}
The above integration is understood over the measure $d\mu_{\Sigma_+}$. On the other hand, since $Lu_*=\Lb v_*=1$, $Lv_*=\Lb u_* =0$ and $\Lambda = 4u_* v_*$, we can easily obtain that
\begin{align*}
|D_L(\Lambda{f})|^2&=|4u_* D_L(v_*{f})|^2=|4u_*D_L\big((u_*+r){f}\big))|^2\lesssim u_*^4 |D_L{f}|^2+u_*^2|D_L\Psi|^2,\\
|\slashed{D}(\Lambda{f})|^2 &\approx u_*^2v_*^2 |\slashed{D}{f}|^2, \ \ |D_\Lb(\Lambda{f})|^2 \lesssim v_*^2u_*^2|D_\Lb {f}|^2+|v_*|^2|{f}|^2.
\end{align*}
Since $\frac{1}{2}-\frac{1}{4(R_*+1)}\leq u_* \leq \frac{1}{2}$, thus $u_*\approx 1$. The above estimate together with \eqref{conformal energy on Sigma plus} completes the proof of the proposition.
\end{proof}
We can further more eliminate the term $\frac{|{f}|^2}{r^2}$ in \eqref{conformal energy estimate on Sigma plus}:
\begin{corollary}\label{corollary conformal energy}
Given the conformal correspondence $({f},G)\mapsto (\widetilde{{f}},\widetilde{G})$, we have
\begin{equation}\label{conformal energy estimate on Sigma plus without zeroth order terms}
\E[\widetilde{{f}},\widetilde{G}](\widetilde{\B}_{+})\lesssim R_*^{-1} \int_{\S_{R_*}^{R_*}} |{f}|^2+\int_{\Sigma_+} r^2|\alphat|^2+|\rhot|^2+|\sigmat|^2+\frac{|\alphabt|^2}{r^2}+|D_L \Psi|^2+|D_L{f}|^2+|\slashed{D} {f}|^2+\frac{|D_\Lb{f}|^2}{r^2}.
\end{equation}
\end{corollary}
\begin{proof}
According to \eqref{volume form of Sigma}, by Newton-Leibniz formula, we have
\begin{align*}
\int_{\Sigma_+} \frac{|{f}|^2}{r^2}&\approx \int_{R_*}^\infty \int_{\mathbf{S}^2}|{f}|^2dVd\vartheta = R_*^{-1} \int_{\S_{R_*}^{R_*}} |{f}|^2-\lim_{V_0\rightarrow \infty}\Big( V_0^{-1}\int_{\Sigma_+\cap V=V_0}|{f}|^2+2\int_{\Sigma_{+}}V_0^{-1}\Re(\overline{D_{\partial_V}{f}}\cdot {f})\Big)\\
& \lesssim R_*^{-1} \int_{\S_{R_*}^{R_*}} |{f}|^2+\int_{\Sigma_{+}}V^{-1}|D_{\partial_V}{f}||{f}| \les R_*^{-1} \int_{\S_{R_*}^{R_*}} |{f}|^2+\int_{\Sigma_{+}}|D_{\partial_V}{f}|^2)+\frac{1}{2}\int_{\Sigma_{+}}\frac{|{f}|^2}{r^2}.
\end{align*}
Thus,
\begin{align*}
\int_{\Sigma_+} \frac{|{f}|^2}{r^2}\lesssim R_*^{-1} \int_{\S_{R_*}^{R_*}} |{f}|^2+\int_{\Sigma_{+}}|D_{\partial_V}{f}|^2.
\end{align*}
According to \eqref{L in terms of U and V}, on $\Sigma_+$, we have
\begin{align*}
|D_{\partial_V}{f}|^2 =|\frac{V}{4t_*r}\big((t_*+r)L-(t_*-r)\Lb\big){f}|^2\lesssim |D_L {f}|^2+\frac{|D_\Lb{f}|^2}{r^2}.
\end{align*}
Therefore,
\begin{align*}
\int_{\Sigma_+} \frac{|{f}|^2}{r^2}\lesssim R_*^{-1} \int_{\S_{R_*}^{R_*}} |{f}|^2+\int_{\Sigma_{+}}|D_L {f}|^2+\frac{|D_\Lb{f}|^2}{r^2}.
\end{align*}
In view of \eqref{conformal energy estimate on Sigma plus}, this completes the proof.
\end{proof}

To bound the righthand side of \eqref{conformal energy estimate on Sigma plus}, we need standard energy estimate and $r^p$-weighted energy estimates on $\Sigma_+$. We take the integration domain $\mathcal{D}$ to be the spacetime slab bounded by $\Sigma_+$ and $\B_{R_*}$, see the grey region in the following picture:

\ \ \ \ \ \ \ \ \ \ \ \ \ \ \ \ \ \ \ \ \ \ \ \ \ \ \ \ \ \ \ \ \ \ \ \ \ \ \includegraphics[width=3 in]{hyperboloidenergy.pdf}

\begin{lemma}\label{energy identities hyperboloid}
For all $G$ and ${f}$, $1\leq p\leq 2$, we have
\begin{equation}\label{basic energy hyperboloid}
\begin{split}
& \ \ \ \int_{\Sigma_+}\frac{(t_*+r)\big(|D_L{f}|^2+|\alpha|^2\big)+ 2t_*\big(|\slashed{D}{f}|^2+|\rho|^2+|\sigma|^2\big)+(t_*-r)\big(|D_\Lb{f}|^2+|\alphab|^2\big)}{8V}\\
&= \E[G,{f}](\B_{R_*})-\int_{\D}\Re\big(\overline{\Box_A{f}} \cdot D_{\partial_t} {f} \big)+\nabla^\mu G_{\mu\nu}\cdot G_0{}^{\nu} + F_{0\mu} {J}[{f}]^\mu,
\end{split}
\end{equation}
and
\begin{equation}\label{r wieghted hyperboloid}
\begin{split}
&\ \ \ \ \int_{\mathcal{B}_{R_*}}r^{p-2}\big(|D_L \Psi|^2+|\slashed{D} \Psi|^2\big)+r^{p}\big(|\alpha(G)|^2+|\rho(G)|^2+|\sigma(G)|^2\big)\\
&=\int_{\Sigma_+}\frac{r^p \big(|D_L\Psi|^2+r^2|\alpha|^2\big)(t_*+r)+r^p \big(|\slashed{D}\Psi|^2+r^2(|\rho|^2+|\sigma|^2)\big)(t_*-r)}{4r^2V} \\
&\ \ \ +\int_{\D}r^{p-3}\Big(p(|D_L \Psi|^2+r^2|\alpha(G)|^2\big)+(2-p)(|\slashed{D}\Psi|^2+r^2|\rho(G)|^2+r^2|\sigma(G)|^2)\Big) \\
&\ \ \ + \int_{\D} r^{p-1}\Re\big(\overline{\Box_A{f}} \cdot D_L \Psi\big)+r^p \nabla^\mu G_{\mu\nu}\cdot G_{L}{}^{\delta}+r^p F_{L\mu} {J}[{f}]^\mu.
\end{split}
\end{equation}
Here $\Psi=rf$.
\end{lemma}
\begin{proof}
We first derive $r^p$-weighted energy identity. The following computations are similar to those in the proof of Lemma \ref{r weighted energy estimates}. We take $X=r^p L$, $\chi = r^{p-1}$ and $Y=\frac{p}{2}r^{p-2} |{f}|^2 L$ in \eqref{divergence of J} and then integrate on $\mathcal{D}$. The expressions of $\mathbf{D_1}$ and $\mathbf{D_2}$ (in \eqref{divergence of J}) are given in \eqref{C3}. According to Stokes formula, we have
\begin{equation*}
\int_{\mathcal{B}_{R_*}}\widetilde{J}[G, {f}]^\mu n_\mu+\int_{\Sigma_+}\iota^*\big(i_X d\text{vol}_m\big)=\int_{\mathcal{D}} \mathbf{D_1}+\mathbf{D_2}
\end{equation*}
On $\B_{R_*}$, the normal $n^\mu$ is $\partial_t$, we have
\begin{equation*}
^{(X)}\widetilde{J}[G, {f}]^\mu n_\mu = \frac{1}{2}r^{p-2}\big(r^2\alpha(G)^2+r^2\rho(G)^2+r^2\sigma(G)^2+|D_L\Psi|^2+|\slashed{D}\Psi|^2\big)-\frac{1}{2}\big((p+1)r^{p-2}|{f}|^2+r^{p-1}\partial_r(|{f}|^2)\big).
\end{equation*}
Therefore, we have
\begin{equation*}
\begin{split}
\int_{\B_{R_*}}\,^{(X)}\widetilde{J}[G, {f}]^\mu n_\mu &=\frac{1}{2}\int_{\B_{r_1}^{r_2}} r^2\alpha(G)^2+r^2\rho(G)^2+r^2\sigma(G)^2+|D_L\Psi|^2+|\slashed{D}\Psi|^2 \\ &\quad-\frac{1}{2}\int_{r_1}^{r_2}\int_{\mathbf{S}^2}\underbrace{(p+1)r^{p}|{f}|^2+r^{p+1}\partial_r(|{f}|^2)}_{=\partial_r(r^{p+1}|{f}|^2)} d\vartheta dr\\
&=\frac{1}{2}L_1+\frac{1}{2}\int_{\S_{r_1}^{r_1}}r^{p-1}|{f}|^2.
\end{split}
\end{equation*}
On the other hand, in view of Lemma \ref{lemma contraction of volume form}, we need to compute $\,^{(X)}\widetilde{J}[G, {f}]_L$ and $\,^{(X)}\widetilde{J}[G, {f}]_\Lb$. In fact, we have
\begin{align*}
r^2\cdot \,^{(X)}\widetilde{J}[G, {f}]_L &= r^p \big(|D_L\Psi|^2+r^2|\alpha|^2\big)-\frac{1}{2}L(r^{p+1}|{f}|^2),\\
r^2\cdot \,^{(X)}\widetilde{J}[G, {f}]_\Lb &= r^p \big(|\slashed{D}\Psi|^2+r^2(|\rho|^2+|\sigma|^2)\big)+\frac{1}{2}\Lb(r^{p+1}|{f}|^2).
\end{align*}
Therefore, by Lemma \ref{lemma contraction of volume form} and replacing $L$ and $\Lb$ in the above formulae by \eqref{L in terms of U and V}, we obtain that
\begin{align*}
\iota^*\big(i_{\widetilde{J}}d\text{vol}_m\big) 
&=-\frac{r^p \big(|D_L\Psi|^2+r^2|\alpha|^2\big)(t_*+r)+r^p \big(|\slashed{D}\Psi|^2+r^2(|\rho|^2+|\sigma|^2)\big)(t_*-r)}{4r^2V}d\mu_{\Sigma}\\
&\quad +\frac{1}{2}\partial_V (r^{p+1}|{f}|^2)dVd\vartheta
\end{align*}
Therefore, we have
\begin{equation}\label{CC6}
\begin{split}
\int_{\Sigma_+}\iota^*\big(i_X d\text{vol}_m\big) &=-\int_{\Sigma_+}\frac{r^p \big(|D_L\Psi|^2+r^2|\alpha|^2\big)(t_*+r)+r^p \big(|\slashed{D}\Psi|^2+r^2(|\rho|^2+|\sigma|^2)\big)(t_*-r)}{4r^2V} \\
&\quad -\frac{1}{2}\int_{\S_{r_1}^{r_1}}r^{p-1}|{f}|^2.
\end{split}
\end{equation}
The $r^p$-weighted energy identity follows immediately.

For the basic energy identity, we simply take  $X=\partial_t$, $\chi = 0$ and $Y=0$ in \eqref{divergence of J}. The identity easily follows if we observe that
\begin{align*}
\iota^*\big(i_{\widetilde{J}}d\text{vol}_m\big) &=-\frac{(t_*+r)|D_L{f}|^2+ 2t_*|\slashed{D}{f}|^2+(t_*-r)|D_\Lb{f}|^2}{8V}d\mu_\Sigma.
\end{align*}
\end{proof}

\subsection{The energy estimates on the conformal conformal compactification}\label{section conformal energy}

We first apply the theory of the previous section to the static solution $({f}, G)=(0,F[q_0])$. By the definition of the charge field $F[q_0]$, according to \eqref{conformal energy on Sigma plus}, we can bound that
\begin{equation}\label{conformal energy for static solution}
\E\big[F[q_0]\big](\widetilde{\B}_{+}) \les \int_{\Sigma_+} u_*^4|r^{-2}|^2+u_+^2 |r^{-3}|^2\les 1.
\end{equation}
This is due to the fact that $|\rho(F[q_0])|$ decays like $r^{-2}$ while all the other components decay at least $r^{-3}$. We also note that on $\Sigma_+$, $u_* \approx 1$.

\begin{remark}
Despite the simplicity, the computation in \eqref{conformal energy for static solution} is of great conceptual importance. Indeed, if one considers conformal energy on a constant time slice, the factor $u_*^4$ would be replace by $r^2$ (near spatial infinity) so that the contribution of the charge part of the field would be divergent. However, \eqref{conformal energy for static solution} shows that the charge part behaves very well near null infinity. This is the reason we choose the inversion to compactify the spacetime over the usually Penrose compactification.
\end{remark}

The main purpose of the current section is to obtain (the $L^2$ bounds up to two derivatives) of $(\widetilde{\phi},\widetilde{F})$ on $\B_{+}$. In view of \eqref{conformal energy estimate on Sigma plus without zeroth order terms}, it is reasonable to define the following quantity:
\begin{equation}
\E_+[{f},{G}]=\int_{\Sigma_+} r^2|\alphat|^2+|\rhot|^2+|\sigmat|^2+\frac{|\alphabt|^2}{r^2}+|D_L \Psi|^2+|D_L{f}|^2+|\slashed{D} {f}|^2+\frac{|D_\Lb{f}|^2}{r^2}.
\end{equation}
In what follows, we take the $({f},{G})$ to be $(\phi^{(\k)},\mathcal{L}_Z^{(\k)} \Fc)$ for $|\k|\leq 2$. In this specific set-up, we first simplify the estimates in Lemma \ref{energy identities hyperboloid}.

We start with identity \eqref{basic energy hyperboloid}. Because $t_*\approx V\approx r$ and $|t_*-r|\approx 1$, its lefthand side is approximately
\begin{equation*}
\int_{\Sigma_+} |\alphak|^2+|\rhok|^2+|\sigmak|^2+|D_L\phi^{(\k)}|^2+|\slashed{D}\phi^{(\k)}|^2+\frac{|\alphabk|^2+|D_\Lb\phi^{(\k)}|^2}{r^2}.
\end{equation*}
The first term of the righthand side is coming from the data hence bounded by ${\epsilonc}R_*^{-6-2\varepsilon_0}$. The second one is precisely the error terms that we have controlled in the exterior region (indeed, $\D\subset \D_{R_*}$), hence also bounded by ${\epsilonc}R_*^{-6-2\varepsilon_0}$. Therefore, via \eqref{energy identities hyperboloid}, we arrive at the following estimates
\begin{equation}\label{basic energy est for phik Fk on Sigma plus}
\int_{\Sigma_+} |\alphak|^2+|\rhok|^2+|\sigmak|^2+|D_L\phi^{(\k)}|^2+|\slashed{D}\phi^{(\k)}|^2+\frac{|\alphabk|^2+|D_\Lb\phi^{(\k)}|^2}{r} \lesssim {\epsilonc}R_*^{-6-2\varepsilon_0}.
\end{equation}

We use the $p=2$ case of \eqref{r wieghted hyperboloid}. The first term of the left hand side is coming from the data hence bounded by ${\epsilonc}R_*^{-4-2\varepsilon_0}$. The second and third terms of the righthand side are precisely the error terms that we have controlled in the exterior region hence also bounded by ${\epsilonc}R_*^{-4-2\varepsilon_0}$. Finally, we use $t_*\approx V\approx r$ and $|t_*-r|\approx 1$ for the first term on the righthand side. Therefore, we arrive at the following estimate:
\begin{equation}\label{weighted energy est for phik Fk on Sigma plus}
\int_{\Sigma_+}|D_L\psi^{(\k)}|^2+r^2|\alpha|^2+r\big(|\slashed{D}\phi^{(\k)}|^2+|\rho|^2+|\sigma|^2)\big) \lesssim_{R_*} {\epsilonc}R_*^{-4-2\varepsilon_0}.
\end{equation}
As a conclusion, we obtain that
\begin{equation}\label{conf estimates aux}
\E_+[\phi^{(\k)},\mathcal{L}_Z^{(\k)}\Fc]\lesssim 1.
\end{equation}
Here we recall that the implicit constant here depends only on the size of the initial data $C_0$ defined before the main theorem.
\begin{remark}
From this point till the end of the paper, we will ignore the dependence on $R_*$ for the universal constants since $R_*$ is already fixed.
\end{remark}
We now have all the preparations to bound the $H^2$ norms of $(\widetilde{\phi},\widetilde{F})$ on $\mathcal{B}_+$.

\bigskip

We take $({f}, G)=(\phi, F)$ in \eqref{conformal energy estimate on Sigma plus without zeroth order terms}. The first term on the righthand of \eqref{conformal energy estimate on Sigma plus without zeroth order terms} is bounded by the initial datum. Therefore, by taking $(\k)=(0)$ in \eqref{conf estimates aux}, the estimate \eqref{conformal energy estimate on Sigma plus without zeroth order terms} gives
\begin{equation*}
\E(\widetilde{\phi}, \widetilde{\Fc})(\B_+) \lesssim 1.
\end{equation*}
On the other hand, we know that $\widetilde{F}=\widetilde{\Fc}+\widetilde{F_{q_0}}$ and we have already shown that $\E(\widetilde{F}[q_0])\lesssim  \mathcal{E}_{\text{initial}}$. Hence, we have
\begin{equation}\label{conformal energy zeroth order final}
\E(\widetilde{\phi}, \widetilde{F})(\B_+) \lesssim 1.
\end{equation}

We now assume $|\k|=1$ and we take $({f},G)=(\Dt_Z \phi, \mathcal{L}_Z F)$ where $Z$ corresponds to the index $(\k)$. If $Z=\Omega_{ij}$, the commutator formula \eqref{commutator for derivative and conformal compactification} has no error term for $\Omega_{ij}$'s. We then can repeat the above procedure. Therefore, we have
\begin{equation*}
\E(\widetilde{D}_{\widetilde{\Omega_{ij}}}\widetilde{\phi}, \mathcal{L}_{\widetilde{\Omega_{ij}}}\widetilde{F})(\B_+) \lesssim 1,
\end{equation*}
where $\widetilde{\Omega_{ij}} =\widetilde{x_i}\partial_{\widetilde{x_j}}-\widetilde{x_j}\partial_{\widetilde{x_i}}$.

If $Z=T$, according to \eqref{commutator for derivative and conformal compactification}, we have $\widetilde{D}_{{\Phi}_*T}\widetilde{\phi}=\widetilde{\Dt_T \phi}-\big(\widetilde{t}-\frac{R_*+1}{2R_*+1}\big)\widetilde{\phi}$.
Similar to the previous case, $\E(\widetilde{\widetilde{D}_T \phi},\mathcal{L}_{{\Phi}_* T}\widetilde{F}=\widetilde{\mathcal{L}_T \Fc})(\B_+)$ are bounded by  \eqref{conformal energy estimate on Sigma plus without zeroth order terms} and \eqref{conf estimates aux}. Since $\widetilde{t}$ and its derivatives on $\B_+$ is bounded, in view of the $L^\infty$ estimates on $\phi$ (which implies the $\widetilde{\phi}$ is bounded in $L^2$) and \eqref{conformal energy zeroth order final}, the energy contributed by $\big(\widetilde{t}-\frac{R_*+1}{2R_*+1}\big)\widetilde{\phi}$ is also bounded. Thus, we conclude that
\begin{equation}\label{A2}
\E(\widetilde{D}_{\widetilde{v}^2\widetilde{L}+ \widetilde{u}^2\widetilde{\Lb}}\widetilde{\phi}, \mathcal{L}_{\widetilde{v}^2\widetilde{L}+ \widetilde{u}^2 \widetilde{\Lb}}\widetilde{F})(\B_+) \lesssim 1.
\end{equation}
If $Z=K$, according to\eqref{commutator for derivative and conformal compactification}, we have $
\widetilde{D}_{{\Phi}_*K}\widetilde{\phi}= \widetilde{\Dt_K \phi}-(R_*+1)^2\big(\widetilde{t}-\frac{R_*^2}{(R_*+1)(2R_*+1)}\big)\widetilde{\phi}$.
Similarly, since $\widetilde{t}$ and its derivatives on $\B_+$ are bounded, we can argue exactly in the same manner that
\begin{equation}\label{A3}
\E(\widetilde{D}_{\partial_{\widetilde{t}}}\widetilde{\phi}, \mathcal{L}_{\partial_{\widetilde{t}}}\widetilde{F})(\B_+) \lesssim 1.
\end{equation}
On the other hand, on $\B_+$, both $\widetilde{r}$ and its inverse are bounded (as well as their derivatives). We also have
\begin{equation*}
\widetilde{v}^2\widetilde{L}+ \widetilde{u}^2\widetilde{\Lb} = \frac{1}{2}\Big(\widetilde{r}^2+\big(\frac{R_*+1}{2R_*+1}\big)\Big)\partial_{\widetilde{t}}+\frac{R_*+1}{2R_*+1}\widetilde{r}\partial_{\widetilde{r}}.
\end{equation*}
Therefore, \eqref{A2} and $\eqref{A3}$ together with all the previous estimates imply that
\begin{equation}\label{conformal side first order}
\E({\widetilde{\phi}, \widetilde{F}})(\B_+) + \E(\widetilde{D}\widetilde{\phi}, \nabla \widetilde{F})(\B_+) \lesssim 1.
\end{equation}

We now assume $|\k|=2$ and we take $({f}, G)=(\phi^{(\mathbf{2})}, \mathcal{L}_Z^{(\mathbf{2})}F)$. Since \eqref{commutator for derivative and conformal compactification} has no error term for $\Omega_{ij}$'s, it is immediate to see that
\begin{equation*}
\E(\widetilde{D}_{\widetilde{\Omega_{ij}}}\widetilde{D}_{\widetilde{\Omega'_{ij}}}\widetilde{\phi}, \mathcal{L}_{\widetilde{\Omega_{ij}}}\mathcal{L}_{\widetilde{\Omega'_{ij}}}\widetilde{F})(\B_+) \lesssim 1.
\end{equation*}
We now consider the case $({f}, G)=(\Dt_T\Dt_T \phi, \mathcal{L}_T\mathcal{L}_T F)$. On $\widetilde{t}=0$ (or $\mathcal{B}_+$), we have
\begin{equation*}
\widetilde{D}_{{\Phi}_*T}\widetilde{D}_{{\Phi}_*T}\widetilde{\phi}=\widetilde{\Dt_T\Dt_T \phi}+\frac{2R_*+2}{2R_*+1}\widetilde{D}_{{\Phi}_* T}\widetilde{\phi}+\Big(2\big(\frac{R_*+1}{2R_*+1}\big)^2+r^2\Big)\widetilde{\phi}.
\end{equation*}
We can bound all the terms on the righthand side and we obtain
\begin{equation}\label{A4}
\E(\widetilde{D}_{\widetilde{v}^2\widetilde{L}+ \widetilde{u}^2\widetilde{\Lb}}\widetilde{D}_{\widetilde{v}^2\widetilde{L}+ \widetilde{u}^2\widetilde{\Lb}}\widetilde{\phi}, \mathcal{L}_{\widetilde{v}^2\widetilde{L}+ \widetilde{u}^2 \widetilde{\Lb}}\mathcal{L}_{\widetilde{v}^2\widetilde{L}+ \widetilde{u}^2 \widetilde{\Lb}}\widetilde{F})(\B_+) \lesssim 1.
\end{equation}
We now consider the case $({f}, G)=(\Dt_K\Dt_K \phi, \mathcal{L}_K\mathcal{L}_K F)$. On $\mathcal{B}_+$, we have
\begin{equation*}
\frac{1}{4}\widetilde{D}_{\widetilde{T}}\widetilde{D}_{\widetilde{T}}\widetilde{\phi}=\widetilde{\Dt_T\Dt_T \phi}+\frac{R_*^2(R_*+1)}{2R_*+1}\widetilde{D}_{\widetilde{T}}\widetilde{\phi}+\Big(\frac{1}{2}(R_*+1)^2+\frac{(R_*+1)^2R_*^4}{(2R_*+1)^2}\Big)\widetilde{\phi}.
\end{equation*}
This implies
\begin{equation}\label{A5}
\E(\widetilde{D}_{\partial_{\widetilde{t}}}\widetilde{D}_{\partial_{\widetilde{t}}}\widetilde{\phi}, \mathcal{L}_{\partial_{\widetilde{t}}}\mathcal{L}_{\partial_{\widetilde{t}}}\widetilde{F})(\B_+) \lesssim 1.
\end{equation}
Similar to \eqref{conformal side first order}, by combining \eqref{A4}, \eqref{A5} and all the previous estimates, we finally obtain that
\begin{equation}\label{conformal side second order}
\E({\widetilde{\phi}, \widetilde{F}})(\B_+) + \E(\widetilde{D}\widetilde{\phi}, \nabla \widetilde{F})(\B_+)+\E(\widetilde{D}^2\widetilde{\phi}, \nabla^2 \widetilde{F})(\B_+) \les 1.
\end{equation}

Finally, to obtain the $H^2$ bound of $({\widetilde{\phi}, \widetilde{F}})$ on the entire ball $\mathcal{B}_{\frac{R_*+1}{2R_*+1}}$, it remains to bound the contribution from $\mathcal{B}_{\frac{R_*+1}{2R_*+1}}-\B_+ =\Phi(\Sigma_-)$. On the other hand, according to the theory of Klainerman-Machedon \cite{MKGkl}, since $\Sigma_-$ is bounded, the solution is also bounded up to two derivatives in $L^2$ norms. This implies immediately that the $H^2$ energy of $({\widetilde{\phi}, \widetilde{F}})$ on $\mathcal{B}_{\frac{R_*+1}{2R_*+1}}-\B_+$ is bounded. Finally, we obtain that
\begin{equation}\label{conformal side H2 bound}
\E({\widetilde{\phi}, \widetilde{F}})(\mathcal{B}_{\frac{R_*+1}{2R_*+1}}) + \E(\widetilde{D}\widetilde{\phi}, \nabla \widetilde{F})(\mathcal{B}_{\frac{R_*+1}{2R_*+1}})+\E(\widetilde{D}^2\widetilde{\phi}, \nabla^2 \widetilde{F})(\mathcal{B}_{\frac{R_*+1}{2R_*+1}})  \lesssim 1.
\end{equation}
Once again, by the Main Theorem proved in Klainerman-Machedon \cite{MKGkl}, we have uniform $H^2$ control of $\widetilde{\phi}$ and $\widetilde{F}$. According to Sobolev inequality on $\mathcal{B}_{\frac{R_*+1}{2R_*+1}}$, we conclude that there exists a constant $C$, so that we have the pointwise bound
\begin{equation*}
|\widetilde{\phi}|+|\widetilde{D}\widetilde{\phi}|+|\widetilde{F}|\les 1.
\end{equation*}
Finally, we use the formulae in \eqref{conformal formulae for null components} and this provides the peeling estimates for the null components of $D\phi$ and $F$ in the interior region. Together with the pointwise estimates derived in the exterior region, this finishes the proof of the main theorem.

\appendix

\section{Tool kit}\label{Appendix tools}
We collect some frequently used inequalities and calculations in this appendix.
\subsection{Gronwall type inequalities}
\begin{lemma}[A variant of Gronwall's inequality]\label{lemma gronwall}
Let $f(t) \in C^0\big([0,T]\big)$ and $C_1$ and $C_2$ are two positive constants. If we have
\begin{equation*}
f(t) \leq C_1 + C_2 \int_{t}^T s^{-1}f(s)ds,
\end{equation*}
for all $t\leq T$, then we have
\begin{equation*}
\begin{split}
f(t)&\leq C_1 + \frac{C_1}{C_2}\Big(\big(\frac{T}{t}\big)^{C_2}-1\Big) \ \ \text{and} \ \ \int_{t}^T s^{-1}f(s)ds\leq \frac{C_1}{C_2}\Big(\big(\frac{T}{t}\big)^{C_2}-1\Big).
\end{split}
\end{equation*}
\end{lemma}

\begin{lemma}\label{lemma changing decay rates}
Let $f(t) \in C^1\big([1,+\infty)\big)$ be a positive function. If for all $r_1\geq 1$, there are positive constants $C$ and $p$ so that
\begin{equation*}
\int_{r_1}^{\infty}f(s)ds\leq C r_1^{-p}.
\end{equation*}
Then, for all $q<p$, we have
\begin{equation*}
\int_{r_1}^{\infty}s^{q}f(s)ds \leq \frac{C p}{p-q} r_1^{q-p}.
\end{equation*}
\end{lemma}

The proofs  are straightforward and we omit the details.

\subsection{Sobolev inequalities}
We first recall the standard Sobolev inequalities on spheres.
\begin{lemma}[Sobolev inequalities on spheres]
Let $\Xi$ be a $(p_0,q_0)$-tensor field ($p_0+q_0\leq 3$ in the current work) on $\mathbf{S}_r$ (the standard sphere of radius $r$). Then, for $p>2$, we have
\begin{equation*}
\begin{split}
\|\Xi\|_{L^\infty(\mathbf{S}_r)} &\lesssim_p r^{-\frac{2}{p}}\big(\|\Xi\|_{L^p(\mathbf{S}_r)}+ r \|\slashed{\nabla}\Xi\|_{L^p(\mathbf{S}_r)}\big),\\
\|\Xi\|_{L^4(\mathbf{S}_r)} &\lesssim r^{-\frac{1}{2}}\big(\|\Xi\|_{L^2(\mathbf{S}_r)}+ r \|\slashed{\nabla}\Xi\|_{L^2(\mathbf{S}_r)}\big),
\end{split}
\end{equation*}
where $\slashed{\nabla}$ is the covariant derivative on $\mathbf{S}_r$.
\end{lemma}

\begin{remark}
We use $\mathcal{L}_{\Omega} \Xi$ to denote all possible Lie derivatives with respect to the rotation vector fields $\Omega_{12}$, $\Omega_{23}$ and $\Omega_{31}$. It is straightforward to check that
\begin{equation*}
|\mathcal{L}_{\Omega}\Xi|^2 \approx |\Xi|^2+ r^2 |\slashed{\nabla}\Xi|^2.
\end{equation*}
Thus, the above Sobolev inequalities can also be stated as
\begin{equation}\label{Sobolev on sphere}
\begin{split}
\|\Xi\|_{L^\infty(\mathbf{S}_r)} &\lesssim_p r^{-\frac{2}{p}}\big(\|\Xi\|_{L^p(\mathbf{S}_r)}+ \|\mathcal{L}_{\Omega}\Xi\|_{L^p(\mathbf{S}_r)}\big),\\
\|\Xi\|_{L^4(\mathbf{S}_r)} &\lesssim r^{-\frac{1}{2}}\big(\|\Xi\|_{L^2(\mathbf{S}_r)}+ \|\mathcal{L}_{\Omega}\Xi\|_{L^2(\mathbf{S}_r)}\big).
\end{split}
\end{equation}
\end{remark}

\begin{lemma}[A Sobolev inequality on parameterized hypersurfaces]
Let $\Xi$ be a tensor field on a parameterized hypersurface with parameters $(s,\vartheta) \in [a,b]\times \mathbf{S}^2$ (equipped with the product measure). Then, we have
\begin{equation*}
\begin{split}
\sup_{s\in [a,b]}\|\Xi(s,\vartheta)\|_{L^4(\mathbf{S}_{\vartheta})} &\lesssim  \|\Xi(s,\vartheta)\|_{L^2(\mathbf{S}_{\vartheta})}+\|\partial_s \Xi(s,\vartheta)\|_{L^2(\mathbf{S}_{\vartheta})}+\|\partial_\vartheta\Xi(s,\vartheta)\|_{L^2(\mathbf{S}_{\vartheta})}
\end{split}
\end{equation*}
\end{lemma}
\begin{proof}
It suffices to prove the inequality for any fixed $s_0\in [a,b]$. According to Sobolev inequality on $\mathbf{S}^2$, we have
\begin{equation*}
\|\Xi(s_0,\vartheta)\|_{L^4(\mathbf{S}^2)} \lesssim \|\Xi(s_0,\vartheta)\|_{W^{\frac{1}{2},2}(\mathbf{S}^2)}.
\end{equation*}
Regarding $s=s_0$ as a codimension 1 hypersurface in $[a,b]\times \mathbf{S}^2$, the standard trace theorem yields
\begin{equation*}
\|\Xi(s_0,\vartheta)\|_{W^{\frac{1}{2},2}(\mathbf{S}^2)} \lesssim \|\Xi(s,\vartheta)\|_{W^{1,2}([a,b]\times \mathbf{S}^2)}.
\end{equation*}
This gives the proof of the inequality.
\end{proof}
The following corollary of the above inequality will be extremely useful when one derives pointwise decay:
\begin{corollary}
For a smooth tensor field $\Xi$ on the incoming null hypersurface $\H_{r_2}^{r_1}$, $\Hb_{r_1}^{r_2}$ or $\B_{r_1}$, we have
\begin{equation}\label{Sobolve on incoming null hypersurfaces}
\begin{split}
 r_2 \|\Xi\|_{L^4(\S_{r_1}^{r_2})}^2&\lesssim \int_{\Hb_{r_2}^{r_1}} |\Xi|^2+\int_{\Hb_{r_2}^{r_1}} |\mathcal{L}_\Lb \Xi|^2+\int_{\Hb_{r_2}^{r_1}} |\mathcal{L}_\Omega \Xi|^2,\\
  r_2 \|\Xi\|_{L^4(\S_{r_1}^{r_2})}^2&\lesssim \int_{\H_{r_2}^{r_1}} |\Xi|^2+\int_{\H_{r_2}^{r_1}} |\mathcal{L}_L \Xi|^2+\int_{\H_{r_2}^{r_1}} |\mathcal{L}_\Omega \Xi|^2,\\
  r_1 \|\Xi\|_{L^4(\S_{r_1}^{r_1})}^2&\lesssim \int_{\B_{r_1}} |\Xi|^2+\int_{\B_{r_1}} |\mathcal{L}_{\partial_r} \Xi|^2+\int_{\B_{r_1}} |\mathcal{L}_\Omega \Xi|^2.
\end{split}
\end{equation}
For the a scalar field ${f}$ (as a section of a line bundle $\mathbf{L}$) and a given connection $A$, we have
\begin{equation}\label{Sobolve for scalar filed on incoming null hypersurfaces}
\begin{split}
 r_2 \|{f}\|_{L^4(\S_{r_1}^{r_2})}^2&\lesssim \int_{\Hb_{r_2}^{r_1}} |{f}|^2+\int_{\Hb_{r_2}^{r_1}} |D_\Lb {f}|^2+\int_{\Hb_{r_2}^{r_1}} |D_\Omega {f}|^2,\\
  r_2 \|{f}\|_{L^4(\S_{r_1}^{r_2})}^2&\lesssim \int_{\H_{r_2}^{r_1}} |{f}|^2+\int_{\H_{r_2}^{r_1}} |D_L {f}|^2+\int_{\H_{r_2}^{r_1}} |D_\Omega {f}|^2,\\
  r_1 \|{f}\|_{L^4(\S_{r_1}^{r_1})}^2&\lesssim \int_{\B_{r_1}} |{f}|^2+\int_{\B_{r_1}} |D_{\partial_r} {f}|^2+\int_{\B_{r_1}} |D_\Omega {f}|^2.
  \end{split}
\end{equation}
\end{corollary}

\begin{proof}
We will apply the previous lemma by taking $a=\frac{-r_2}{2}$, $b=-\frac{r_1}{2}$ and using $s=u$ and $\vartheta$ to parameterize $\Hb_{r_2}^{r_1}$. Since we have
\begin{align*}
\int_{\Hb_{r_2}^{r_1}} |\mathcal{L}_L \Xi|^2 &= \int_{b}^{a} |\partial_u (r\Xi)|^2 du d\vartheta- \int_{b}^{a} |\Xi|^2 du d\vartheta\ \ \text{and} \ \ \int_{\Hb_{r_2}^{r_1}} |\mathcal{L}_\Omega \Xi|^2 = \int_{b}^{a} |\partial_\vartheta (r\Xi)|^2 du d\vartheta,
\end{align*}
we obtain that (notice that $r\geq 1$)
\begin{align*}
\|r\Xi\|_{W^{1,2}([a,b]\times \mathbf{S}^2)} \lesssim \int_{\Hb_{r_2}^{r_1}} |\Xi|^2+\int_{\Hb_{r_2}^{r_1}} |\mathcal{L}_L \Xi|^2+\int_{\Hb_{r_2}^{r_1}} |\mathcal{L}_\Omega \Xi|^2.
\end{align*}
According to the previous lemma, we obtain that
\begin{align*}
\sup_{u\in [a,b]}\big(\int_{\mathbf{S}^2} r^4 |\Xi(s,\vartheta)|^4 d\vartheta\big)^\frac{1}{2}  \lesssim \int_{\Hb_{r_2}^{r_1}} |\Xi|^2+\int_{\Hb_{r_2}^{r_1}} |\mathcal{L}_L \Xi|^2+\int_{\Hb_{r_2}^{r_1}} |\mathcal{L}_\Omega \Xi|^2.
\end{align*}
By setting $u=b$, this completes the proof the corollary.

For the second inequality, we take $\Xi_\varepsilon = \sqrt{|{f}|^2+\varepsilon^2}$ and apply the first inequality. Since $|\Lb(\Xi_\varepsilon)| =\big|\frac{(D_\Lb {f}, {f})}{\sqrt{|{f}|^2+\varepsilon^2}}\big|$, we have $|\Lb(\Xi_\varepsilon)|\leq |D_\Lb {f}|$  uniformly in $\varepsilon$. Similarly, $|\Omega(\Xi_\varepsilon)| \leq |D_\Omega {f}|$. The first inequality then gives
\begin{equation}
 r_2 \|\sqrt{|{f}|^2+\varepsilon^2}\|_{L^4(\S_{r_1}^{r_2})}^2\lesssim \int_{\Hb_{r_2}^{r_1}} \big(|{f}|^2+\varepsilon^2\big)+\int_{\Hb_{r_2}^{r_1}} |D_\Lb {f}|^2+\int_{\Hb_{r_2}^{r_1}} |D_\Omega {f}|^2.
\end{equation}
By passing to the limit $\varepsilon \rightarrow 0$, we obtain the second inequality.
\end{proof}

The following lemma is also useful to deal with lower order terms.
\begin{lemma}\label{lemma A7} For a scalar field ${f}$ on an outgoing null hypersurface $\H_{r_1}$, for all $r_2 \geq r_1$ or on an incoming null hypersurface $\Hb_{r_2}^{r_1}$, we have
\begin{equation}\label{Sobolev trace estimates on outgoing null hypersurfaces}
\begin{split}
  \|{f}\|_{L^2(\S_{r_1}^{r_2})}^2&\lesssim \|{f}\|_{L^2(\S_{r_1}^{r_1})}^2+\frac{1}{r_1}\int_{\H_{r_1}} |D_L (r{f})|^2,\\
  \|{f}\|_{L^2(\S_{r_1}^{r_2})}^2&\lesssim \|{f}\|_{L^2(\S_{r_2}^{r_2})}^2+\frac{1}{r_1}\int_{\Hb^{r_1}_{r_2}} |D_\Lb (r{f})|^2.
\end{split}
\end{equation}
\end{lemma}
\begin{proof}
Indeed, we have
\begin{align*}
 \|{f}\|_{L^2(\S_{r_1}^{r_2})}^2- \|{f}\|_{L^2(\S_{r_1}^{r_1})}^2&= \int_{\frac{r_2}{2}}^{\frac{r_1}{2}}\int_{\mathbf{S}^2}L(|rf|^2)d\vartheta dv  \leq 2 \int_{\frac{r_2}{2}}^{\frac{r_1}{2}}\int_{\mathbf{S}^2} |D_L(rf)||rf|d\vartheta dv\\
 &\leq  2 \int_{\frac{r_2}{2}}^{\frac{r_1}{2}}\int_{\mathbf{S}^2} |D_L(rf)|^2 r^2 d\vartheta dv+ 2 \int_{\frac{r_2}{2}}^{\frac{r_1}{2}}\frac{1}{r^2}\int_{\mathbf{S}^2} |rf|^2d\vartheta dv.
\end{align*}
The Gronwall's inequality then completes the proof.
\end{proof}

\subsection{Geometric Calculations}
We frequently compare Lie derivative $\mathcal{L}_{\Lb}$ and covariant derivative $\slashed{\nabla}$. Indeed, we have
\begin{equation}\label{formula compare Lie and nablaslash}
\slashed{\nabla}_\Lb X_A -\mathcal{L}_{\Lb} X_A =\frac{2}{r}X_A, \ \ \slashed{\nabla}_L X_A -\mathcal{L}_{L} X_A =-\frac{2}{r}X_A
\end{equation}
and for vector fields from $\mathcal{Z}$, we have
\begin{align*}
\mathcal{L}_{T}L &= 0, \ \mathcal{L}_{T} \Lb = 0, \ \mathcal{L}_{T}e_A = 0,\ \ \mathcal{L}_{\Omega_{ij}} L = 0, \ \mathcal{L}_{\Omega_{ij}} \Lb=0, \ \mathcal{L}_{\Omega_{ij}} e_A\perp e_A \ \mathcal{L}_{\Omega_{ij}} e_A\perp L , \ \mathcal{L}_{\Omega_{ij}} e_A\perp \Lb, \\
\mathcal{L}_{K} L &= -2vL, \ \mathcal{L}_{K} \Lb = -2u\Lb, \ \mathcal{L}_{K} e_A = -te_A, \ \mathcal{L}_{S} L = -L, \ \mathcal{L}_{S} \Lb = -\Lb, \ \mathcal{L}_{K} e_A = -\frac{t}{r}e_A.
\end{align*}
For a tensor field $\Xi$, we frequently take Lie derivatives along $Z$ or decompose it in null frames. The next lemma record the commutators of these two operations. The proof is a straightforward computation.

For a 2-form $G$, it is straightforward to check that
\begin{align*}
\mathcal{L}_Z \alpha(G)_A =\alpha(\mathcal{L}_Z G)_A + G(\mathcal{L}_Z L,e_A), & \ \ \mathcal{L}_Z \alphab(G)_A =\alphab(\mathcal{L}_Z G)_A + G(\mathcal{L}_Z \Lb,e_A),\\
\mathcal{L}_Z \rho(G) =\rho(\mathcal{L}_Z G) + \frac{1}{2}G(\mathcal{L}_Z \Lb,L)+\frac{1}{2}G(\Lb,\mathcal{L}_Z L), & \ \ \mathcal{L}_Z \sigma(G) =\sigma(\mathcal{L}_Z G) + G(\mathcal{L}_Z e_1,e_2)+G(e_1,\mathcal{L}_Z e_2).
\end{align*}
Based on these formulas, we have
\begin{lemma}\label{lemma commuting Z with null decomposition} For $Z \in \mathcal{Z}$, if $Z\notin \{S,K\}$, we have
\begin{equation*}
\mathcal{L}_Z \alpha(G)_A =\alpha(\mathcal{L}_Z G)_A, \ \ \mathcal{L}_Z \alphab(G)_A =\alphab(\mathcal{L}_Z G)_A, \ \
\mathcal{L}_Z \rho(G) =\rho(\mathcal{L}_Z G), \ \ \mathcal{L}_Z \sigma(G) =\sigma(\mathcal{L}_Z G).
\end{equation*}
Otherwise, we have
\begin{equation*}
\begin{split}
\mathcal{L}_S \alpha(G)_A &=\alpha(\mathcal{L}_S G)_A -\alpha(G)_A, \ \ \mathcal{L}_S \alphab(G)_A =\alphab(\mathcal{L}_S G)_A-\alphab(G)_A, \\
\mathcal{L}_S \rho(G) &=\rho(\mathcal{L}_S G)-2\rho(G), \ \ \mathcal{L}_S \sigma(G) =\sigma(\mathcal{L}_S G)-2\frac{t}{r}\sigma(G).
\end{split}
\end{equation*}
and
\begin{equation*}
\begin{split}
\mathcal{L}_K \alpha(G)_A &=\alpha(\mathcal{L}_K G)_A -2v\alpha(G)_A, \ \ \mathcal{L}_K \alphab(G)_A =\alphab(\mathcal{L}_K G)_A-2u\alphab(G)_A, \\
\mathcal{L}_K \rho(G) &=\rho(\mathcal{L}_K G)-2t\rho(G), \ \ \mathcal{L}_K \sigma(G) =\sigma(\mathcal{L}_K G)-2t\sigma(G).
\end{split}
\end{equation*}
\end{lemma}

Finally, we collect some calculation on integrated quantities on hypersurfaces.
For $\gamma\neq 3$, we define
\begin{equation}\label{D14}
\mathbf{E}_\gamma^{\slash}=\int_{\H_{r_1}^{r_2}}\frac{1}{r^\gamma}|{f}|^2,\ \ \mathbf{E}_\gamma^{\backslash} \int_{\Hb_{r_1}^{r_2}}\frac{1}{r^\gamma}|{f}|^2,  \ \ \mathbf{E}_\gamma^{-} =\int_{\B_{r_1}^{r_2}}\frac{1}{r^\gamma}|{f}|^2.
\end{equation}
We have
\begin{equation}\label{D15}
\mathbf{E}_\gamma^{\slash}=\underbrace{\frac{1}{3-\gamma}\big(\frac{r_1+r_2}{2}\big)^{1-\gamma}\int_{\S_{r_1}^{r_2}}|{f}|^2-\frac{1}{3-\gamma}r_1^{1-\gamma}\int_{\S_{r_1}^{r_1}}|{f}|^2
}_{\mathbf{E}_{\gamma,0}^{\slash}}-\frac{2}{3-\gamma}\int_{\H_{r_1}^{r_2}}r^{1-\gamma}\Re(\overline{D_L{f}}\cdot {f}).
\end{equation}
Similarly, we have
\begin{equation}
\mathbf{E}_\gamma^{\backslash}=\underbrace{\frac{1}{3-\gamma}r_2^{1-\gamma}\int_{\S_{r_2}^{r_2}}|{f}|^2
-\frac{1}{3-\gamma}\big(\frac{r_1+r_2}{2}\big)^{1-\gamma}\int_{\S_{r_1}^{r_2}}|{f}|^2}_{\mathbf{E}_{\gamma,0}^{\backslash}}+\frac{2}{3-\gamma}\int_{\Hb_{r_1}^{r_2}}r^{1-\gamma}\Re(\overline{D_\Lb{f}}\cdot {f}),
\end{equation}
and
\begin{equation}
\mathbf{E}_\gamma^{-}=\underbrace{\frac{1}{3-\gamma}r_2^{1-\gamma}\int_{\S_{r_2}^{r_2}}|{f}|^2
-\frac{1}{3-\gamma}r_1^{1-\gamma}\int_{\S_{r_1}^{r_1}}|{f}|^2}_{\mathbf{E}_{-,0}^{\backslash}}+\frac{2}{3-\gamma}\int_{\B_{r_1}^{r_2}}r^{1-\gamma}\Re(\overline{D_{\partial_r}{f}}\cdot {f}),
\end{equation}
As an application, we prove the following Hardy type inequality:
\begin{lemma}\label{lemma Hardy on H}
For $\gamma > 3$, we have
\begin{equation}\label{inequality hardy on H}
\int_{\H_{r_1}^{r_2}} r^{-\gamma}|{f}|^2 +\frac{1}{r_2^{\gamma-1}}\int_{\S_{r_1}^{r_2}}|{f}|^2\lesssim_\gamma   r_1^{-\gamma+1}\int_{\S_{r_1}^{r_1}}|{f}|^2+r_1^{-\gamma+2}\int_{\H_{r_1}^{r_2}}\big|D_L{f}\big|^2.
\end{equation}
\end{lemma}
\begin{proof}
In view of \eqref{D15}, by discarding the first term on the righthand side, we have
\begin{align*}
\int_{\H_{r_1}^{r_2}}\frac{1}{r^\gamma}|{f}|^2 +\frac{1}{r_2^{\gamma-1}}\int_{\S_{r_1}^{r_2}}|{f}|^2& \lesssim_\gamma r_1^{-(\gamma-1)}\int_{\S_{r_1}^{r_1}}|{f}|^2
+2\int_{\H_{r_1}^{r_2}}r^{-(\gamma-1)}\big|D_L(rf)\big|\big|{f}\big|\\
&\leq C_\gamma r_1^{-3}\int_{\S_{r_1}^{r_1}}|{f}|^2
+C_\gamma\int_{\H_{r_1}^{r_2}}r^{-\gamma+2}\big|D_L{f}\big|^2+\frac{1}{2}\int_{\H_{r_1}^{r_2}}\frac{1}{r^\gamma}|{f}|^2
\end{align*}
Thus,
\begin{equation*}
\int_{\H_{r_1}^{r_2}} r^{-\gamma}|{f}|^2 +\frac{1}{r_2^{\gamma-1}}\int_{\S_{r_1}^{r_2}}|{f}|^2\lesssim r_1^{-\gamma+1}\int_{\S_{r_1}^{r_1}}|{f}|^2+\int_{\H_{r_1}^{r_2}}r^{-\gamma+2}\big|D_L{f}\big|^2.
\end{equation*}
This completes the proof.
\end{proof}

\bibliography{shiwu}{}

\begin{thebibliography}{10}

\bibitem{Stefanos:avectorfield}
Y.~Angelopoulos, S.~Aretakis, and D.~Gajic.
\newblock {A vector field approach to almost-sharp decay for the wave equation
  on spherically symmetric, stationary spacetimes}.
\newblock 2016.
\newblock ar{X}iv:1612.01565.

\bibitem{Stefanos:Latetime}
Y.~Angelopoulos, S.~Aretakis, and D.~Gajic.
\newblock Late-time asymptotics for the wave equation on spherically symmetric,
  stationary spacetimes.
\newblock {\em Adv. Math.}, 323:529--621, 2018.

\bibitem{Lydia:MKG:small}
L.~Bieri, S.~Miao, and S.~Shahshahani.
\newblock Asymptotic properties of solutions of the {M}axwell {K}lein {G}ordon
  equation with small data.
\newblock {\em Comm. Anal. Geom.}, 25(1):25--96, 2017.

\bibitem{fieldschrist}
Y.~Choquet-Bruhat and D.~Christodoulou.
\newblock Existence of global solutions of the {Y}ang-{M}ills, {H}iggs and
  spinor field equations in {$3+1$} dimensions.
\newblock {\em Ann. Sci. \'Ecole Norm. Sup. (4)}, 14(4):481--506 (1982), 1981.

\bibitem{ChristodoulouYangM}
Y.~Choquet-Bruhat and D.~Christodoulou.
\newblock Existence of global solutions of the {Y}ang-{M}ills, {H}iggs and
  spinor field equations in {$3+1$} dimensions.
\newblock {\em Ann. Sci. \'Ecole Norm. Sup. (4)}, 14(4):481--506 (1982), 1981.

\bibitem{Christodoulou:Book:GR1}
D.~Christodoulou.
\newblock {\em Mathematical problems of general relativity. {I}}.
\newblock Zurich Lectures in Advanced Mathematics. European Mathematical
  Society (EMS), Z\"urich, 2008.

\bibitem{asymLkl}
D.~Christodoulou and S.~Klainerman.
\newblock Asymptotic properties of linear field equations in {M}inkowski space.
\newblock {\em Comm. Pure Appl. Math.}, 43(2):137--199, 1990.

\bibitem{newapp}
M.~Dafermos and I.~Rodnianski.
\newblock A new physical-space approach to decay for the wave equation with
  applications to black hole spacetimes.
\newblock In {\em X{VI}th {I}nternational {C}ongress on {M}athematical
  {P}hysics}, pages 421--432. World Sci. Publ., Hackensack, NJ, 2010.

\bibitem{Moncrief1}
D.~Eardley and V.~Moncrief.
\newblock The global existence of {Y}ang-{M}ills-{H}iggs fields in
  {$4$}-dimensional {M}inkowski space. {I}. {L}ocal existence and smoothness
  properties.
\newblock {\em Comm. Math. Phys.}, 83(2):171--191, 1982.

\bibitem{Moncrief2}
D.~Eardley and V.~Moncrief.
\newblock The global existence of {Y}ang-{M}ills-{H}iggs fields in
  {$4$}-dimensional {M}inkowski space. {II}. {C}ompletion of proof.
\newblock {\em Comm. Math. Phys.}, 83(2):193--212, 1982.

\bibitem{Pedro:YM:sph}
V.~Georgiev and P.~Schirmer.
\newblock The asymptotic behavior of {Y}ang-{M}ills fields in the large.
\newblock {\em Comm. Math. Phys.}, 148(3):425--444, 1992.

\bibitem{klinvar}
S.~Klainerman.
\newblock Uniform decay estimates and the lorentz invariance of the classical
  wave equation.
\newblock {\em Comm. Pure Appl. Math.}, 38(3):321--332, 1985.

\bibitem{MKGkl}
S.~Klainerman and M.~Machedon.
\newblock On the {M}axwell-{K}lein-{G}ordon equation with finite energy.
\newblock {\em Duke Math. J.}, 74(1):19--44, 1994.

\bibitem{YMkl}
S.~Klainerman and M.~Machedon.
\newblock Finite energy solutions of the {Y}ang-{M}ills equations in
  {$\mathbb{R}^{3+1}$}.
\newblock {\em Ann. of Math. (2)}, 142(1):39--119, 1995.

\bibitem{kl:EE}
S.~Klainerman and F.~Nicol\`o.
\newblock {\em The evolution problem in general relativity}, volume~25 of {\em
  Progress in Mathematical Physics}.
\newblock Birkh\"auser Boston, Inc., Boston, MA, 2003.

\bibitem{Kl:peeling:EE}
S.~Klainerman and F.~Nicol\`o.
\newblock Peeling properties of asymptotically flat solutions to the {E}instein
  vacuum equations.
\newblock {\em Classical Quantum Gravity}, 20(14):3215--3257, 2003.

\bibitem{yang:mMKG}
S.~Klainerman, Q.~Wang, and S.~Yang.
\newblock {Global solution for massive Maxwell-Klein-Gordon equations}.
\newblock 2018.
\newblock ar{X}iv:1801.10380.

\bibitem{LindbladMKG}
H.~Lindblad and J.~Sterbenz.
\newblock Global stability for charged-scalar fields on {M}inkowski space.
\newblock {\em IMRP Int. Math. Res. Pap.}, pages Art. ID 52976, 109, 2006.

\bibitem{Jonathan:stabilityofEE}
J.~Luk and S.~Oh.
\newblock {Global nonlinear stability of large dispersive solutions to the
  Einstein equations}.
\newblock preprint.

\bibitem{OhMKG4}
S.~Oh and D.~Tataru.
\newblock Global well-posedness and scattering of the {$(4+1)$}-dimensional
  {M}axwell-{K}lein-{G}ordon equation.
\newblock {\em Invent. Math.}, 205(3):781--877, 2016.

\bibitem{Deivy:decayMKG:trivialout}
D.~Petrescu.
\newblock Time decay of solutions of coupled {M}axwell-{K}lein-{G}ordon
  equations.
\newblock {\em Comm. Math. Phys.}, 179(1):11--23, 1996.

\bibitem{Maria:timedecay:mMKG}
M.~Psarelli.
\newblock Time decay of {M}axwell-{K}lein-{G}ordon equations in
  {$4$}-dimensional {M}inkowski space.
\newblock {\em Comm. Partial Differential Equations}, 24(1-2):273--282, 1999.

\bibitem{Shu}
W.~Shu.
\newblock Asymptotic properties of the solutions of linear and nonlinear spin
  field equations in {M}inkowski space.
\newblock {\em Comm. Math. Phys.}, 140(3):449--480, 1991.

\bibitem{shu2}
W.~Shu.
\newblock Global existence of {M}axwell-{H}iggs fields.
\newblock In {\em Nonlinear hyperbolic equations and field theory ({L}ake
  {C}omo, 1990)}, volume 253 of {\em Pitman Res. Notes Math. Ser.}, pages
  214--227. Longman Sci. Tech., Harlow, 1992.

\bibitem{yangMKG}
S.~Yang.
\newblock Decay of solutions of {M}axwell-{K}lein-{G}ordon equations with
  arbitrary {M}axwell field.
\newblock {\em Anal. PDE}, 9(8):1829--1902, 2016.

\bibitem{yangILEMKG}
S.~Yang.
\newblock On the global behavior of solutions of the
  {M}axwell--{K}lein--{G}ordon equations.
\newblock {\em Adv. Math.}, 326:490--520, 2018.

\end{thebibliography}
\bibliographystyle{plain}

\end{document}